\documentclass{amsart}
\usepackage[backend=biber, style=numeric]{biblatex}

\newtheorem{theorem}{Theorem}[section]
\newtheorem{conjecture}[theorem]{Conjecture}
\newtheorem{fact}[theorem]{Fact}
\newtheorem{lemma}[theorem]{Lemma}
\newtheorem{corollary}[theorem]{Corollary}
\newtheorem{claim}[theorem]{Claim}
\newtheorem{proposition}[theorem]{Proposition}

\newtheorem*{mainthm}{Main Theorem}
\newtheorem*{thm}{Theorem}

\theoremstyle{definition}
\newtheorem{definition}[theorem]{Definition}

\newtheorem{convention}[theorem]{Convention}
\newtheorem{notation}[theorem]{Notation}

\newtheorem{assumption}[theorem]{Assumption}

\theoremstyle{remark}
\newtheorem{remark}[theorem]{Remark}

\newcommand{\rk}{\operatorname{rk}}

\bibliography{references}

\begin{document}
	\title{Restricted Trichotomy in Characteristic Zero}
	\author{Benjamin Castle}
	\address{Department of mathematics\\
		Ben Gurion University of the Negev\\
		Be'er Sehva\\
		Israel}
	\email{bcastle@berkeley.edu}
	\subjclass{Primary 0C345; Secondary 14A99}
	\thanks{Partially supported by the Field Institute for Research in Mathematical Sciences, and by NSF Grant DMS \#1800692}
	\maketitle
	
	\begin{abstract}
		We prove the characteristic zero case of Zilber's Restricted Trichotomy Conjecture. That is, we show that if $\mathcal M$ is any non-locally modular strongly minimal structure interpreted in an algebraically closed field $K$ of characteristic zero, then $\mathcal M$ itself interprets $K$; in particular, any non-1-based structure interpreted in $K$ is mutually interpretable with $K$. Notably, we treat both the `one-dimensional' and `higher-dimensional' cases of the conjecture, introducing new tools to resolve the higher-dimensional case and then using the same tools to recover the previously known one-dimensional case.
	\end{abstract}
	
	\section{Introduction} 
	
	In this paper we prove Zilber's Restricted Trichotomy Conjecture for algebraically closed fields of characteristic zero in all dimensions:
	
	\begin{mainthm}[Theorem \ref{general main theorem}] Let $K$ be an algebraically closed field of characteristic zero, and let $\mathcal M$ be a structure interpretable in $K$. If $\mathcal M$ is not 1-based, then $K$ is interpretable in $\mathcal M$.
	\end{mainthm}

	Here `1-based' is a model theoretic notion capturing those structures of `at most linear complexity'; in particular, any 1-based structure interpretable in $K$ can be equipped with a dimension theory satisfying certain properties found in vector spaces. Meanwhile the structure $\mathcal N=(N,...)$ is said to \textit{interpret} the structure $\mathcal N'=(N',...)$ if one can describe $\mathcal N'$ up to isomorphism using only the language of $\mathcal N$, allowing for division by equivalence relations (formally this means there are an $\mathcal N$-definable set $X$, an $\mathcal N$-definable equivalence relation $E$ on $X$, and a bijection $X/E\rightarrow N'$, such that every $\mathcal N'$-definable subset of any power $(N')^k$ is the image of an $\mathcal N$-definable subset of $X^k$ under $X\rightarrow X/E\rightarrow N'$). 
	
	Intuitively, if $\mathcal N$ interprets $\mathcal N'$, then one thinks of $\mathcal N'$ as `weaker' than $\mathcal N$, obtained by retaining only some of the full information from $\mathcal N$. The above theorem thus concerns which restricted information from algebraic geometry over $K$ is enough to reconstruct the field $K$. It is well known that a 1-based structure can never interpret an algebraically closed field (in our setting, this follows from Fact \ref{1-based fact}); conversely, our theorem says this is the only obstruction. In other words, we show that if (1) $\mathcal M$ can be defined using only constructible sets over $K$, and (2) $\mathcal M$ is `richer' than a pure vector space in a precise sense, then a copy of $K$ can be found inside $\mathcal M$. 
	
	This theorem can be seen as a general reconstruction result for constructible sets over $K$, equipped with any data admitting a first-order description. As such, it is expected that one can find applications in anabelian geometry, and in more general reconstruction problems along the lines of \cite{KLOS}. In particular, we note that in many situations there are simple tests for a structure being non-1-based (see Remark \ref{1-based remark}).
	
	\subsection{Strong Minimality and Zilber's Conjecture}
	
	The problem treated in this paper was originally proposed by Boris Zilber, as part of his program of classifying \textit{strongly minimal} structures. Note that it follows from standard results in model theory that one can assume the structure $\mathcal M$ in our theorem is strongly minimal (see Fact \ref{1-based fact}); indeed, the proof of the theorem will take place almost exclusively in the strongly minimal setting. Formally, we recall that a definable set $X$ in a structure $\mathcal M$ is said to be \textit{strongly minimal} if  $X$ is infinite, while every definable subset of $X$ is uniformly finite or cofinite: that is, for every definable family $\{Y_t:t\in T\}$ of subsets $Y_t\subset X$, there is an integer $d$ such that for every $t\in T$, one of the sets $Y_t$ and $X-Y_t$ has size at most $d$. Meanwhile $\mathcal M=(M,...)$ is a \textit{strongly minimal structure} if its universe $M$ is a strongly minimal set.
	
	The class of strongly minimal structures serves as a common generalization of pure sets, vector spaces, and algebraically closed fields, which allows for the purely combinatorial development of a well-behaved dimension theory for definable sets (see Section 2); for example, over an algebraically closed field this dimension theory agrees with the usual dimension theory of varieties. Originated in the 1970s and 1980s, Zilber's program was an attempt to classify all strongly minimal structures in terms of the these three main examples. Among many other formulations, for example, Zilber made the following influential conjecture (see \cite{Zil84}, Theorem 3.1 and Conjecture B):

	\begin{conjecture}[Zilber's Trichotomy Conjecture]\label{zc} If $\mathcal M=(M,...)$ is strongly minimal, then exactly one of the following holds:
		\begin{enumerate}
			\item The geometry of $\mathcal M$ is the trivial geometry.
			\item The geometry of $\mathcal M$ is an affine or projective geometry over a division ring.
			\item $\mathcal M$ interprets an algebraically closed field.
		\end{enumerate} 
	\end{conjecture}

	Here the \textit{geometry} of $\mathcal M$ is a certain associated combinatorial object which captures the dimension theory of $\mathcal M$. One should note that by results of Zilber and Hrushovski (see \cite{Hru85}), the union of cases (1) and (2) above is equivalent to an abstract model theoretic notion called \textit{local modularity} (see Definition \ref{local modularity}), which places a bound on the dimensions of certain families of definable sets, and is equivalent to 1-basedness for strongly minimal structures. Thus Conjecture \ref{zc} equivalently predicts that every non-locally modular (or non-1-based) strongly minimal structure interprets an algebraically closed field. One should also note that Conjecture \ref{zc} can be seen as a model theoretic analog of Artin's construction of a field from an abstract plane geometry (see the second chapter of \cite{Artin}). Indeed, one can show that any non-locally modular strongly minimal structure interprets a \textit{pseudoplane} -- an incidence system satisfying a coarser set of axioms (see for example \cite{Bue85}). Combined with the fact that every infinite $\omega$-stable field is algebraically closed (\cite{Mac71}), Conjecture \ref{zc} then further reduces to the assertion that every strongly minimal pseudoplane interprets an infinite field. 
	
	\subsection{The Restricted Trichotomy Conjecture}
		
	As it turns out, Conjecture \ref{zc} fails in full generality, as a large class of counterexamples was constructed by Hrushovski in \cite{Hru93}. Conjecture \ref{zc} has remained a guiding principle in model theory, however, and is known to hold in several restricted settings (e.g. \cite{HZ}, \cite{HruSok},\cite{PZ}). Indeed, known cases of the conjecture have been at the forefront of many exciting applications of model theory to algebraic geometry (e.g. \cite{Hru96}, \cite{Hru01}, \cite{Sca06}). Notably, the common theme behind cases where Conjecture \ref{zc} holds is the presence of some underlying geometric structure controlling the definable sets in $\mathcal M$.
	
	Among the various special cases of Conjecture \ref{zc}, the following (dating back to the 1980s, though most explicitly stated in \cite{Zil14}) has remained open. 
	
	\begin{conjecture}[Zilber's Restricted Trichotomy Conjecture]\label{rst} Let $\mathcal M$ be a strongly minimal structure which is interpretable in an algebraically closed field $K$. If $\mathcal M$ is not locally modular, then $\mathcal K$ is interpretable in $\mathcal M$.
	\end{conjecture}

	The main result of this paper, equivalently stated, is a solution of Conjecture \ref{rst} for the field of complex numbers: 
	
	\begin{thm}[Theorem \ref{main theorem}] Conjecture \ref{rst} holds when $K=\mathbb C$.
	\end{thm}

	The general characteristic zero case of Conjecture \ref{rst} then follows from Theorem \ref{main theorem} and basic model theory. Meanwhile the statement for arbitrary interpreted structures in characteristic zero (Theorem \ref{general main theorem}) then follows since any non-1-based structure of finite Morley rank interprets a non-locally modular strongly minimal structure (see Fact \ref{1-based fact}).
	
	In the remainder of the introduction we will give some historical remarks on Conjecture \ref{rst}, and then a brief outline of the method of proof. To begin, let us examine Conjecture \ref{rst} in more detail.
	
	First, by well-known facts in the model theory of algebraically closed fields (i.e. elimination of imaginaries and quantifiers (\cite{PoiEli} and \cite{Tar}), and the classification of interpretable fields (\cite{Poi}, Theorem 4.15)), Conjecture \ref{rst} has the following concrete reformulation:
	
	\begin{conjecture}[Restricted Trichotomy Conjecture, Concrete Formulation]\label{concrete rst} Let $K$ be an algebraically closed field, and let $\mathcal M=(M,...)$ be a non-locally modular strongly minimal structure. Assume that the universe $M$ of $\mathcal M$, as well as all $\mathcal M$-definable subsets of cartesian powers of $M$, are constructible sets over $K$ (i.e. finite Boolean combinations of affine algebraic sets). Then $\mathcal M$ interprets an infinite field.
	\end{conjecture}
		
	That is, Conjecture \ref{rst} predicts that Conjecture \ref{zc} should hold whenever the structure $\mathcal M$ can be defined using only algebraic geometry over a fixed algebraically closed field. We stress that the definable sets in $\mathcal M$ consist of \textit{some}, but not necessarily \textit{all}, constructible subsets of powers of $M$. In particular, while in other special cases of Conjecture \ref{zc} one has access to \textit{all} definable sets coming from some ambient geometry (so one can work with closed sets, smooth sets, etc.), in Conjecture \ref{rst} one does not a priori have access to \textit{any} relatively closed (resp. smooth, etc.) sets. Part of the challenge, then, is sorting through an abstractly presented collection of `random' constructible sets (which may have large frontiers, many singularities and small components, etc.) in order to recover the `correct' smooth geometric structure lying in the background. 
	
	We should point out that the structure $\mathcal M$ in Conjecture \ref{rst} can only interpret $K$ if the universe $M$ has dimension one as a constructible set. For this reason Conjecture \ref{rst} has historically been divided into the so called `one-dimensional' and `higher-dimensional' cases (i.e. when $\dim M=1$ and when $\dim M>1$) -- where the goal of the higher dimensional case is really to show that all such $\mathcal M$ are locally modular. Essentially all past results have been in the one-dimensional case, culminating in a full solution in dimension one by Hasson and Sustretov in the preprint \cite{HS}; meanwhile the higher-dimensional case has remained a wide open problem. The main accomplishment of the current paper, then, is really a solution of the higher-dimensional case in characteristic zero. Coincidentally, the method used applies in both cases (one-dimensional and higher-dimensional), in particular yielding a new proof of the one-dimensional case; thus we have chosen to include both parts, so that there is a single, self-contained proof in the literature.
	
	As stated above, we expect our theorem to be applicable to reconstruction problems in algebraic geometry. However, as a point of caution let us clarify what we mean. Indeed, note that even if the structure $\mathcal M=(M,...)$ from Conjecture \ref{rst} recovers a copy of $K$, it might not recover the entire collection of constructible subsets of powers of $M$; this is essentially because $\mathcal M$ might only see the full structure on an interpreted `copy' of $M$, coming from the given copy of $K$. In fact there is a precise sense in which the `original' $M$ is a sort of `finite cover' of the interpreted copy (see \cite{HZ}, Theorem B); however there are simple examples where this finite cover cannot be eliminated. Thus in most applications one will need to both (1) use our theorem to recover a field, and (2) make an additional argument to show that the `finite cover' in question can be made trivial in the appropriate context. There is already one example due to Zilber where this outline was successful, in which a previously known case of Conjecture \ref{rst} was used to solve a reconstruction problem pertaining to Jacobian varieties (\cite{Zil14}). We note, then, that the main idea of Zilber's approach toward item (2) above is well known in model theory and seems fairly general; so it is likely that much more general statements will follow from the full solution of Conjecture \ref{rst}. 
	
	\subsection{Past Results} 
	
	Let us now discuss the main historical developments in the one-dimensional case of Conjecture \ref{rst}. The works of Martin (\cite{Mart}), Rabinovich-Zilber (\cite{RZ}), and Marker-Pillay (\cite{MP}) each treat the case that $M=K$ and addition is definable in $\mathcal M$; in particular Marker and Pillay completely solve Conjecture \ref{rst} for structures on $M=\mathbb C$ containing addition. Here the presence of addition is very helpful, as strong minimality and a group operation combine to simplify the possible geometric abnormalities of definable sets in $\mathcal M$.
	
	Shortly after Marker and Pillay's result, the monograph \cite{Rab} of Rabinovich claimed to solve the general case of Conjecture \ref{rst} when $M=K$. This setting represents a significant jump in complexity, as one loses many geometric properties of $\mathcal M$ that follow from having addition. Accordingly, Rabinovich's argument is very technical, and has not been universally accepted by the community. Thus it seemed natural to hope for a more general and accessible treatment of the result; but no such treatment emerged over the ensuing twenty years, and Rabinovich's paper remained the furthest point of progress toward the full conjecture.
	
	Approximately twenty years after its publication, Zilber used Rabinovich's result to solve the aforementioned problem on Jacobians (\cite{Zil14}). In his paper he attempted to renew interest in the full statement of Conjecture \ref{rst}, challenging the community again to find a new proof of a more general case than Rabinovich's. This challenge was answered a few years later Hasson and Sustretov, whose preprint mentioned above (\cite{HS}) treated the entire one-dimensional case. Their paper is considerably less technical than that of Rabinovich, but they still need to navigate through various issues related to singularities and isolated points of definable sets. In particular, the methods used for dealing with isolated points only seem to work under the assumption that $\dim M=1$. 
	
	\subsection{The Historical Strategy: Defining Tangency}
	
	Most of the previous papers on the one-dimensional case of Conjecture \ref{rst} follow the same general strategy, which goes back to ideas of Zilber and Rabinovich. Let us now review this strategy and compare it to the current paper. 
	
	The main idea of the existing strategy is to show that $\mathcal M$ can in some weak sense `define' when two \textit{plane curves} are tangent at a diagonal point $(x,x)\in M^2$ (here a \textit{plane curve} refers to a one-dimensional $\mathcal M$-definable subset of $M^2$, not an algebraic curve in $K^2$). One then views such plane curves as generalizations of functions $M\rightarrow M$ (which can take finitely many values at each point), and in doing so recovers a natural notion of \textit{slope} for plane curves at $(x,x)$. The elements of $K$ are then identified, roughly, with a suitable one-dimensional family of plane curves through $(x,x)$ -- say $\{C_t:t\in T\}$ -- by identifying each $t\in T$ with the slope of $C_t$ at $(x,x)$. Of course such an identification requires the availability (in $\mathcal M$) of a `rich enough' family which witnesses all possible slopes; the existence of such a family is one of the main uses of the non-local modularity of $\mathcal M$.
	
	To truly identify a field, one then needs to recover the addition and multiplication operations on the underlying set. To do this one introduces definable operations on plane curves, and shows that they induce the desired operations on slopes. For multiplication this is done using the chain rule applied to an abstract `composition' operation (generalizing the composition of functions to the multi-valued case). Of course one needs to deal with the possibility that the $\{C_t\}$ might not be closed under composition; this is why the identification of tangency is so important, as it allows one to definably identify a composition $C_s\circ C_t$ with some $C_u$ ($u\in T$) attaining the same slope at $(x,x)$.
	
	On the other hand, there is no obvious operation on general plane curves which induces addition of slopes. For this reason, the general approach (and that used by Hasson and Sustretov) has been to build a field in two stages. In the first stage, a one-dimensional group (say $G$) is constructed using the composition operation described above; once this is done, the entire process is started over using $G$ in place of $M$, working at the point $(0,0)$. One can now generate multiplication in the same way as before, while recovering addition via a `point-wise sum' operation. 
	
	We should point out here that the `two-stage' method above is not the approach taken by Rabinovich: indeed, though she still uses a detection of tangency, her construction of a field happens in only one stage. We should also point out that since one does not have a perfect definition of tangency (typically only a weak approximation can be achieved), the above method is quite imprecise. In reality, one does not directly build a group or a field, but instead builds an abstract configuration of points carrying the same independence structure as a certain tuple from a group or a field. One then quotes results of Hrushovski (the \textit{group configuration} and \textit{field configuration} theorems, see \cite{Hru85} or \cite{Bou}) to deduce the interpretability of a genuine group or field from the presence of such a collection of points.
	
	\subsection{The Main Complications} There are many difficulties in the rough strategy outlined above; we highlight three. First, there is no obvious way to make an analogous field construction work if $\dim M>1$, as the one-dimensional case relies heavily on the identification of slopes through a point with elements of a fixed one-dimensional algebraic group. Second, because one is fixing a point $(x,x)$ to lie on all plane curves considered, and then performing operations on these curves, it is possible that $(x,x)$ is a singularity on many of the curves appearing in the argument; thus the notion of slope is not truly well-defined, and one needs to navigate through more difficult geometric reasoning than desired.
	
	Most important, though, is the difficulty of finding any reasonable way to detect the tangency of two curves $C_1,C_2$ at a point. Historically this is done with formulas of the form $|C_1\cap C_2|<N$, where $N$ is the number of intersection points of a `generic' pair of curves of the same complexity. However, because the curves used are not guaranteed to be pure-dimensional, there is nothing preventing the decreased intersection cardinality at a `multiple intersection' from being canceled out by an `extra' intersection at an isolated point. Thus in a sense one needs to be able to identify and remove isolated points of curves before this `counting intersections' argument can work -- a task which has proven quite difficult over the years (indeed this is the main reason the problem becomes easier if one already has addition, because in this case one can always find large families of pure-dimensional plane curves). Notably Hasson and Sustretov found a novel way to `avoid' isolated points in certain key pieces of the argument when $\dim M=1$. However their approach does not seem to work in higher dimensions, as they rely heavily on the finiteness of the set of isolated points of a one-dimensional set; while if $\dim M>1$ the analogous complications are more generally caused by `components of intermediate dimension.'
	
	\subsection{The New Strategy: From Tangency to Closure}
	
	The main observation of this paper is that there is an alternative way to detect tangency of curves, which does not depend on pure-dimensionality, and thus avoids certain technically demanding parts of e.g. Hasson and Sustretov's argument. Moreover, this alternate method can be developed in any dimension (in general replacing `tangency' with `non-transversality'), and in higher dimensions happens to reduce Conjecture \ref{rst} to a relatively simple problem in algebraic geometry.
	
	The idea has already been used to an extent in similar papers (e.g. \cite{HZ}, \cite{HK}, \cite{EHP}). Namely, we view the `double intersection' caused by a non-transversality as a ramification point of a certain projection, and then use the topological characterization of ramification via the non-openness of the diagonal (e.g. \cite{EGA4}, Corollaire 17.4.2). More precisely, assume our curves come from larger families, say $C_{\hat t}\in\{C_t:t\in T\}$ and $D_{\hat u}\in\{D_u:u\in U\}$, with intersection point $\hat x\in M^2$. Then, assuming certain smoothness conditions, we will see that the intersection at $\hat x$ is non-transversal if and only if $(\hat x,\hat x,\hat t,\hat u)$ belongs to the Zariski frontier of the set of $(x,x',t,u)$ with $x\neq x'$ and $x,x'\in C_t\cap D_u$ (see Proposition \ref{smooth fiber product}(3) and Lemma \ref{detecting tangency equivalence}). By reinterpreting the problem in this way, we therefore reduce the detection of non-transversality down to the detection of closure points of $\mathcal M$-definable sets.
	
	We then draw inspiration from several past papers on strongly minimal structures interpretable in \textit{o-minimal} theories (\cite{PetStaACF},\cite{HK},\cite{EHP}), which developed a common method for detecting closure points of definable sets under certain conditions. Unfortunately we have weaker hypotheses and a loftier goal than previous iterations of this method: namely, all previous instances assumed that $\mathcal M$ was already a group, and only attempted to detect the frontier points of plane curves in $\mathcal M$. However, the result was quite powerful: in each case it was shown that the frontier an $\mathcal M$-definable plane curve was finite (even though the plane curves in question were 2-dimensional in the ambient geometry), and this was used to conclude various other strong geometric properties.
	
	Our strategy, then, is to develop a general method for detecting closure points of \textit{arbitrary} $\mathcal M$-definable sets (without the assumption of a group structure), by adapting the strategy of the aforementioned papers and generalizing it to higher dimensions; we then seek to `recognize' the frontier point $(\hat x,\hat x,\hat t,\hat u)$ described above by placing it into this more general framework. Now the most perfect statement (let us call it the `ideal case') would be if the frontier of \textit{any} $\mathcal M$-definable set is contained in an $\mathcal M$-definable set of smaller dimension -- indeed, it is fairly easy to see that this would suffice for our purposes (see Lemma \ref{ideal scenario}). However, it turns out to be somewhat difficult to find a correct higher dimensional adaptation of the method, especially without a group operation built in -- and directly establishing the ideal case seems quite far out of reach.
	
	Instead, our main innovation is the identification of a certain \textit{weakening} of the ideal case, which rather surprisingly can be proven in full generality (with no need for a group operation) using an induction on dimension. Precisely, if $x\in\overline X$, we show that one of three things happens (Proposition \ref{closure prop}): (1) $x$ belongs to an $\mathcal M$-definable set of smaller dimension than $X$ (the ideal case), (2) some coordinate $x_i\in M$ of $x$ belongs to a constructible subset of $M$ of smaller dimension than $M$ (i.e. $x$ is not \textit{coordinate-wise generic}), or (3) in a certain precise sense (related to depencies between the coordinates of a point), $x$ has a similar appearance to `most' elements of $X$. Curiously, the effect of condition (3) will be that we are often able to `fully' detect frontier points $x=(x_1,...,x_k)\in\operatorname{Fr}(X)$ (i.e. in the ideal sense) if two of the $x_i$'s are equal; in particular this will work for the aforementioned point $(\hat x,\hat x,\hat t,\hat u)$, which rather intriguingly gives us exactly the result we need (see Theorem \ref{detecting generic non-transversalities}).
	
	\subsection{Summary of the Paper}
	
	We now outline the presentation of the argument. Section 2 gives some background on strongly minimal structures and local modularity, and the associated machinery of Morley rank and plane curves. The most important notions, which do not all appear in the literature, are \textit{canonical bases} (Definition \ref{Cb}), \textit{almost faithfulness} (Definition \ref{almost faithful}), \textit{local modularity} (Definition \ref{local modularity}), \textit{standard} and \textit{excellent} families (Definitions \ref{standard family} and \ref{excellent family}) and \textit{semi-indistinguishable} points (Definition \ref{indistinguishable}).
	
	Section 3 gives the setting and conventions that we use throughout. The most important content here is Lemma \ref{dim rk equality} and its consequences, as well as the notion of \textit{coherence} (Definition \ref{coherent}).
	
	Section 4 then gives some geometric preliminaries, isolating many of the facts we use on general constructible sets. The content of this section is essentially summarized in Proposition \ref{smooth fiber product}, whose purpose is to make the geometric reasoning as easy as possible for the rest of the paper.
	
	Section 5 introduces two key notions -- namely that of a \textit{generic non-transversality} (Definition \ref{generic non-transversality}), as well as what it means for $\mathcal M$ to \textit{detect generic non-transversalities} (Definition \ref{detect generic non-transversality}). These notions attempt to formalize the `weak definition of tangency' that is needed for the proof to work. We then give the aforementioned characterization of generic non-transversalities in terms of closure points (Lemma \ref{detecting tangency equivalence}).
	
	Note that we use the term `generic' non-transversality because we only consider intersections between generic curves in infinite families, and require the intersection point to be generic both in each curve and in $M^2$. In particular we never consider intersections at diagonal points, so we are not in the situation of previous authors. In turns out, however, that our definition is sufficient for the historical strategy to work, and is even necessary for our proof in higher dimensions to work. Moreover, working only at generic points allows us to ignore singularities and other unwanted phenomena -- so that our `random' constructible sets can effectively be treated as smooth complex varieties throughout (this is the essential content of Proposition \ref{smooth fiber product}). Thus working with only generic non-transversalities seems to be an advantageous way to approach the problem.
	
	Once we have the notion of generic non-transversalities, we present the main theorem as a corollary of three smaller theorems, whose proofs occupy Sections 6 through 9. Precisely, assuming $\mathcal M$ is strongly minimal, not locally modular, and interpreted in $\mathbb C$ (and after making some harmless adjustments to the setting -- see Assumption \ref{M and K}), we will show:

	\begin{enumerate}
		\item If $\mathcal M$ detects generic non-transversalities, then $\dim M=1$ (Theorem \ref{n=1}).
		\item If $\mathcal M$ detects generic non-transversalities and $\dim M=1$, then $\mathcal M$ interprets an algebraically closed field (Theorem \ref{field interpretation}).
		\item $\mathcal M$ detects generic non-transversalities (Theorem \ref{detecting generic non-transversalities}).
	\end{enumerate}
	
	Note that it is immediate from the combination of (1), (2), and (3) that $\mathcal M$ interprets an algebraically closed field. The reader should view the three theorems above as disjoint pieces of the proof, which serve as an organizational guide for the paper as a whole. Let us now discuss these three results in more detail.
		
	Section 6 contains the proof of (1) -- that is, assuming $\mathcal M$ detects generic non-transversalities we show that the universe $M$ can be taken to be an algebraic curve (Theorem \ref{n=1}). This is really a short exercise in compactness and the purity of the ramification locus. The idea is that generic non-transversalities, when they exist, always occur in codimension 1 in a precise sense. Using this (and the detection of generic non-transversalities), we obtain two $\mathcal M$-definable sets whose dimensions differ by 1; it is then easy to see this implies $\dim M=1$.
		
	Section 7 contains the construction of a field when $\dim M=1$ and $\mathcal M$ detects generic non-transversalities (Theorem \ref{field interpretation}). The reader could view this as a shortened version of Hasson and Sustretov's paper, where many of the technical issues they encounter have been eliminated. However there are two interesting differences to point out. First, our assumption on generic non-transversalities restricts us to working with curves through a \textit{generic} point $(x,y)\in M^2$ (rather than a diagonal point). This complicates the recovery of the field operations, because in general the sum or composition of two curves through $(x,y)$ need not contain $(x,y)$. To fix this, we build the operations on curves through $(x,y)$ via sequences of several iterated sums and compositions (Lemma \ref{detecting products at a point}), applied along a generic sequence of coordinates which `wraps around' back to $(x,y)$ in the end. Note that while this process certainly complicates matters on some level, the proofs still seem easier overall due to the lack of singularities and isolated points at each step. 
	
	Second, our construction of a field happens in only one stage, as in Rabinovich's monograph (this is rather fortunate, as the methods we use would not actually work in the `second stage' of Hasson and Sustretov's argument). The trick is an unconventional interpretation of the `sum' operation which does not rely on the presence of a group: namely, we build a certain set $\Gamma\subset M^3$ which serves as an abstract replacement for a group operation, even though it is likely not one (Lemma \ref{gamma exists} and Definition \ref{formal sum}); and then show that, nevertheless, substituting $\Gamma$ into the definition of `point-wise addition' still generates additive behavior at the level of differentials (Lemma \ref{slope addition}). A key point here is that we can freely and independently `coordinatize' the tangent spaces to $M$ at each element appearing in the argument (see Lemma \ref{normalization exists}), allowing us to essentially `force' our definition of slopes to respect addition. This seems to be a new idea, since in previous papers all points appearing were the same, and thus so were the coordinatizations.
	
	Section 8 contains our `higher dimensional' detection process for closures of definable sets (Proposition \ref{closure prop}), as described above. This is the most technical part of the paper. In addition to the difficulty of placing the original argument of \cite{PetStaACF}, \cite{HK}, and \cite{EHP} in a higher dimensional context, we need to account for the lack of a group operation. In part, this is taken care of by considering only \textit{coordinate-wise generic} points (i.e. points each of whose coordinates is generic in $M$). However, the group operation is also very helpful in dealing with \textit{semi-indistinguishable} points (Definition \ref{indistinguishable}); since we don't have a group, we will have to perform some rather intricate steps to tread through the issues these points can create (see Propositions \ref{frontier proof} and \ref{generic non-trivial hypersurfaces}).
	
	Finally, in Section 9 we use Proposition \ref{closure prop} to show that $\mathcal M$ detects generic non-transversalities (Theorem \ref{detecting generic non-transversalities}). As described above, the trick is that the frontier point we need to detect has two equal coordinates, which lets us reduce to the `ideal case' of frontier detection. Once we have Theorem \ref{detecting generic non-transversalities}, we then end the paper by easily concluding the main theorem.	
	
	\subsection{About the Characteristic Assumption} The vast majority of the proof goes through in characteristic $p>0$ as well, with some minor modifications. Most of the differences boil down to the availability of simpler differential reasoning in characteristic zero, caused by generic smoothness of dominant morphisms (see Proposition \ref{smooth fiber product}, for example). In particular, in the absence of generic smoothness the `right' definition of generic non-transversalities would be more complicated (the challenge being that in order to connect non-transversality to closure points, one needs a suitable version of Lemma \ref{detecting tangency equivalence}(3)). 
	
	As it turns out, however, the only unavoidable consequence of dropping generic smoothness would be that the results of section 7 do not imply the one-dimensional case in positive characteristic; in particular, the results of section 6 on the higher-dimensional case can be salvaged without generic smoothness in an only slightly more complicated way, which we plan to address in an upcoming paper (see the next paragraph). We should remind the reader, then, that because Hasson and Sustretov give a complete proof of the one-dimensional case in all characteristics (getting around generic smoothness by using `higher-order slopes'), there is in general no harm in omitting positive characteristic in dimension one.
	
	Outside of generic smoothness, the main issues with positive characteristic occur in Section 8 (e.g. the treatment of analytically open maps, and Lemma \ref{open is enough}; while it is possible that a purely algebraic adaptation of this argument exists, the author encountered some subtle difficulties in searching for one). The problem is that, especially in this section, one needs a well-behaved Hausdorff refinement of the Zariski topology. In a future paper, we plan to address the higher-dimensional case in positive characteristic by introducing a valuation topology and working in the theory ACVF of algebraically closed valued fields. However, since this will require a more complicated and abstract argument, and because of the already high complexity of the current paper, we have opted to stick to the complex numbers for now.
	 
\subsection{Acknowledgements} This work is based on ideas from the author's PhD thesis, completed at UC Berkeley in 2021 under the supervision of Tom Scanlon. We thank Tom for his guidance during these earlier stages, including suggesting this problem.

Subsequent work was carried out during postdoctoral study, both at the Fields Institute for Research in Mathematical Sciences in Fall 2021, and at Notre Dame University in Spring 2022. We thank both of these institutions for hosting the author.

Finally, we thank Assaf Hasson, Dave Marker, Kobi Peterzil, Jinhe Ye, and Dmitry Sustretov for helpful conversations on the material; and moreover Assaf Hasson, Sergei Starchenko, and Chieu-Minh Tran for helpful comments on earlier versions of the paper.
	
	\section{Preliminaries on Strongly Minimal Structures}
	
	We begin our work by discussing some of the main tools used in strongly minimal structures -- notably Morley rank, genericity, canonical bases, and families of plane curves. Most of this material is well-known to model theorists, so we omit known or easy proofs. In particular, any `facts' without explicit references can be found in (or follow easily from results in) either \cite{Mar} or \cite{Pil}. We will however introduce some new terminology for families of plane curves -- namely the notions of standard, good, and excellent families -- so we will include more details when discussing these topics. Note that we will assume familiarity with basic model theory (e.g. languages, structures, definable sets, etc.); we refer the reader to any of the standard texts (e.g. \cite{Mar} or \cite{CK}) for background on these notions.
	
	\subsection{Rank}
	
	\begin{convention}\label{sm conventions}
		Throughout the rest of this section, we fix a strongly minimal structure $\mathcal M=(M,...)$ in a language $\mathcal L(\mathcal M)$, and assume the cardinality of $M$ is strictly greater than $|\mathcal L(\mathcal M)|+|\aleph_0|$. By an \textit{interpretable set} in $\mathcal M$, we mean the quotient of a definable set by a definable equivalence relation. We work throughout in the structure $\mathcal M^{\textrm{eq}}$ obtained by adjoining to $\mathcal M$ all interpretable sets (without parameters). However, to avoid confusion the word \textit{definable} will still only refer to subsets of cartesian powers of $M$. By a \textit{set of parameters} we always mean a set of tuples from $\mathcal M^{\textrm{eq}}$ of cardinality less than $|M|$; we then say that $X$ is definable (respectively interpretable) \textit{over} $A$ if it can be defined (respectively interpreted) using only parameters from $A$. Finally, the notation $\operatorname{acl}(A)$ refers to the model theoretic algebraic closure of $A$ in $\mathcal M^{\textrm{eq}}$ -- that is, the union of all finite sets interpretable over $A$.
	\end{convention} 
	\begin{remark} Alternatively, one could just think of $M$ as a strongly minimal definable set in a stable structure $\mathcal N$. This is an equivalent formulation, because such a set can be equipped with all $\mathcal N$-definable subsets of powers of $M$, thereby turning it into a strongly minimal structure in its own right.
	\end{remark}
	We first recall the existence of a well-behaved \textit{rank} (given by \textit{Morley rank}) for interpretable sets and tuples.
	
	\begin{fact}\label{rank exists} For each non-empty interpretable set $X$ in $\mathcal M$, there is an associated non-negative integer $\rk(X)$, satisfying the following properties:
		\begin{enumerate} 
			\item $\rk(X)=0$ if and only if $X$ is finite; in general $\rk(M^r)=r$, $\rk(X\cup Y)=\max\{\rk(X),\rk(Y)\}$, and $\rk(X\times Y)=\rk(X)+\rk(Y)$.
			\item Rank is invariant under automorphisms and finite-to-one interpretable maps.
			\item Rank is interpretable in families: that is, if $\{X_t:t\in T\}$ is an interpretable family of sets, then for each $r\geq 0$ the set $\{t:\rk(X_t)=r\}$ is interpretable without additional parameters.
			\item Every interpretable set $X$ of rank $r\geq 0$ is a finite disjoint union $X_1\cup...\cup X_k$ of interpretable sets of rank $r$, each of which cannot be further broken down into a disjoint union of interpretable sets of rank $r$.
			\item If $X$ is definable over $A$, then $\rk(X)$ is the smallest $r$ such that there is a finite-to-one definable map $f:X\rightarrow M^r$; moreover, for $r=\rk(X)$, there is such a map which is definable over $A$.
		\end{enumerate}
	\end{fact}
	\begin{remark} Note that an arbitrary structure interpreted in an algebraically closed field has \textit{finite Morley rank}, and thus carries a rank satisfying many of the above properties. The properties specific to the strongly minimal setting are (3), (5), and the formulas for $\rk(M^r)$ and $\rk(X\times Y)$ in (1).
	\end{remark}
	For example, if $\mathcal M$ is an algebraically closed field, and $X\subset M^n$ is definable, then $\rk(X)$ is the dimension of the Zariski closure of $X$ (in $M^n$) as an algebraic variety.
	\begin{definition} We will refer to an interpretable set $X\neq\emptyset$ as \textit{stationary} if it is not the disjoint union of two interpretable subsets of the same rank. 
	\end{definition} For example, $M^r$ is stationary for each $r$. Note that Fact \ref{rank exists}(4) says that every interpretable set is a finite union of stationary sets of the same rank. These sets will be called the \textit{stationary components} of $X$. They are well-defined up to \textit{almost equality} (see Definition \ref{almost equal}). In general, if $X$ is interpretable over $A$, then one can arrange $X_1,...,X_k$ to be interpretable over $\operatorname{acl}(A)$.
	
	\begin{definition} The notion of rank extends to tuples over parameters sets, as follows: if $a$ is a tuple from $\mathcal M^{\textrm{eq}}$, and $A$ is a set of parameters, then $\rk(a/A)$ is the smallest rank of an interpretable set over $A$ containing $a$. If $A=\emptyset$, we sometimes write just $\rk(a)$.
	\end{definition}

	Throughout the paper, we will implicitly use various properties of ranks of tuples. The most important property, from which most others can be deduced, is the following:

	\begin{convention} As is customary in model theory, the notation $Aa$ means $A\cup\{a\}$ (and similarly for $AB$, $ab$, etc.).
	\end{convention} 
	
	\begin{fact}\label{ranks of tuples facts}
		Let $a$ and $b$ be tuples, and $A$ a set of parameters. Then $$\rk(ab/A)=\rk(a/A)+\rk(b/Aa).$$
	\end{fact}

	\begin{remark} Fact \ref{ranks of tuples facts} is generally known as `additivity' of rank.
	\end{remark} 
	
	\subsection{Genericity and Independence}
	
	We next recall the following definitions:
	
	\begin{definition}\label{general generic point} Let $X$ be interpretable over $A$, and let $a\in X$. Then $a$ is \textit{generic in $X$ over $A$} if $\rk(a/A)=\rk(X)$.
	\end{definition}
	\begin{definition}
		Let $a_1,...,a_n$ be tuples, and $A$ a set. Then $a_1,...,a_n$ are \textit{independent over $A$} if $\rk(a_1...a_n/A)=\rk(a_1/A)+...+\rk(a_n/A)$.
	\end{definition}

	For example, by additivity, two tuples $a$ and $b$ are independent over $A$ if and only if $\rk(a/Ab)=\rk(a/A)$. Moreover, if $X$ is interpretable over $A$, then two generics $x,y\in X$ over $A$ are independent over $A$ if and only if $(x,y)$ is generic in $X^2$ over $A$.	
	
	Recalling that sets of parameters are assumed to be smaller than $|M|$, we have:
	\begin{lemma} Let $X$ be interpretable over $A$ and non-empty. Then there is a generic element of $X$ over $A$.
	\end{lemma}
	\begin{proof} By Convention \ref{sm conventions}, $\mathcal M$ is saturated.
	\end{proof}

	We also have a notion of genericity of sets:
	
	\begin{definition}\label{general generic set} Let $\emptyset\neq X\subset Y$ be interpretable. We will say that $X$ is \textit{generic} in $Y$ if $\rk(X)=\rk(Y)$, and \textit{large} in $Y$ if $\rk(Y-X)<\rk(Y)$.
	\end{definition}
	Note that if $X$ is large in $Y$ then $X$ is generic in $Y$; the converse holds if $Y$ is stationary.
	
	The following notions can be defined using Definition \ref{general generic set}. They are not necessarily standard (especially (5)), but are useful for our purposes:
	
	\begin{definition}\label{almost equal}
		Let $X\neq\emptyset$, $Y$, and $f:X\rightarrow Y$ be interpretable. 
		\begin{enumerate}
			\item $X$ is \textit{almost contained in }$Y$ if $X\cap Y$ is large in $X$.
			\item $X$ and $Y$ are \textit{almost equal} if they are almost contained in each other.
			\item $f$ is \textit{almost surjective} if $f(X)$ is large in $Y$.
			\item $f$ is \textit{almost finite-to-one} if the union of the finite fibers of $f$ is large in $X$.
			\item $f$ is \textit{generically dominant} if the image of any generic subset of $X$ is generic in $Y$.
		\end{enumerate}
	\end{definition}

	\begin{remark}\label{generic dominance equivalence} It is easy to see that $f:X\rightarrow Y$ is generically dominant if and only if the union of all fibers of rank $\rk(X)-\rk(Y)$ is large in $X$. This will be useful in the next section.
	\end{remark} 

	We note that each of the above notions can also be expressed in terms of generic points:
	
	\begin{lemma}\label{generic point characterization of function properties} Suppose in Definition \ref{almost equal} that $X$, $Y$, and $f$ are all interpretable over $A$. 
		\begin{enumerate}
			\item $X$ is almost contained in $Y$ if every generic element of $X$ over $A$ belongs to $Y$.
			\item $f$ is almost surjective if and only if whenever $y$ is generic in $Y$ over $A$, there is some $x\in X$ with $f(x)=y$.
			\item $f$ is almost finite-to-one if and only if whenever $x$ is generic in $X$ over $A$ we have $x\in\operatorname{acl}(Af(x))$.
			\item $f$ is generically dominant if and only if whenever $x$ is generic in $X$ over $A$, $f(x)$ is generic in $Y$ over $A$.
		\end{enumerate}
	\end{lemma}
	We also note the following facts, which are most easily seen using \ref{generic point characterization of function properties}:
	\begin{lemma}\label{generic dominance facts} Let $f:X\rightarrow Y$ be interpretable.
		\begin{enumerate}
			\item If $f$ is almost finite-to-one then $\rk(X)\leq\rk(Y)$.
			\item If $f$ is almost surjective then $\rk(X)\geq\rk(Y)$.
			\item If $f$ is almost surjective and almost finite-to-one then $\rk(X)=\rk(Y)$ and $f$ is generically dominant.
			\item If $X\subset Z\times Y$ is a family of sets $X_y\subset Z$, with each $X_y$ having the same rank, and $f:X\rightarrow Y$ is the projection, then $f$ is generically dominant.
		\end{enumerate}
	\end{lemma}

	\subsection{Canonical Bases and Almost Faithfulness}
	
	We next discuss canonical bases and their application in producing almost faithful families of sets.
	
	\begin{definition}\label{Cb} Let $X$ be a stationary interpretable set. Then by the \textit{canonical base of $X$}, we mean the canonical base (in the sense of stability theory) of the unique generic type of $X$. That is, a tuple $c$ from $\mathcal M^{\textrm{eq}}$ is a canonical base of $X$ if it codes the set of generic elements of $X$, in the sense that an automorphism of $M$ fixes $c$ if and only if it preserves some generic subset of $X$.
	\end{definition}

	The canonical base of $X$ is not truly well-defined, but it is well-defined up to interdefinability: that is, if $c$ and $c'$ are canonical bases of $X$, then $c$ is interpretable over $c'$ and vice versa. For this reason one identifies all canonical bases with each other in practice, and just writes $c=\operatorname{Cb}(X)$.
	
	The main associated facts are:
	
	\begin{fact}\label{cb facts}
		Let $X$ be stationary.
		\begin{enumerate}
			\item $\operatorname{Cb}(X)$ exists, and can be taken to be a finite tuple in $\mathcal M^{\textrm{eq}}$.
			\item If $c=\operatorname{Cb}(X)$, then there is a set $X'$ which is interpretable over $c$ and almost equal to $X$.
			\item If $X$ is a stationary component of $Y$, and $Y$ is interpretable over $A$, then $\operatorname{Cb}(X)\in\operatorname{acl}(A)$. 
		\end{enumerate}
	\end{fact}
	\begin{proof} (1) and (3) are standard. For (2), let $X$ be interpretable over $A$, and without loss of generality assume $c\in A$. Let $q$ be the generic type of $X$ over $A$, and let $p$ be the restriction of $q$ to $c$; so by definition of the canonical base, $p$ is stationary and $\rk(p)=\rk(q)$. Let $X'\in p$ be stationary of the same rank as $p$. Then $$\rk(X')=\rk(p)=\rk(q)=\rk(X).$$ Moreover, since $X\cap X'\in q$, we get $\rk(X\cap X')\geq\rk(q)=\rk(X)$. Thus $X$ and $X'$ are almost equal.
	\end{proof}

	Suppose that $\{X_t:t\in T\}$ is an interpretable family of sets, each of which has rank $r$. We say that this family is \textit{faithful} if for $t\neq t'$ we have $\rk(X_t\cap X_{t'})<r$. One of the main uses of canonical bases, at least in the context of strong minimality, is in producing faithful families. Namely, using canonical bases, one can show that if $\{X_t:t\in T\}$ is an interpretable family of rank $r$ sets whose generic members are stationary, then $\{X_t\}$ admits a faithful reparametrization $\{Y_u:u\in U\}$: this means that for generic $t\in T$ the set $X_t$ is almost equal to some $Y_u$, and vice versa, and moreover the family $\{Y_u\}$ is faithful. 
	
	The ability to take faithful families is nice -- however in some situations we might wish to stick exclusively to \textit{definable} (not interpretable) families. In this context we can find reparametrizations with a slightly weaker property called \textit{almost faithfulness}. We first state the following well-known fact, which is the reason that almost faithful families exist (see \cite{Hru92b}):
	
	\begin{fact}\label{weak e of i} $\mathcal M$ eliminates imaginaries up to codes for finite sets. That is, let $X$ be any set which is interpretable over a set $A$. Then there are a definable set $Y$ and an interpretable finite-to-one surjective map $f:Y\rightarrow X$. Moreover, if the set $\operatorname{acl}(\emptyset)\cap M$ is infinite, then one can take $Y$ and $f$ to be interpretable over $A$. 
	\end{fact}

	\begin{assumption}\label{acl(0) infinite}
		For the rest of this section, we assume the set $\operatorname{acl}(\emptyset)\cap M$ is infinite. This can be arranged by adding infinitely many constants symbols to $\mathcal L(\mathcal M)$.
	\end{assumption}

	We now give the definition of almost faithfulness and the relevant facts:
	
	\begin{definition}\label{almost faithful} Let $\mathcal X=\{X_t:t\in T\}$ be an interpretable family of rank $r$ sets. We say that $\mathcal X$ is \textit{almost faithful} if for each $t$ there are only finitely many $t'$ such that $\rk(X_t\cap X_{t'})=r$. If $\mathcal X$ is almost faithful, then by $\rk(\mathcal X)$ we mean the rank of the parameter space $T$.
	\end{definition}

	\begin{remark} If $\mathcal X=\{X_t:t\in T\}$ is an almost faithful family of rank $r$ subsets of an interpretable set $S$, then we identify $\mathcal X$ with its \textit{graph} $X\subset S\times T$. We caution the reader that while $\rk(\mathcal X)$ refers the rank of $T$, $\rk(X)$ will refer to the rank of $X$ as an interpretable set in its own right. Thus $\rk(X)$ is typically larger than $\rk(\mathcal X)$.
	\end{remark}
	
	Most of the key properties of almost faithful families boil down to the following, which are easy to check:
	
	\begin{lemma}\label{almost faithful facts} Let $\{X_t:t\in T\}$ be an almost faithful family of rank $r$ sets over $A$.
		\begin{enumerate}
			\item If $Y$ is any rank $r$ set, then there are only finitely many $t$ such that $X_t\cap Y$ has rank $r$.
			\item If $\{Y_u:u\in U\}$ is another almost faithful family of rank $r$ sets over $A$, and $X_t\cap Y_u$ has rank $r$ for some $t$ and $u$, then $t$ and $u$ are interalgebraic over $A$ (i.e. $t\in\operatorname{acl}(Au)$ and vice versa).
			\item If $X$ is a stationary component of $X_t$ for some $t\in T$, and $c=\operatorname{Cb}(X)$, then $c$ and $t$ are interalgebraic over $A$.
		\end{enumerate}
	\end{lemma}
			
	 As hinted at above, perhaps the most important fact about almost faithful families is that they always exist in a precise sense, and can be taken to be indexed by definable sets:
	
	\begin{lemma}\label{almost faithful families exist} Let $X\subset M^n$ be a stationary definable set of rank $r$, and let $c=\operatorname{Cb}(X)$. Let $A$ be a set of parameters, and let $k=\rk(c/A)$. Then there are a definable set $T$ over $A$, an almost faithful family $\{X_t:t\in T\}$ of rank $r$ subsets of $M^n$ which is definable over $A$, and an element $\hat t\in T$, having the following properties:
		\begin{enumerate}
			\item $T$ has rank $k$, and $\hat t$ is generic in $T$ over $A$.
			\item $X$ is almost contained in $X_{\hat t}$.
			\item $c$ and $\hat t$ are interalgebraic over $A$.
		\end{enumerate}
	Moreover, one can choose $\mathcal X$ to satisfy either of the following (but not both simultaneously):
		\begin{enumerate}
			\item[(a)] $X$ is almost equal to $X_{\hat t}$.
			\item[(b)] $T$ is a generic subset of $M^k$.
		\end{enumerate}
	\end{lemma}
	Let us sketch the proof:
	\begin{proof}
		Note that (3) follows from (2) and Lemma \ref{almost faithful facts}(3). So we need only show (1) and (2). Now by Fact \ref{cb facts}, there is a set $X'$ almost equal to $X$ which is definable over $c$. Using the Compactness Theorem, we thus find a rank $k$ faithful interpretable family $\{Y_u:u\in U\}$ over $A$, containing $X'$ as a generic member $Y_{\hat u}$ over $A$, and indexed by a sufficiently large finite fragment of the type of $c$ over $A$ (so in fact $\hat u=c$). By Fact \ref{weak e of i}, there is an interpretable finite-to-one surjective function $f:T\rightarrow U$ over $A$, where $T$ is definable. Note that $\rk(T)=\rk(U)=k$. We then reparametrize by $T$, setting $X_t=Y_{f(t)}$ for each $t$. The result is a rank $k$ almost faithful definable family containing $X'$ as a generic member $X_{\hat t}$ over $A$ whenever $f(\hat t)=\hat u$. We thus satisfy (1), (2), and (a). 
		
		Now suppose we want to satisfy (b) instead. By Fact \ref{rank exists}(5), there is a finite-to-one definable map $g:T\rightarrow M^k$ over $A$. We thus reparametrize by $g(T)$: for $v\in g(T)$, set $Z_v$ to be the union of all $X_t$ for $g(t)=v$. Using that $g$ is finite-to-one, it follows that each $Z_v$ has rank $r$, and $\{Z_v\}$ is still almost faithful. Moreover, $X'$ is contained in $Z_{g(\hat t)}$, which shows that $X$ is almost contained in $Z_{g(\hat t)}$. So this family satisfies (2) and (b). Now as mentioned above we deduce (3) from (2), and since $\rk(g(T))=\rk(T)=\rk(c/A)=k$, (2) and (3) imply (1). 
		
		We note, however, that since each $Z_v$ might now have several stationary components, we cannot hope to preserve (a).
		
	\end{proof}

	\subsection{1-Basedness and Local Modularity}

		We now discuss 1-basedness and local modularity, including the characterization of local modularity in terms of plane curves. We present this characterization as a definition, because it is what is most often used in practice.
	
	\begin{definition} A \textit{plane curve} in $\mathcal M$ is a rank 1 definable subset of $M^2$.
		\end{definition} 
	
	The following fact follows from Theorem 3.4.2 of \cite{Hru85}.
	
	\begin{fact}\label{characterization of local modularity} The following are equivalent:
		\begin{enumerate}
			\item There is a stationary plane curve whose canonical base has rank at least $2$.
			\item For each positive integer $r$, there is a stationary plane curve whose canonical base has rank at least $r$.
			\item There is a definable almost faithful family of plane curves of rank 2.
			\item For each positive integer $r$, there is a definable almost faithful family of plane curves of rank $r$. 
			\item For each positive integer $r$, there is a definable almost faithful family of plane curves over $\emptyset$ of rank $\geq r$.
		\end{enumerate}
	\end{fact}
	
	\begin{definition}\label{local modularity} $\mathcal M$ is \textit{not locally modular} if the conditions of Fact \ref{characterization of local modularity} hold. Otherwise $\mathcal M$ is \textit{locally modular}.
	\end{definition}

	\begin{remark}\label{local modularity when not saturated}
		To be precise, Fact \ref{characterization of local modularity} only holds if $\operatorname{acl}(\emptyset)\cap M$ is infinite and $|M|>|\mathcal L(\mathcal M)|+\aleph_0$. For general strongly minimal $\mathcal M$, locally modularity can be defined by either (3) or (4) in Fact \ref{characterization of local modularity}, and the resulting notion is invariant under taking elementary extensions and adding parameters. 
	\end{remark}
	We now briefly consider structures $\mathcal N=(N,...)$ which are not strongly minimal, in order to discuss 1-basedness (though we still use $\mathcal M$ for our given strongly minimal structure). Precisely we consider such $\mathcal N$ which have finite Morley rank, though it would be enough to assume $\mathcal N$ is interpretable in an algebraically closed field. We note that the theory of generics and canonical bases also applies in this more general framework: in particular if $\mathcal N$ is saturated, then every stationary $\mathcal N$-definable set has a canonical base in $\mathcal N^{\textrm{eq}}$.
	
	The following is a corollary of standard facts in stability theory:
	
	\begin{fact}\label{1-based fact} Assume $\mathcal N$ is saturated, has finite Morley rank. Then the following are equivalent:
		\begin{enumerate}
			\item Whenever $X$ is stationary and interpretable over a set $A$, and $a\in X$ is generic over $A$, the canonical base of $X$ is algebraic over $a$.
			\item Every strongly minimal set in $\mathcal N^{\textrm{eq}}$ (viewed with its full induced structure from $\mathcal N$) is locally modular.
		\end{enumerate}
			If moreover $\mathcal N$ is strongly minimal, then $\mathcal N$ is locally modular if and only if (1) and (2) hold.
	\end{fact}

	\begin{proof} Let 2' be the statement `every minimal type in $\mathcal N^{\textrm{eq}}$ is locally modular.' The equivalence of 1 and 2' is given by Theorems 2.7 and 5.3 of \cite{Bue86}, and the equivalence of 2' and 2 follows from Theorem 1 of \cite{Bue85b}. For strongly minimal $\mathcal N$, the equivalence of 2 and local modularity follows from Theorem 2 of \cite{Bue85}. 
	\end{proof}

	\begin{definition} Let $\mathcal N$ be a structure of finite Morley rank.
		\begin{enumerate}
			\item If $\mathcal N$ is saturated, we say that $\mathcal N$ is \textit{1-based} if the conditions of Fact \ref{1-based fact} hold.
			\item If $\mathcal N$ is not saturated, we say that $\mathcal N$ is 1-based if some (equivalently any) saturated elementary extension of it is.
		\end{enumerate}
	\end{definition}

	\begin{remark}\label{1-based remark} Condition (1) of Fact \ref{1-based fact} roughly places a bound on the ranks of definable families of sets. In general, one thinks of 1-based theories as either `trivial' or `linear,' with no further complexity. For example, no infinite field can ever be 1-based. Moreover, if $(G(K),\cdot)$ is the group of points of an algebraic group over an algebraically closed field $K$, and $V(K)\subset G^n(K)$ is the set of points of a proper relatively closed subvariety $V$ of some $G^n$ which is not a coset of an algebraic subgroup, then the structure $(G(K),\cdot,V(K))$ is never 1-based (by the main theorem of \cite{HP87}).
	\end{remark}

	We now revert to our strongly minimal structure $\mathcal M$. In the ensuing subsections, we will proceed to discuss various properties of almost faithful families of plane curves in $\mathcal M$. Our main goal is to `preen' a given almost faithful family until it has a number of nice properties. We will then deduce that non-local modularity is further equivalent to the existence of a certain special type of family (which we call an \textit{excellent family}).
		
	\subsection{Special Types of Points}
	
	Let $\mathcal C=\{C_t\}_{t\in T}$ be a rank 2, almost faithful family of plane curves. In an ideal world, each element of $M^2$ would belong to a rank 1 subfamily of the curves in $\mathcal C$, and any two distinct elements of $M^2$ would belong to finitely many common curves. Unfortunately this is not guaranteed to be the case. In this subsection we study obstructions to these phenomena. Our goal is to show that such obstructions are `rare enough' that they won't impede on our work.
	
	\begin{notation} Let $\mathcal C=\{C_t\}_{t\in T}$ be a family of plane curves in $\mathcal M$. For each $x\in M^2$, we denote the fiber $\{t\in T:x\in C_t\}$ by ${_xC}$.
	\end{notation}

	We begin with the following definitions:

	\begin{definition}\label{common points} Let $\mathcal C=\{C_t:t\in T\}$ be an almost faithful family of plane curves of rank $r\geq 1$.
	\begin{enumerate}
	\item $x\in M^2$ is \textit{common} in $\mathcal C$ if ${_xC}$ has rank $r$.
	\item $x\in M^2$ is \textit{normal} in $\mathcal C$ if ${_xC}$ has rank $r-1$.
	\item $x\in M^2$ is \textit{rare} in $\mathcal C$ if ${_xC}$ is non-empty and has rank less than $r-1$.
	\item $x\in M^2$ is \textit{absent} in $\mathcal C$ if ${_xC}$ is empty.
	\item $t\in T$ is \textit{elitist} in $\mathcal C$ if $C_t$ contains infinitely many non-normal points of $\mathcal C$.
	\end{enumerate}
	\end{definition}
	
	So a normal point belongs to the `expected number' of curves in the family; while a common point belongs to `too many', and a rare or absent point belongs to `too few'. Note that all five of these notions are definable. Moreover, note that $t\in T$ is elitist if and only if $C_t$ contains infinitely many common or rare points.
	
	The following facts are easy to check:
	
	\begin{lemma}\label{finitely many common points} Let $\mathcal C=\{C_t:t\in T\}$ be an almost faithful family of plane curves of rank $r\geq 1$. 
		\begin{enumerate}
			\item The set of common points of $\mathcal C$ is finite.
			\item The set of normal points of $\mathcal C$ is generic, and thus large, in $M^2$.
			\item The sets of rare and absent points of $\mathcal C$ each have rank at most 1.
			\item The set of elitist indices of $\mathcal C$ is finite.
		\end{enumerate}
		\end{lemma}
	\begin{proof} (1) is proven in more generality in Lemma \ref{hypersurfaces finitely many common points}. (2) follows from (1) via a short rank computation. (3) follows immediately from (2). (4) follows from (3) and Lemma \ref{almost faithful facts}(1).
	\end{proof}
		
		We also discuss the following relation on pairs of points in the plane, roughly saying that they belong to too many common curves.
	
		\begin{definition}\label{indistinguishable} Let $\mathcal C=\{C_t:t\in T\}$ be an almost faithful family of plane curves of rank $r\geq 2$. Let $x,x'\in M^2$. Then $x$ and $x'$ are \textit{semi-indistinguishable} in $\mathcal C$ if the set ${{_xC}}\cap{_{x'}C}$ has rank at least $r-1$.
			\end{definition}
		
		\begin{remark} We note that this is not in general an equivalence relation -- though it would be if each $x\in M^2$ was normal and each ${_xC}$ was stationary. As is, the sets ${_xC}$ could generally have two components each, for example, and then one could have $x$ and $x'$ such that ${_xC}$ and ${_{x'}C}$ share only one component. It is because of this possibility that we have called such points \textit{semi-indistinguishable} instead of \textit{indistinguishable}.
		\end{remark}
		
		The main observation to make about semi-indistinguishable points is the following:
		
		\begin{lemma}\label{indistinguishable interalgebraic}
			Let $\mathcal C=\{C_t:t\in T\}$ be an almost faithful family of plane curves of rank $r\geq 2$ over $A$.
			\begin{enumerate}
			\item If $x\in M^2$ is not common in $\mathcal C$, then there are only finitely many $x'\in M^2$ which are semi-indistinguishable from $x$ in $\mathcal C$.
			\item If $\mathcal C$ has no common points, then any two points in $M^2$ which are semi-indistinguishable in $\mathcal C$ are interalgebraic over $A$.
			\end{enumerate}
		\end{lemma}
		\begin{proof} (1) follows by realizing that, if $x'$ is semi-indistinguishable from $x$, then it is a common point of the subfamily of $\mathcal C$ indexed by ${_xC}$. (2) follows from (1).
		\end{proof}
	
	\subsection{Families of Intersections and Standard Families}
	
		In this subsection we define the family of intersections of two almost faithful families of plane curves, and point out several easy rank properties. We then isolate the notion of a \textit{standard family}, and observe that they always exist in a precise sense. This subsection, combined with the later notion of \textit{excellence}, contains much of the language and material regarding families of curves that we will revisit for the rest of the paper.
		
		\begin{definition} Let $\mathcal C=\{C_t:t\in T\}$ and $\mathcal D=\{D_u:u\in U\}$ be almost faithful families of plane curves, with graphs $C\subset M^2\times T$ and $D\subset M^2\times U$. Then by the \textit{family of intersections of $\mathcal C$ and $\mathcal D$}, we mean the family $\mathcal I=\{I_{(t,u)}:(t,u)\in T\times U\}$ given by $I_{(t,u)}=C_t\cap D_u$. That is, the graph of $\mathcal I$ is the set $I=\{(x,t,u)\in M^2\times T\times U:x\in C_t\cap D_u\}$.
		\end{definition}
	
		The following properties are all easy to check; we omit the proofs:
		
		\begin{lemma}\label{family ranks}
			Let $\mathcal C=\{C_t:t\in T\}$ and $\mathcal D=\{D_u:u\in U\}$ be almost faithful families of plane curves of ranks $r_{\mathcal C},r_{\mathcal D}\geq 1$, respectively. Let $\mathcal I=\mathcal I_{\mathcal C,\mathcal D}$ be the family of intersections, and let $C$, $D$, and $I$ be the graphs of $\mathcal C$, $\mathcal D$, and $\mathcal I$.
			\begin{enumerate}
				\item $\rk(C)=r_{\mathcal C}+1$, and the projections $C\rightarrow M^2$ and $C\rightarrow T$ are generically dominant. 
				\item $\rk(R)=r_{\mathcal C}+r_{\mathcal D}$, and if $\mathcal C$ and $\mathcal D$ have no common points then the projections $I\rightarrow C$ and $I\rightarrow D$ are generically dominant. 
				\item If either $r_{\mathcal C}\geq 2$ or $r_{\mathcal D}\geq 2$ then the projection $I\rightarrow T\times U$ is almost surjective and almost finite-to-one.
			\end{enumerate}
		\end{lemma}
	
		For example, it follows from Lemma \ref{family ranks} that $(x,t)\in C$ is generic if and only if $t\in T$ is generic and $x\in C_t$ is generic over $t$, if and only if $x\in M^2$ is generic and $t\in{_xC}$ is generic. We will freely use facts of this nature throughout the paper.
		
		We now discuss standard families:
		
		\begin{definition}\label{standard family}
			Let $A$ be a set of parameters. By a \textit{standard family of plane curves over $A$}, we mean an almost faithful family $\{C_t:t\in T\}$ of rank $r\geq 1$ over $A$ which has no common points and is indexed by a generic subset of $M^r$. If $A=\emptyset$, we simply call $\mathcal C$ a \textit{standard family of plane curves}.
		\end{definition}
	
		Thus Lemma \ref{family ranks}(2) applies whenever $\mathcal C$ and $\mathcal D$ are standard families.
		
		We point out the following essential restatement of Lemma \ref{almost faithful families exist} (at least in case (b)) for standard families:
		
		\begin{lemma}\label{standard families exist} Let $X$ be a strongly minimal plane curve with canonical base $c$. Let $A$ be a set of parameters, and assume that $\rk(c/A)=r\geq 1$. Then there are a rank $r$ standard family $\mathcal C=\{C_t:t\in T\}$ of plane curves over $A$, and an element $\hat t\in T$, with the following properties:
			\begin{enumerate}
				\item $\hat t$ is generic in $T$ over $A$.
				\item $X$ is almost contained in $C_{\hat t}$.
				\item $c$ and $\hat t$ are interalgebraic over $A$.
			\end{enumerate}
		\end{lemma}
		\begin{proof} In light of Lemma \ref{almost faithful families exist}, the only thing to check is that $\mathcal C$ can be taken to have no common points. This follows from Lemma \ref{finitely many common points} -- indeed, if $\mathcal C$ has common points then we can replace it with the family $\mathcal C'$, defined by letting $C'_t$ contain those elements of $C_t$ which are not common in $\mathcal C$. Then $\mathcal C'$ is still definable over $A$, and each $C'_t$ is cofinite in $C_t$; it follows easily that $\mathcal C'$ is standard over $A$, and (1)-(3) still hold.
		\end{proof}
	
	\subsection{Dual Families}
	
		Let $\mathcal C=\{C_t:t\in T\}$ be a rank 2 almost faithful family of plane curves. By Lemma \ref{finitely many common points}(2), the sets ${_xC}$ for $x\in M$ are generically of rank 1. It is thus natural to think of the ${_xC}$ as a family of curves in $T$ indexed by $M^2$. If moreover $T$ is a generic subset of $M^2$, we obtain a `dual' family of plane curves, exactly analogous to the duality between points and lines in a projective plane. In this subsection we make this connection precise.
		
		\begin{definition}\label{good family} A family $\mathcal C=\{C_t:t\in T\}$ of plane curves is called \textit{good} if each of the following holds:
			\begin{enumerate}
				\item $\mathcal C$ is almost faithful.
				\item $T$ is a generic subset of $M^2$.
				\item $\mathcal C$ has no common or rare points.
				\end{enumerate}
			\end{definition}
		
		So a good family is standard of rank 2 over any set defining it. Our goal is to show that every good family of plane curves has a natural dual in the sense described above, and that the dual is also a good family. We do this now:
		
		\begin{lemma}\label{dual is good} Let $\mathcal C=\{C_t:t\in T\}$ be a good family of plane curves, with graph $C\subset M^2\times T$. Then the set $C^\vee=\{(t,x):(x,t)\in C\}$ is also the graph of a good family of plane curves, which is indexed by the set $T^\vee$ of normal points of $\mathcal C$.
			\end{lemma}
		\begin{proof} We only sketch the proof, as it is straightforward. First, it follows from the goodness of $\mathcal C$ that $C\subset T^\vee\times T$, so $C^\vee\subset T\times T^\vee$. Also, by the definition of normality it follows that each $C^\vee_x={_xC}$ is a plane curve. Thus $C^\vee $ is indeed the graph of a family of plane curves indexed by $T^\vee$. Call this family $\mathcal C^\vee$. It remains to show that $\mathcal C^\vee$ is good. This is essentially immediate -- namely, we note:
			\begin{enumerate}
			\item That $\mathcal C^\vee$ is almost faithful is a restatement of Lemma \ref{indistinguishable interalgebraic}.
			\item That $T^\vee$ is generic in $M^2$ is a restatement of Lemma \ref{finitely many common points}(2).
			\item That $\mathcal C^\vee$ has no common or rare points follows since each $C_t$ has rank 1.
		\end{enumerate}
		\end{proof}
	
		\begin{definition} If $\mathcal C$ is a good family of plane curves, then the family constructed in Lemma \ref{dual is good} will be called the \textit{dual family of $\mathcal C$}, and denoted $\mathcal C^\vee$.
		\end{definition}
		The reader might wish to note that $\mathcal C^{\vee\vee}=\mathcal C$, but we will not need this.
	
		\subsection{Non-triviality and Excellent Families}
		In the final subsection we discuss an improvement of the condition of goodness, which we call \textit{excellence}. First we discuss non-triviality of plane curves:
		
		\begin{lemma}\label{non-trivial characterization} Let $X\subset M^2$ be a plane curve. Then exactly one of the following holds:
				\begin{enumerate}
					\item There is some $m\in M$ such that $X$ contains a cofinite subset of either $M\times\{m\}$ or $\{m\}\times M$.
					\item Each of the two projections $X\rightarrow M$ is finite-to-one.
				\end{enumerate}
			\end{lemma}
			\begin{proof} Immediate by the strong minimality of $M$.
			\end{proof}
			
			\begin{definition} We call plane curves satisfying (1) of Lemma \ref{non-trivial characterization} \textit{trivial}, and those satisfying (2) \textit{non-trivial}.
			\end{definition}
			
			If $X$ is a stationary (i.e. strongly minimal) plane curve, then triviality is equivalent to $X$ being `horizontal' or `vertical' (i.e. having finite symmetric difference with a set of the form $M\times\{m\}$ or $\{m\}\times M$). For general $X$, triviality is equivalent to one of the stationary components of $X$ being horizontal or vertical.
			
			Note that triviality is definable in families, since finiteness is. Thus if $X$ is a stationary non-trivial plane curve, in applying Lemma \ref{standard families exist} we can insist that the resulting family contain only non-trivial plane curves.
			
			Note also that if $X$ is stationary and trivial, then $\operatorname{Cb}(X)$ is a single point in $M$ (i.e. the coordinate appearing infinitely often in $X$) -- and thus $\rk(\operatorname{Cb}(X))\leq 1$. The following is then easy to check:
			
			\begin{lemma}\label{generic non-triviality}
				If $\mathcal C=\{C_t:t\in T\}$ is an almost faithful family of plane curves of rank $\geq 2$ over $A$, and $t\in T$ is generic over $A$, then $C_t$ is non-trivial.
			\end{lemma}
			\begin{proof} By Lemma \ref{almost faithful facts}(3), for each stationary component $X\subset X_t$ we have $$\rk(\operatorname{Cb}(X))\geq\rk(\operatorname{Cb}(X)/A)=\rk(t/A)\geq 2,$$ so $X$ cannot be trivial.
			\end{proof}
			
			We now define:
			
			\begin{definition}\label{excellent family}
				A family $\mathcal C$ of plane curves is \textit{excellent} if it is good, and moreover every plane curve in $\mathcal C$ or $\mathcal C^\vee$ is non-trivial.
			\end{definition}
		
			Our main goal is to show the following lemma and proposition. Again, we only sketch the proofs:
			
			\begin{lemma}\label{excellent families always exist} Let $\mathcal C=\{C_t:t\in T\}$ be an almost faithful family of plane curves, with $T$ a generic subset of $M^2$. If $C$ is the graph of $\mathcal C$, then there is a large definable subset $C'\subset C$ which is the graph of an excellent family of plane curves. 
			\end{lemma}
			\begin{proof}
				First we note the following:
				\begin{claim} Let $D\subset M^4$ be a definable set of rank 3.
					\begin{enumerate}
						\item If $D$ is the graph of a family of plane curves, then there is a large definable subset $D'\subset D$ which is the graph of a good family of plane curves.
						\item If $D$ is the graph of a good family of plane curves, then there is a large definable subset $D'\subset D$ such that each of the four projections $D'\rightarrow M^3$ is finite-to-one.
						\item If each of the four projections $D\rightarrow M^3$ is finite-to-one, then there is a large definable subset $D'\subset D$ which is the graph of a family of plane curves.
					\end{enumerate}
				\end{claim}
			\begin{proof}
				Assume that $D$ is definable over $A$.
				\begin{enumerate}
					\item Let $D$ be the graph of $\mathcal D=\{D_u:u\in U\}$. Let $U'$ be the set of $u\in U$ which are not elitist in $\mathcal D$, and for $u\in U'$ let $D'_u$ be the set of $x\in D_u$ which are normal in $\mathcal D$. It is easy to check that the family $\mathcal D'=\{D'_u:u\in U'\}$ will do, for example using Lemma \ref{finitely many common points}.
					\item For any generic $(x_1,x_2,x_3,x_4)\in D$ over $A$, it follows from Lemma \ref{generic non-triviality} (applied to each of $\mathcal D$ and $\mathcal D^\vee$) that each $x_i$ is algebraic over $A$ and the other $x_j$'s. Now apply the Compactness Theorem.
					\item Let $D'$ be the set of $(x_1,x_2,x_3,x_4)\in D$ such that the fiber $D_{(x_3,x_4)}$ has rank 1. It is easy to see that $D'$ is large in $D$.
				\end{enumerate}
			\end{proof}
			Now using the claim, we build a chain of large subsets $C\supset D\supset E\supset F\supset G$ with the following properties:
			\begin{itemize}
				\item $D$ is the graph of a good family of plane curves.
				\item Each of the four projections $E\rightarrow M^3$ is finite-to-one.
				\item $F$ is the graph of a family of plane curves.
				\item $G$ is the graph of a good family of plane curves.
			\end{itemize}
			Let $G$ be the graph of $\mathcal G$. Then $\mathcal G$ is good and $G\subset E$, which implies that $\mathcal G$ is excellent. 
			\end{proof}	
		
		Finally, we conclude:
		
	\begin{proposition}\label{excellent families exist} $\mathcal M$ is not locally modular if and only if there is an excellent family of plane curves in $\mathcal M$.
	\end{proposition}
	\begin{proof} Clearly an excellent family witnesses non-local modularity. Conversely, assume that $\mathcal M$ is not locally modular. Then there is a plane curve $X$ whose canonical base $c$ has rank $\geq 2$. Choose any set $A$ such that $\rk(c/A)=2$. Then by Lemma \ref{standard families exist} there is a standard family of plane curves over $A$ of rank 2, from which Lemma \ref{excellent families always exist} produces an excellent family.
	\end{proof}

	\section{The Setting, Conventions, and Coherence}
	
	In this section we fix the setting we will work in for the rest of the paper. Notably, this will involve clarifying some issues surrounding the terminology of rank and genericity. We then introduce the convenient notion of coherence, which will help us to deal with these issues smoothly throughout.
	
	\subsection{The Setup}\label{setup}
We will work throughout with a strongly minimal structure $\mathcal M=(M,...)$ which is interpretable in the complex field and is not locally modular. By elimination of imaginaries and quantifiers in algebraically closed fields, we may assume the universe $M$ to be a \textit{constructible} subset of some $\mathbb C^d$ (that is, a finite Boolean combination of affine algebraic sets). Note that this implies $|M|=|\mathbb C|$, so in particular $M$ is uncountable.

Now by Proposition \ref{excellent families exist}, there is an excellent family of plane curves in $\mathcal M$, which we may assume after adding finitely many constants to be definable over $\emptyset$ in $\mathcal M$. Since this excellent family is already a witness to non-local modularity, we may without loss of generality disregard the rest of the structure of $\mathcal M$; in particular, we may assume that the language of $\mathcal M$ is countable, so that the cardinality assumption of Convention \ref{sm conventions} holds. On the other hand, if we then add a countably infinite set of constant symbols to the language, we may further assume that Assumption \ref{acl(0) infinite} holds, and thereby that we can use all facts from the previous section. In particular, note that adding constants does not affect the interpretability of a field in $\mathcal M$, so these added assumptions are harmless.

Now assuming that the language of $\mathcal M$ is countable, it follows that $\mathcal M$ is interpretable in $\mathbb C$ over a countable set of parameters. We may thus work with a fixed expansion $\mathcal K$ of $\mathbb C$ by countably many constants, and thereby assume that $\mathcal M$ is interpretable over $\emptyset$ in $\mathcal K$. In particular, since we only add countably many constant symbols, $\mathcal K$ is also strongly minimal, and also satisfies the cardinality assumption of Convention \ref{sm conventions}. This means that we can take generic points of definable sets in the sense of either $\mathcal M$ or $\mathcal K$.

Putting the last three paragraphs together, we now make the following assumption:

\begin{assumption}\label{M and K} From now through the proof of Theorem \ref{main theorem}, we fix the following data:
	\begin{enumerate}
		\item A structure $\mathcal K$ expanding the complex field by countably many constants.
		\item A strongly minimal structure $\mathcal M$ in a countable language which is interpretable in $\mathcal K$ over $\emptyset$, and whose universe $M$ is a constructible subset of some affine space over $\mathbb C$.
		\item A countably infinite set of constant symbols in the language of $\mathcal M$ whose interpretations in $M$ are distinct.
	\end{enumerate}
	Moreover, we assume that there is an excellent family of plane curves in $\mathcal M$ which is definable over $\emptyset$ in $\mathcal M$.
\end{assumption}

	\begin{convention} We will refer to definable (respectively interpretable) sets in $\mathcal M$ as \textit{$\mathcal M$-definable} (respectively \textit{$\mathcal M$-interpretable}). As in Convention \ref{sm conventions}, we reserve the term \textit{$\mathcal M$-definable} for subsets of cartesian powers of $M$. Definable sets in $\mathcal K$ (equivalently, constructible subsets of powers of $\mathbb C$) will either be called \textit{$\mathcal K$-definable} or \textit{constructible}, and interpretable sets in $\mathcal K$ will be called \textit{$\mathcal K$-interpretable}.
	\end{convention}

	We also make the following convention:
	
	\begin{convention}\label{variety} We will interpret the term \textit{variety} in the classical sense, as a set of points defined by equations -- that is, we freely identify a variety over $\mathbb C$ in the scheme-theoretic sense with its set of $\mathbb C$-points. We also interpret the \textit{Zariski topology} on $V$ as the topology whose closed sets are given by closed subvarieties of $V$ (again, in the classical sense).
	\end{convention}
	
	We adopt Convention \ref{variety} not as an insult to scheme theory, but because it is more suited to doing model theory. Of course, since $\mathbb C$ is algebraically closed, there is little harm in forgetting scheme-theoretic points. 
	
	Since $M$ is embedded into an affine space over $\mathbb C$, we will treat all $\mathcal M$-definable sets as inheriting both the Zariski and analytic topologies from $\mathbb C$. Thus, for example, even if we started with a strongly minimal structure on a projective variety $V$, we have opted to endow $V$ with the affine topologies inherited from a discontinuous embedding into a higher dimensional affine space. We stress that we will work with \textit{both} topologies (Zariski and analytic) on each power of $M$ -- and in many cases we will be intentionally vague about which topology we refer to. This is because of the following:

\begin{fact}[\cite{Mum}, I.10, Corollary 1]\label{zariski and analytic closures agree} Let $X\subset\mathbb C^d$ be any constructible set. Then the Zariski and analytic closures of $X$ are the same.
\end{fact}

In particular, the following is justified:

\begin{convention}
	Let $X$ be any constructible subset of some $\mathbb C^d$ (so in particular $X$ could be any $\mathcal M$-definable set). The notation $\overline X$ will denote the closure of $X$ with respect to either the Zariski or analytic topology on $\mathbb C^d$. The notation $\operatorname{Fr}(X)$ will denote the frontier of $X$, namely $\overline{X}-X$.
\end{convention}

\subsection{Dimension and Rank}

Note that since there are two strongly minimal structures at play in our setting, we have two ways to interpret each of the concepts of the previous section. We now discuss how these interpretations relate to each other, and specify the terminology we will use to distinguish them. In large part, we will treat $\mathcal M$ as the `default' structure, referring to $\mathcal K$ only in the background. However, there are a few exceptions. Let us make them explicit: 

\begin{convention}\label{dim rk conventions} Throughout the rest of the paper, we adopt the following:
	\begin{enumerate}
		\item Unless otherwise stated (and in particular everywhere outside Sections 4 and 7), all tuples are assumed to be taken in $\mathcal M^{\textrm{eq}}$. As in Convention \ref{sm conventions}, all parameter sets are assumed to be sets of tuples from $\mathcal M^{\textrm{eq}}$ of cardinality less than $|M|=2^{\aleph_0}$. 
		\item The term \textit{rank}, and the notation rk, will always refer to the notion of rank in $\mathcal M$. 
		\item The term \textit{plane curve}, and all properties of plane curves and families of plane curves, are interpreted in the sense of $\mathcal M$. Similarly, the terms \textit{stationary}, \textit{stationary component}, and \textit{canonical base} are always interpreted in the sense of $\mathcal M$.
		\item When referring to the rank functions of $\mathcal K$, we will use the term \textit{dimension}, and the notation dim.
		\item On the other hand, the terms \textit{generic}, \textit{independent}, and \textit{algebraic}, and the notation acl(A), are always interpreted in the sense of $\mathcal K$. We will refer to the corresponding terms in $\mathcal M$ with the prefix $\mathcal M$ (e.g. $\mathcal M$-generic), and algebraic closure in $\mathcal M^{\textrm{eq}}$ will be denoted $\operatorname{acl}_{\mathcal M}(A)$. 
	\end{enumerate}
\end{convention}

\begin{remark} It is not truly necessary to differentiate between parameter sets in $\mathcal K$ and $\mathcal M^{\textrm{eq}}$. Indeed, the reader may wish to verify as an exercise that every tuple from $\mathbb C$ is necessarily $\mathcal K$-definable over a finite tuple from $\mathcal M^{\textrm{eq}}$; after doing this, one can treat even parameters from $\mathcal K$ as residing in $\mathcal M^{\textrm{eq}}$. We have chosen to differentiate the two structures anyway because it seemed more natural to the author.
\end{remark}

The choices made in Convention \ref{dim rk conventions} were intended to allow for the most intuitive presentation of the paper. In particular, note that our use of \textit{dimension} agrees with the usual dimension theory of varieties; our use of \textit{genericity} is geometrically natural, encompassing notions like smoothness, etc.; and our use of \textit{algebraicity} is just the usual field-theoretic algebraic closure. However, things might still get confusing, and we will try to remind the reader at times how certain statements should be interpreted. Perhaps most importantly, note that a \textit{plane curve} is always a rank 1 subset of $M^2$ -- so in particular does not need to be an algebraic curve, or contained in the actual `plane' $\mathbb A^2(\mathbb C)$. There should be no confusion created here, as we will not need any of the theory of plane curves in the algebraic sense.

Note that $\dim M>0$, because $M$ is infinite. Note also that $M^r$ has rank $r$ and dimension $r\cdot\dim M$ for each $r$, since both rank and dimension respect products (i.e. applying Fact \ref{rank exists}(2) in both $\mathcal M$ and $\mathcal K$). Generalizing this, the most important fact to note about dimension and rank is the following: 

\begin{lemma}\label{dim rk equality} Let $X$ be $\mathcal M$-interpretable and of rank $r\geq 0$. Then $\dim X=r\cdot\dim M$.
\end{lemma}
\begin{proof} We induct on $r$. First suppose $r=0$. Then $X$ is finite, so also $\dim X=0$, and there is nothing to prove.
	
	Now assume $r\geq 1$ and the lemma holds for all $r'<r$. By Fact \ref{weak e of i}, there are an $\mathcal M$-definable set $Y$ and an $\mathcal M$-interpretable finite-to-one surjective map $f:Y\rightarrow X$. Since both rank and dimension are preserved under finite-to-one maps, we get $\rk(Y)=r$ and $\dim Y=\dim X$. So it suffices to show that $\dim Y=r\cdot\dim M$.
	
	Now by Fact \ref{rank exists}(5), there is an $\mathcal M$-definable finite-to-one map $g:Y\rightarrow M^r$. By the same reasoning as above, $\rk(g(Y))=r$. Then since $M^r$ is stationary of rank $r$, we get $\rk(M^r-g(Y))<r$, so by induction $\dim(M^r-g(Y))<r\cdot\dim M$. Then since $\dim(M^r)=r\cdot\dim M$, it must be that $\dim(g(Y))=r\cdot\dim M$. Finally, again by preservation under finite-to-one maps, we get $\dim Y=r\cdot\dim M$, as desired.
\end{proof}

\begin{remark} For intuition, we note that Lemma \ref{dim rk equality} is analogous to the fact that the real dimension of a complex space is always double its complex dimension, which follows from the equality $\dim_{\mathbb R}(\mathbb C)=2$.
\end{remark}

\begin{remark}\label{no prefix for sets} It follows easily from Lemma \ref{dim rk equality} that if $\emptyset\neq X\subset Y$ are $\mathcal M$-interpretable, then $X$ is generic in $Y$ if and only if $X$ is $\mathcal M$-generic in $Y$. Thus, for genericity of $\mathcal M$-interpretable sets, we will omit the prefix $\mathcal M$. In fact we will do the same for all terms defined in Definitions \ref{general generic set} and \ref{almost equal}, since each of these notions can be expressed in terms of the genericity of a particular set (for generic dominance this is by Remark \ref{generic dominance equivalence}, and for all other notions it is by definition). 
\end{remark}

\begin{remark} We caution that Lemma \ref{dim rk equality} does not say anything about stationarity. In general a set which is stationary (in the sense of $\mathcal M$) might break into several components over $\mathcal K$. In particular each $M^r$ is stationary, but it is possible for example that $M$ could be the union of two irreducible curves; in this case each $M^r$ would be the union of $2^r$ irreducible $r$-dimensional varieties.
\end{remark}

For genericity of tuples, we do get distinct notions in $\mathcal M$ and $\mathcal K$, since a priori an $\mathcal M$-generic element of a set need not necessarily be generic. However, the converse holds -- indeed, the following two corollaries are essentially immediate from Lemma \ref{dim rk equality} (and, technically, the fact that $\mathcal M$ is interpreted over $\emptyset$ in $\mathcal K$):

\begin{corollary}\label{dim rk inequality} Let $a$ be a tuple, and $A$ a set of parameters. Then $\dim(a/A)\leq\rk(a/A)\cdot\dim M$.
\end{corollary}
\begin{proof} Let $r=\rk(a/A)$. Let $X$ be $\mathcal M$-interpretable over $A$ of rank $r$ with $a\in X$. Then $X$ is also $\mathcal K$-interpretable over $A$, so $\dim(a/A)\leq\dim(X)=r\cdot\dim M$. 
\end{proof}

\begin{corollary}\label{K gen implies M gen} Let $X$ be $\mathcal M$-interpretable over a set $A$ of parameters. Let $a\in X$. If $a$ is generic in $X$ over $A$, then $a$ is $\mathcal M$-generic in $X$ over $A$.
\end{corollary}
\begin{proof} Let $r=\rk(X)$, and assume $a$ is generic in $X$ over $A$. So $\dim(a/A)=r\cdot\dim M$. Then by Corollary \ref{dim rk equality} we get $\rk(a/A)\geq\frac{\dim(a/A)}{\dim M}=r$, which shows that $a$ is $\mathcal M$-generic in $X$ over $A$. 
\end{proof}

\begin{remark} It is not immediate that independence should imply $\mathcal M$-independence, but we will not need this.
\end{remark}

\subsection{Coherence}
In this subsection we introduce the notion of \textit{coherence}, which will help to streamline many computational arguments as we go. We begin with the definition, which is natural in light of Corollary \ref{dim rk inequality}:

\begin{definition}\label{coherent} Let $a$ be a tuple, and $A$ a set of parameters. Then $a$ is \textit{coherent over} $A$ if $\dim(a/A)=\rk(a/A)\cdot\dim M$.
\end{definition}

The idea is that coherent points are those for which $\mathcal M$-genericity can be transferred to full genericity. Precisely, we note the following:

\begin{lemma}\label{coherent iff generic} Let $X$ be $\mathcal M$-interpretable over $A$, and let $a\in X$. Then $a$ is generic in $X$ over $A$ if and only if it is both $\mathcal M$-generic in $X$ over $A$ and coherent over $A$.
\end{lemma}
\begin{proof}
	Let $r=\rk(X)$. First assume that $a$ is generic in $X$ over $A$. By Corollary \ref{K gen implies M gen}, $a$ is also $\mathcal M$-generic over $A$. By $\mathcal M$-genericity, $\rk(a/A)=r$. By Lemma \ref{dim rk equality} and genericity, $\dim(a/A)=r\cdot\dim M=\rk(a/A)\cdot\dim M$. So $a$ is coherent over $A$.
	
	Now assume that $a$ is both $\mathcal M$-generic in $X$ over $A$ and coherent over $A$. By $\mathcal M$-genericity, $\rk(a/A)=r$. Then by coherence, $\dim(a/A)=r\cdot\dim M=\dim(X)$. So $a$ is generic in $X$ over $A$.
\end{proof}

Coherent tuples enjoy the following useful preservation properties:

\begin{lemma}\label{coherent preservation} Let $a$ and $b$ be tuples, and $A$ a set of parameters.
	\begin{enumerate}
		\item $ab$ is coherent over $A$ if and only if $a$ is coherent over $A$ and $b$ is coherent over $Aa$.
		\item If $a$ is coherent over $A$ and $b\in\operatorname{acl}_{\mathcal M}(Aa)$, then $b$ is coherent over $A$ and $a$ is coherent over $Ab$.
		\item If $a$ and $b$ are each coherent over $A$, and $a$ and $b$ are independent over $A$, then $a$ and $b$ are also $\mathcal M$-independent over $A$, and $ab$ is coherent over $A$.
	\end{enumerate}
\end{lemma}
\begin{proof} Throughout we assume $A=\emptyset$ for ease of notation.
	\begin{enumerate}
		\item We have each of the following:
		\begin{enumerate}
			\item $\dim(a,b)=\dim a+\dim(b/a)$.
			\item $\rk(a,b)\cdot\dim M=\rk(a)\cdot\dim M+\rk(b/a)\cdot\dim M$.
			\item $\dim(a,b)\leq\rk(a,b)\cdot\dim M$.
			\item $\dim a\leq\rk(a)\cdot\dim M$.
			\item $\dim(b/a)\leq\rk(b/a)\cdot\dim M$. 
		\end{enumerate}
		We are asked to prove that (c) is an equality if and only if (d) and (e) are equalities. But this follows immediately from (a) and (b).
		\item We have $\rk(b/Aa)=0$. By Corollary \ref{dim rk inequality}, we get $\dim(b/Aa)=0$ as well. Thus $b$ is coherent over $Aa$. Then by (1), $ab$ is coherent over $A$; and by (1) again (switching the order of $a$ and $b$), we conclude that $b$ is coherent over $A$ and $a$ is coherent over $Ab$.
		\item Let $X$ be a set of rank $\rk(b/A)$ which is $\mathcal M$-interpretable over $A$ and contains $b$. Then $b$ is $\mathcal M$-generic in $X$ over $A$, so by Lemma \ref{coherent iff generic} $b$ is actually generic in $X$ over $A$. Since $a$ and $b$ are independent, $b$ is thus also generic in $X$ over $Aa$. So again by Lemma \ref{coherent iff generic}, $b$ is coherent over $Aa$. It now follows by (1) that $ab$ is coherent over $A$. Moreover, since $b$ is generic in $X$ over both $A$ and $Aa$, Corollary \ref{K gen implies M gen} gives that $b$ is $\mathcal M$-generic in $X$ over both $A$ and $Aa$, which shows that $a$ and $b$ are $\mathcal M$-independent over $A$.
	\end{enumerate}
\end{proof}

Throughout, we will often use Lemma \ref{coherent iff generic} to reduce questions of dimension and genericity to questions of rank and $\mathcal M$-genericity, while separately verifying the coherence of the required tuples using Lemma \ref{coherent preservation}. We do this mainly for ease of understanding, as computations with rank tend to be simpler and (we hope) easier to follow. As an illustration of this method, we end this section by giving two examples:

The first example was already mentioned in Remark \ref{no prefix for sets}, but as we will use it repeatedly we give a more explicit treatment, and a (hopefully illuminating) alternate proof:

\begin{lemma}\label{generic dominance no prefix} Let $f:X\rightarrow Y$ be $\mathcal M$-interpretable. Then $f$ is generically dominant in the sense of $\mathcal M$ if and only if it is generically dominant in the sense of $\mathcal K$. Thus we have only one notion of generic dominance of functions.
\end{lemma}
\begin{proof} Assume that $X$, $Y$, and $f$ are $\mathcal M$-interpretable over $A$, and first assume that $f$ is generically dominant in the sense of $\mathcal K$. Let $X'\subset X$ be $\mathcal M$-generic. Then $X'$ is also generic in $X$, so $f(X')$ is generic, and thus also $\mathcal M$-generic, in $Y$. Thus $f$ is generically dominant in the sense of $\mathcal M$.
	
	Now assume that $f$ is generically dominant in the sense of $\mathcal M$. In this case we use the equivalent condition in Lemma \ref{generic point characterization of function properties}. That is, let $x\in X$ be generic over $A$, and let $y=f(x)$. We claim that $y$ is generic in $Y$ over $A$. To see this, we first note by Lemma \ref{coherent iff generic} that $x$ is both $\mathcal M$-generic over $A$ and coherent over $A$. By assumption on $f$, we conclude that $y$ is $\mathcal M$-generic in $Y$ over $A$. Moreover, by definition we have $y\in\operatorname{acl}_{\mathcal M}(Ax)$, so by Lemma \ref{coherent preservation} $y$ is also coherent over $A$. Then by Lemma \ref{coherent iff generic} again, $y$ is generic in $Y$ over $A$. This shows that $f$ is generically dominant in the sense of $\mathcal K$.
\end{proof}

For example, Lemma \ref{generic dominance no prefix} immediately gives the following useful corollary (by combining with Lemma \ref{family ranks}):

\begin{corollary}\label{coherent family ranks}
	Let $\mathcal C=\{C_t:t\in T\}$ be a rank $\geq 1$ almost faithful family of plane curves, $\mathcal M$-interpretable over a set $A$, with graph $C\subset M^2\times T$. Then an element $(x,t)\in M^2\times T$ is generic in $C$ over $A$ if and only if $t$ is generic in $T$ over $A$ and $x$ is generic in $C_t$ over $At$, if and only if $x$ is generic in $M^2$ over $A$ and $t$ is generic in ${_xC}$ over $Ax$.
\end{corollary}

\begin{remark}
	Note that the more obvious argument used in the `right to left' direction of Lemma \ref{generic dominance no prefix} does not work for the second direction, because an arbitrary generic subset of $X$ might not be $\mathcal M$-definable. The point, then, is that working with coherent tuples allows us even in this sort of situation to reduce our work to the structure $\mathcal M$. 
\end{remark}

Our second example is a useful improvement on Lemmas \ref{almost faithful families exist} and \ref{standard families exist} in the context of coherence:

\begin{lemma}\label{coherent version of almost faithful family existence} Let $X\subset M^n$ be a stationary $\mathcal M$-definable set of rank $r$, and let $c=\operatorname{Cb}(X)$. Let $A$ be a set of parameters, and assume that $\rk(c/A)=k$ and $c$ is coherent over $A$. Then there are an $\mathcal M$-definable set $T$ over $A$, an almost faithful family $\{X_t:t\in T\}$ of rank $r$ subsets of $M^n$ which is $\mathcal M$-definable over $A$, and an element $\hat t\in T$, having the following properties:
	\begin{enumerate}
		\item $T$ has rank $k$, and $\hat t$ is generic in $T$ over $A$ (that is, in the full sense of $\mathcal K$).
		\item $X$ is almost contained in $X_{\hat t}$.
		\item $c$ and $\hat t$ are $\mathcal M$-interalgebraic over $A$.
	\end{enumerate}
	Moreover, one can arrange either of the following (but not both simultaneously):
	\begin{enumerate}
		\item[(a)] $X$ is almost equal to $X_{\hat t}$.
		\item[(b)] If $X$ is a plane curve, then $\{X_t:t\in T\}$ is a standard family of plane curves over $A$.
	\end{enumerate}
\end{lemma}
\begin{proof} Directly applying Lemmas \ref{almost faithful families exist} and \ref{standard families exist} yields $\{X_t:t\in T\}$ and $\hat t\in T$ satisfying almost all of the requirements -- the only difference being that $\hat t$ is only guaranteed to be $\mathcal M$-\textit{generic} in $T$ over $A$, not fully generic. So our only task is to verify the genericity of $\hat t$. But since $c$ is coherent over $A$, it follows from (3) and Lemma \ref{coherent preservation} that $\hat t$ is also coherent over $A$. Then since $\hat t$ is $\mathcal M$-generic in $T$ over $A$, Lemma \ref{coherent iff generic} gives that $t$ is indeed generic in $T$ over $A$.
\end{proof}

	\section{Geometric Preliminaries}
	
	In this section we give a few facts that we will need from algebraic and complex geometry. The main goal is Proposition \ref{smooth fiber product}, which we will use repeatedly. The average reader would not be too hindered by taking the statement of this proposition as a black box and, modulo basic definitions, skipping the rest of the section. 
	
	Note that the structure $\mathcal M$ will only appear in side remarks in this section, as the facts we discuss naturally pertain to all complex constructible sets. Thus (with the sole exception of Remark \ref{tangent space over cb}) one should interpret parameter sets as sets of tuples from $\mathcal K$. Equivalently, one could think of parameters as just coming from $\mathbb C$, and all facts stated would still be true; we stick to $\mathcal K$ only so as to not introduce even more notions of genericity, etc.
	
	\subsection{Local Smoothness and Tangent Spaces}
	
	Here we develop the basic notion of smoothness in constructible sets, and the associated tangent space at a smooth point. There should be nothing surprising here -- we just need to clarify that what we will do later is sensical.
	
	\begin{definition}\label{smooth point} Let $X\subset\mathbb C^n$ be a constructible set, and let $x\in X$. We say that $x$ is \textit{smooth in} $X$, or $X$ is \textit{smooth near} $x$, if there is a relatively Zariski open set $X'\subset X$ which contains $x$ and is a smooth variety of dimension $\dim X$.
		\end{definition}
	
	Equivalently, one forms the set of smooth points of $X$ by taking the smooth locus (as a variety) of the Zariski interior of $X$ in its Zariski closure. In particular, by this reformulation we have the following:
	
	\begin{lemma}\label{smooth locus definable} Let $X$ be a non-empty constructible set, and let $X^S$ be its set of smooth points.
	\begin{enumerate} 
		\item $X^S$ is $\mathcal K$-definable over the same parameters that define $X$.
		\item $X^S$ is large in $X$.
		\item If $x$ is generic in $X$ over any set of parameters defining $X$ in $\mathcal K$, then $x\in X^S$. 
	\end{enumerate}
	\end{lemma}
	\begin{proof} Note by the previous paragraph that $X^S$ is Zariski open in its Zariski closure; thus it is constructible, and so $\mathcal K$-definable. Moreover, $X^S$ is clearly invariant under field automorphisms fixing the parameters that define $X$; this implies that $X^S$ is defined over the same parameters, which shows (1). For (2) it suffices to note that $X$ is large in its closure, and the smooth locus of a variety is dense (thus large) in each top-dimensional irreducible component. Finally, (3) is a restatement of (2).
	\end{proof}

	We also have:
	
	\begin{lemma}\label{smooth product} If $X$ and $Y$ are constructible, $x$ is smooth in $X$, and $y$ is smooth in $Y$, then $(x,y)$ is smooth in $X\times Y$.
	\end{lemma} 
	\begin{proof} Smooth varieties are closed under products.
	\end{proof}

	Let us now continue to discuss our main use of smooth points -- namely the availability of a well-behaved tangent space:
	
	\begin{definition}\label{tangent space} Let $X\subset\mathbb C^n$ be a constructible set, and let $x\in X$ be a smooth point. Let $X'$ be as in Definition \ref{smooth point}. Then by the \textit{tangent space} $T_x(X)$ to $X$ at $x$, we mean the Zariski tangent space $T_x(X')$ to the variety $X'$, viewed as an affine linear subspace of $\mathbb C^n$ which contains $x$.
	\end{definition}

	Note that in the situation of Definition \ref{tangent space} the tangent space $T_x(X)$ is well-defined; in particular, any two smooth relative neighborhoods of $x$, each as in Definition \ref{smooth point}, will have the same Zariski tangent space at $x$.
	
	In what follows, we will use the typical structure of $T_x(X)$ as a $\mathbb C$-vector space. We have identified the tangent space as an affine linear set in Definition \ref{tangent space} in order to see it as a definable object (see Lemma \ref{tangent space definable} below).

	The restriction to smooth points in Definition \ref{tangent space} is not necessary in general for a tangent space to make sense, but it is convenient for our purposes. In particular, by Lemmas \ref{smooth locus definable} and \ref{smooth product}, we have the following: if $X_1,...,X_k$ are constructible, and $x_i$ is generic in $X_i$ for each $i$ over some set of parameters, then $(x_1,...,x_k)$ is smooth in $X_1\times...\times X_k$, and thus Definition \ref{tangent space} makes sense. Indeed, it will turn out that points of this form are all that we really use.
	
	We should point out that all of the usual functorial properties of tangent spaces work in our context, because we are just taking a special case of Zariski tangent spaces. Indeed, because we restrict to smooth points, we even have $\dim(T_x(X))=\dim X$ at all points we consider. We also point out the following, similarly to Lemma \ref{smooth locus definable}:
	
	\begin{lemma}\label{tangent space definable} Let $X$ be a constructible set, and let $x\in X$ be smooth. Then the tangent space $T_x(X)$, and its associated vector space structure, are $\mathcal K$-definable over the union of $x$ with the parameters defining $X$.
	\end{lemma}
	\begin{proof} One needs only note that the structure of $T_x(X)$ is computed directly from the polynomials defining $X$, expanded around the point $x$.
	\end{proof}

	\begin{remark}\label{tangent space over cb} We end this subsection by clarifying a subtle issue related to the distinction between definable sets and types, and their relationship to smoothness. Namely, suppose that $
	C$ is a strongly minimal plane curve in $\mathcal M$, with canonical base $c$. Then the assertion `$x$ is generic in $C$ over $c$' is well-defined, even if $C$ is not defined over $c$. Indeed, in this case $c$ still codes the generic $\mathcal M$-type of $C$, and the genericity of $x$ simply says that $x$ both realizes this $\mathcal M$-type over $c$ and is coherent over $c$.
	
	Furthermore, if $x$ is generic in $C$ over $c$ in the sense described above, then $x$ is automatically smooth in $C$, and the tangent space $T_x(C)$ is determined entirely by $c$. Indeed, since $c=\operatorname{Cb}(C)$, by Fact \ref{cb facts} there is some plane curve $C'$ which is $\mathcal M$-definable over $c$ and differs from $C$ by finitely many points. So $x$ is generic in $C'$ over $c$, and is therefore smooth in $C'$. Then, since $C\Delta C'$ is finite, the two sets agree in some Zariski neighborhood of $x$; thus $x$ is also smooth in $C$, and the tangent spaces $T_x(C)$ and $T_x(C')$ are the same. In subsequent sections, we will freely exploit smoothness and tangent spaces at generic points of curves which are given over canonical bases. 
	\end{remark}

	\subsection{Non-Injective Points}
	
	In this short subsection we discuss the natural generalization of ramification to maps of arbitrary constructible sets -- namely the notion of \textit{non-injective points}. These points have played an important role in many similar papers before now (\cite{HZ},\cite{HK},\cite{EHP}). Indeed, understanding non-injective points of certain maps is one of the keys to detecting tangencies of plane curves.
	
	\begin{definition} Let $X\subset\mathbb C^m$ and $Y\subset\mathbb C^n$ be constructible sets, and $f:X\rightarrow Y$ a constructible function. Let $x\in X$. We say that $x$ is an \textit{injective point} of $f$ if there is an analytic neighborhood $U\ni x$ such that the restriction of $f$ to $U$ is injective.	Otherwise we say that $x$ is a \textit{non-injective point} of $f$.
	\end{definition}
	
	We have the following immediate consequences:
	\begin{lemma}\label{non-injective characterization} 
		Let $X\subset\mathbb C^m$ and $Y\subset\mathbb C^n$ be constructible sets, and $f:X\rightarrow Y$ a constructible function. Let $x\in X$. 
		\begin{enumerate}
			\item $x$ is a non-injective point of $f$ if and only if $(x,x)$ belongs to the frontier of the set $$X\times_YX-\Delta=\{(x_1,x_2)\in X^2:x_1\neq x_2\wedge f(x_1)=f(x_2)\}.$$
			\item If $X$ and $Y$ are varieties and $f$ is a morphism, then $x$ is a non-injective point of $f$ if and only if $f$ ramifies at $x$ (that it, the induced map on tangent spaces is non-injective at $x$).
		\end{enumerate}
	\end{lemma}
	\begin{proof}
		(1) is an easy topology exercise, using only that the analytic topology on $X\times X$ is the product topology. (2) is then immediate, since it is a standard fact that ramification is equivalent to (1) (for example this follows from \cite{EGA4}, Corollaire 17.4.2).
	\end{proof}

	\subsection{Analytically Open Maps}
One of the main uses of characteristic zero in this paper is a particular application of maps which are open in the analytic topology. In this subsection we give the basic results that we need.

\begin{definition} Let $f:X\rightarrow Y$ be a continuous map of topological spaces, and let $x\in X$.
	\begin{enumerate}
		\item We say that $f$ is \textit{open near} $x$ if there is an open neighborhood $U$ of $x$ such that the restriction of $f$ to $U$ is an open map.
		\item If $X\subset\mathbb C^m$ and $Y\subset\mathbb C^n$ are constructible sets, and $f$ is a continuous constructible map, we say that $f$ is \textit{analytically open near} $x$ if $f$ is open near $x$ in the analytic topologies on $X$ and $Y$.
	\end{enumerate} 
\end{definition}

Later on, we will need to know that a certain map is analytically open near several specific points. Our goal now is to give the condition that will imply openness at these points. We first recall the following result about complex spaces (\cite{GraRem}, p.107):

\begin{fact}\label{complex spaces fact} Let $f:X\rightarrow Y$ be a holomorphic map between pure dimensional complex spaces of the same dimension. Assume that $Y$ is locally irreducible, and all fibers of $f$ are discrete. Then $f$ is an open map.
\end{fact}

Now our desired statement is:

\begin{lemma}\label{conditions for openness} Let $X$ and $Y$ be smooth varieties of the same dimension, and let $f:X\rightarrow Y$ be a morphism. Let $x\in X$, and set $y=f(x)$. If the fiber $f^{-1}(y)$ is finite, then $f$ is analytically open near $x$. 
\end{lemma}
\begin{proof}
	By assumption, $f$ is quasi-finite at $x$. Since the quasi-finite locus of a morphism is open, there is a relatively Zariski open set $X'\subset X$ such that the restriction of $f$ to $X'$ is quasi-finite. Note that by smoothness we have $\dim X'=\dim X=\dim Y$. Thus Fact \ref{complex spaces fact} applies restricted to $X'$; we conclude that $f$ is analytically open on $X'$, which is enough to prove the lemma.
\end{proof}

Before moving on, we note the following easy corollary:

\begin{corollary}\label{generic openness} Let $X$ and $Y$ be $\mathcal K$-definable sets of the same dimension, and let $f:X\rightarrow Y$ be a function which is given by a coordinate projection. Assume that $f$ is almost finite-to-one. Let $x\in X$ be generic over any set of parameters defining $X$, $Y$, and $f$. Then $f$ is analytically open near $x$.
\end{corollary}
\begin{proof}
	Since $f$ is almost finite-to-one, $f(x)=y$ is generic in $Y$ over the relevant parameters, and therefore has a finite fiber under $f$. Now since $x$ and $y$ are generic in $X$ and $Y$, respectively, we can restrict to smooth varieties $X'\subset X$ and $Y'\subset Y$ containing the respective points. Shrinking if necessary, we may assume that $f(X')\subset Y'$. Thus, since it is given by a projection, $f$ defines a morphism from $X'$ to $Y'$. Note that $\dim X'=\dim Y'$, since $\dim X=\dim Y$. Thus Lemma \ref{conditions for openness} applies.
\end{proof}

	\subsection{Behavior Under Fiber Products}
	
	We now give the main result of this section -- namely Proposition \ref{smooth fiber product} below -- whose purpose is to summarize the behavior of the notions we have been discussing under fiber products. In fact this result encompasses much of the geometric content of the paper. In essence, it says that the `ideal' behavior of smoothness, non-injectivity, and analytic openness, when restricted to nice enough points in a fiber product, can be determined `one set at a time'.
	
	\begin{notation} In what follows, given a set $S\subset A\times B$ and an element $b\in B$, we denote the fiber $\{a\in A:(a,b)\in S\}$ by $S_b$.
	\end{notation}
	
	\begin{proposition}\label{smooth fiber product}
		Let $X$, $Y_1$, $Y_2$, $Z_1\subset X\times Y_1$, and $Z_2\subset X\times Y_2$ be constructible sets. Assume for $i=1,2$ that $X$, $Y_i$, and $Z_i$ are $\mathcal K$-definable over a set $A_i$. Assume further that all four of the projections $Z_i\rightarrow X$ and $Z_i\rightarrow Y_i$ are generically dominant. Let $W=\{(x,y_1,y_2)\in X\times Y_1\times Y_2:(x,y_1)\in Z_1\wedge(x,y_2)\in Z_2\}$. Let $(x,y_1,y_2)\in W$ such that each $(x,y_i)$ is generic in $Z_i$ over $A_i$. Then:
		\begin{enumerate}
			\item There is a relatively Zariski open subset $W'\subset W$ containing $(x,y_1,y_2)$ which is a smooth variety of dimension $\dim Z_1+\dim Z_2-\dim X$. 
			\item If $\dim W=\dim Z_1+\dim Z_2-\dim X$, then $(x,y_1,y_2)$ is smooth in $W$.
			\item For each $i$, $x$ is smooth in $(Z_i)_{y_i}$. Moreover, the intersection $T_x((Z_1)_{y_1})\cap T_x((Z_2)_{y_2})$ is non-trivial if and only if $(x,y_1,y_2)$ is a non-injective point of the projection $W\rightarrow Y_1\times Y_2$.
			\item If $\dim Z_1+\dim Z_2=\dim Y_1+\dim Y_2+\dim X$ and $(Z_1)_{y_1}\cap(Z_2)_{y_2}$ is finite, then $W\rightarrow Y_1\times Y_2$ is analytically open near $(x,y_1,y_2)$.
		\end{enumerate}
	\end{proposition}
	\begin{remark} Before proceeding with the proof, we make two quick comments:
		\begin{enumerate}
			\item Note that if $A_1=A_2=A$ and $y_1$ and $y_2$ are independent over $Ax$, the proposition is essentially trivial. The point is that we don't specify any such relationship between the $A_i$ and $y_i$.
			\item Note also that the hypotheses are satisfied by any pair of standard families of plane curves in $\mathcal M$; that is, let $A_1=A_2=\emptyset$, let $X=M^2$, and let each $Z_i$ be the graph of a standard family $\mathcal Z_i$. Then it follows by Lemma \ref{family ranks} that all of the conditions of the proposition are met (including the added dimension conditions of (2) and (4)). In this case $W$ is the graph of the family of intersections $\mathcal I_{\mathcal Z_1,\mathcal Z_2}$. 
		\end{enumerate}
	\end{remark}
	\begin{proof} 		
	\begin{enumerate}		
		\item First note that (1) is a local statement, so in proving it we may shrink each set as much as we want. 
		
		Now since the four given projections are all generically dominant, it follows that $x$ is generic in $X$ over each $A_i$, and each $y_i$ is generic in $Y_i$ over $A_i$. We conclude:
	
	\begin{claim} $x$ is smooth in $X$, each $y_i$ is smooth in $Y_i$, and each $(x,y_i)$ is smooth in $Z_i$.
	\end{claim}
	\begin{proof} By Lemma \ref{smooth locus definable} and the fact that each point is generic in the relevant set.
	\end{proof}

	By the Claim, we may replace each of $X$, $Y_1$, $Y_2$, $Z_1$, and $Z_2$ with its smooth locus, and thereby assume that each is a smooth variety -- in particular we can still assume after shrinking that each $Z_i\subset X\times Y_i$. In fact, shrinking further, we may assume that each of these five sets is irreducible. So we may view the projections $Z_i\rightarrow X$ and $Z_i\rightarrow Y_i$ as morphisms of smooth irreducible varieties. 
		
	Note, then, that each of these four morphisms is dominant, by the genericity of $x$ and $y_i$. So by generic smoothness (which holds in characteristic zero, see e.g. Theorem 25.3.1 of \cite{Vak}), we may assume after further shrinking that each $Z_i\rightarrow X$ and $Z_i\rightarrow Y_i$ is smooth. But $W$ is just the fiber product of the $Z_i\rightarrow X$; thus $W\rightarrow X$ is also smooth, which by the smoothness of $X$ implies that $W$ is smooth.
	
	Moreover, by smoothness, each fiber of $Z_i\rightarrow X$ is either empty or of dimension $\dim Z_i-\dim X$. Thus each fiber of $W\rightarrow X$ is either empty or of dimension $\dim Z_1+\dim Z_2-2\dim X$, which implies that $$\dim W\leq\dim X+(\dim Z_1+\dim Z_2-2\dim X)=\dim Z_1+\dim Z_2-\dim X.$$ On the other hand, the fiber $W_x$ is non-empty, as witnessed by $(x,y_1,y_2)$. Moreover, smooth morphisms are open -- so since $X$ is irreducible, the set of non-empty fibers is dense in $X$. This shows that in fact $\dim W=\dim Z_1+\dim Z_2-\dim X$, which is enough to prove (1).
	
	\item Immediate from (1).
	
	\item First let us retain the various shrinkings employed in proving (1); note that this is not a problem, since the statement of (3) is local. In particular, in what follows we assume that each of $X$, $Y_1$, $Y_2$, $Z_1$, and $Z_2$ is a smooth irreducible variety, and that each projection $Z_i\rightarrow X$ and $Z_i\rightarrow Y_i$ is smooth. By the proof of (1), we may further deduce that $W$ and $W\rightarrow X$ are smooth. 
	
	Throughout this proof, we view all tangent spaces as sub-vector spaces of the space $$T_{(x,y_1,y_2)}(X\times Y_1\times Y_2)=T_x(X)\times T_{y_1}(Y_1)\times T_{y_2}(Y_2).$$ Let $0_x$, $0_{y_1}$, and $0_{y_2}$ denote the zero vectors in these three factors, respectively. In reality $0_x$ is just $x$, and so on -- but of course we use different notation when viewing as a vector space.
	
	We now continue with the proof. The next two claims show that $x$ is smooth in each $(Z_i)_{y_i}$:
	
	\begin{claim}\label{smooth fiber dimension} For each $i$, $\dim((Z_i)_{y_i})=\dim Z_i-\dim Y_i$.
	\end{claim} 
	\begin{proof} This follows from the smoothness of $Z_i\rightarrow Y_i$ and the fact that $(Z_i)_{y_i}$ is non-empty (witnessed by $x$). 
	\end{proof}
	\begin{claim}\label{generic in family implies generic in fiber} For each $i$, $x$ is generic in $(Z_i)_{y_i}$ over $A_iy_i$. In particular, $x$ is smooth in $(Z_i)_{y_i}$.
	\end{claim}
	\begin{proof} By genericity, $\dim(y_i/A_i)=\dim Y_i$ and $\dim(xy_i/A_i)=\dim Z_i$. So by additivity, $\dim(x/A_iy_i)=\dim Z_i-\dim Y_i$. By Claim \ref{smooth fiber dimension}, this is enough.
	\end{proof}
	
	We now move toward the second clause of (3).
	
	\begin{lemma}\label{tangent space of fiber product} The tangent space $T_{(x,y_1,y_2)}(W)$ is precisely the fiber product (as vector spaces) of $T_{(x,y_1)}(Z_1)$ and $T_{(x,y_2)}(Z_2)$ over $T_x(X)$.
	\end{lemma}
	\begin{proof} For convenience let us denote the specified fiber product of tangent spaces as $V$. Now the maps $W\rightarrow Z_i$ induce an embedding $T_{(x,y_1,y_2)}(W)\rightarrow V$, via the inclusion map. On the other hand, since the maps $Z_i\rightarrow X$ are smooth, it follows that the maps $T_{(x,y_i)}(Z_i)\rightarrow T_x(X)$ are surjective. On then easily computes that $$\dim V=\dim Z_1+\dim Z_2-\dim X=\dim(T_{(x,y_1,y_2)}(W)),$$ which shows that the embedding above is actually an isomorphism.
	\end{proof} 

	Now by Lemma \ref{non-injective characterization}, the assertion that $(x,y_1,y_2)$ is a non-injective point of the projection $W\rightarrow Y_1\times Y_2$ is equivalent to the existence of a non-zero vector $v\in T_x(X)$ such that $(v,0_{y_1},0_{y_2})\in T_{(x,y_1,y_2)}(W)$. By Lemma \ref{tangent space of fiber product}, the assertion that $(v,0_{y_1},0_{y_2})\in T_{(x,y_1,y_2)}(W)$ is equivalent to the assertion that each $(v,0_{y_i})\in T_{(x,y_i)}(Z_i)$ -- or in other words, that each $(v,0_{y_i})$ belongs to the kernel of the map $T_{(x,y_i)}(Z_i)\rightarrow T_{y_i}(Y_i)$.  Finally, since $Z_i\rightarrow Y_i$ is smooth, an easy dimension comparison yields that the kernel of $T_{(x,y_i)}(Z_i)\rightarrow T_{y_i}(Y_i)$ is just $T_x((Z_i)_{y_i})\times\{0_{y_i}\}$. The second clause of (3) now follows.
	
	\item  Again, we retain the various shrinkings from (1). This is again not a problem, because (a) in the proof of (1) we did not alter the dimensions of $X$, $Y_i$, or $Z_i$; (b) shrinking clearly preserves the finiteness of $(Z_1)_{y_1}\cap(Z_2)_{y_2}$, and (c); the conclusion of (4) is a local statement. 
	
	Now after shrinking, by (1) we have $\dim W=\dim Z_1+\dim Z_2-\dim X$. Combined with the added assumption on the dimensions of $X$, $Y_i$, and $Z_i$, it follows that $\dim W=\dim(Y_1\times Y_2)$. Then since $(Z_1)_{y_1}\cap(Z_2)_{y_2}$ is finite, Lemma \ref{conditions for openness} applies to $W\rightarrow Y_1\times Y_2$ at $(x,y_1,y_2)$. 
	\end{enumerate}
	\end{proof}

	\section{Generic Non-transversalities and Ramification}
	
	In this section we introduce the notion of $\mathcal M$ `detecting generic non-transversalities,' and develop some associated machinery. This notion will serve as an organizing principle for the rest of the paper.
	
	\subsection{Generic Non-transversalities}
	
	We start with the basic definitions:
	
	\begin{definition}\label{generic non-transversality} By a \textit{generic non-transversality} in $\mathcal M$, we mean a five-tuple $(x,C_1,C_2,c_1,c_2)$ satisfying the following:
		\begin{enumerate}
			\item Each $C_i$ is a strongly minimal plane curve with canonical base $c_i$.
			\item Each $c_i$ is coherent over $\emptyset$.
			\item $x$ is generic in $M^2$ over $\emptyset$, and generic in each $C_i$ over $c_i$.
			\item The intersection of tangent spaces $T_x(C_1)\cap T_x(C_2)$ is non-trivial.
		\end{enumerate}
	\end{definition}

	\begin{definition}\label{detect generic non-transversality} We say that $\mathcal M$ \textit{detects generic non-transversalities} if for any generic non-transversality $(x,C_1,C_2,c_1,c_2)$, the canonical bases $c_1$ and $c_2$ are $\mathcal M$-dependent over $\emptyset$.
	\end{definition}

	Definition \ref{detect generic non-transversality} is our attempt to give the weakest possible relevant meaning to the notion of `defining tangency'. Critically, note that we only consider tangency at points which are generic in each curve, and also generic over $\emptyset$. This restriction will make the necessary differential geometric reasoning much easier. Note also that we only consider plane curves with coherent canonical bases; this is because we only want to deal with generic curves in families.
	
	\subsection{Equivalent Conditions}
	
	It is convenient to work with canonical bases in the definition of generic non-transversalities. However, in proving that $\mathcal M$-detects generic non-transversalities it will be easier to work with standard families of plane curves. In this section we point out that these two approaches are equivalent. As a consequence, we deduce that generic non-transversalities correspond to ramification points (or rather, non-injective points), of certain maps. Later on, this observation will be key to detecting generic non-transversalities in $\mathcal M$.
	
	\begin{definition}\label{FG tangency} Let $\mathcal C=\{C_t:t\in T\}$ and $\mathcal D=\{D_u:u\in U\}$ be standard families of plane curves, and let $I\subset M^2\times T\times U$ be the graph of the family of intersections $\mathcal I_{\mathcal C,\mathcal D}$. Let $t$ be generic in $T$, and let $u$ be generic in $U$. Let $x$ be generic in each of $M^2$ over $\emptyset$, $C_t$ over $t$, and $D_u$ over $u$. 
		\begin{enumerate} 
			\item We say that $(x,t,u)$ is a \textit{generic} $(\mathcal C,\mathcal D)$-\textit{non-transversality} if the intersection of tangent spaces $T_x(C_t)\cap T_x(D_u)$ is non-trivial. 
			\item We say that $(x,t,u)$ is a \textit{generic} $(\mathcal C,\mathcal D)$-\textit{non-injectivity} if $(x,t,u)$ is non-injective in the projection $I\rightarrow T\times U$.
	\end{enumerate}
	\end{definition}

	\begin{remark}\label{generic family tangency iff ramification}
		It follows from Proposition \ref{smooth fiber product}(3) that $(x,t,u)$ is a generic $(\mathcal C,\mathcal D)$-non-transversality if and only if it is a generic $(\mathcal C,\mathcal D)$-non-injectivity. We distinguish the two equivalent notions for ease of presentation later on.
	\end{remark}

	Before proceeding, we point out the following easy lemma:
	
	\begin{lemma}\label{tangent space of component} Let $C$ be a plane curve which is $\mathcal M$-definable over a set $A$, and let $x\in C$ be generic over $A$. Then there is a unique stationary component $S$ of $C$, with canonical base $s$, such that $x$ is generic in $S$ over $As$. Moreover, for this $S$ we have $T_x(S)=T_x(C)$.
	\end{lemma}

	\begin{proof} Since $x$ is generic in $C$ over $A$, it is also $\mathcal M$-generic in $C$ over $A$. So $x$ realizes one of the finitely many stationary $\mathcal M$-generic types of $C$ over $\operatorname{acl}_{\mathcal M}(A)$. In other words, there is a unique stationary component $S$ of $C$, with canonical base $s$, such that $x$ is $\mathcal M$-generic in $S$ over $As$. On the other hand, by Lemma \ref{coherent iff generic} $x$ is coherent over $A$, and by Fact \ref{cb facts} $s$ is $\mathcal M$-algebraic over $A$; so by Lemma \ref{coherent preservation}, $x$ is coherent over $As$. Then $x$ is in fact generic in $S$ over $As$, as desired.
		
	Finally, note by genericity that the tangent spaces $T_x(S)$ and $T_x(C)$ are both $n$-dimensional vector subspaces of $T_x(M^2)$. But clearly $T_x(S)\subset T_x(C)$, so these spaces are equal.
	\end{proof}

	We now deduce the following, which is the main goal of this subsection.
	
	\begin{lemma}\label{detecting tangency equivalence}
		The following are equivalent:
		\begin{enumerate}
			\item $\mathcal M$ detects generic non-transversalities.
			\item Whenever $\mathcal C=\{C_t:t\in T\}$ and $\mathcal D=\{D_u:u\in U\}$ are standard families of plane curves, and $(x,t,u)$ is a generic $(\mathcal C,\mathcal D)$-non-transversality, the parameters $t$ and $u$ are $\mathcal M$-dependent over $\emptyset$.
			\item Whenever $\mathcal C=\{C_t:t\in T\}$ and $\mathcal D=\{D_u:u\in U\}$ are standard families of plane curves, and $(x,t,u)$ is a generic $(\mathcal C,\mathcal D)$-non-injectivity, the parameters $t$ and $u$ are $\mathcal M$-dependent over $\emptyset$. 
		\end{enumerate}
	\end{lemma}
	\begin{proof} As in Remark \ref{generic family tangency iff ramification}, the equivalence of (2) and (3) follows from Proposition \ref{smooth fiber product}(3). We will prove the equivalence of (1) and (2).
		
	First assume (1). Let $\mathcal C$, $\mathcal D$, and $(x,t,u)$ be as in (2). Let $S_1$ and $s_1$ be as in Lemma \ref{tangent space of component} for $x$ and $C_t$ over $t$, and similarly let $S_2$ and $s_2$ be as in Lemma \ref{tangent space of component} for $x$ and $D_u$ over $u$. Note by Corollary \ref{coherent family ranks} that $x$ is generic in $M^2$. Moreover, note by Lemma \ref{almost faithful facts} that $s_1$ and $s_2$ are $\mathcal M$-interalgebraic with $t$ and $u$, respectively; then by Lemmas \ref{coherent iff generic} and \ref{coherent preservation}, and the genericity of $t$ and $u$ in $T$ and $U$, it follows that each $s_i$ is coherent over $\emptyset$. So by Lemma \ref{tangent space of component}, we conclude that $(x,S_1,S_2,s_1,s_2)$ is a generic non-transversality. By assumption it follows that $s_1$ and $s_2$ are $\mathcal M$-dependent over $\emptyset$. But then by $\mathcal M$-interalgebraicity, so are $t$ and $u$. Thus we have proven (2).
		
	Now assume (2), and let $(x,S_1,S_2,s_1,s_2)$ be a generic non-transversality. Note that each $s_i$ has rank at least 1; indeed otherwise $x$ would not be generic in $M^2$ over $\emptyset$, as witnessed by $x\in S_i$. Now let $\mathcal C=\{C_t:t\in T\}$ and $t$ be as in Lemma \ref{coherent version of almost faithful family existence} for $S_1$ and $s_1$ over $\emptyset$, and let $\mathcal D=\{D_u:u\in U\}$ and $u$ be as in Lemma \ref{coherent version of almost faithful family existence} for $S_2$ and $s_2$ over $\emptyset$. Then since $t$ and $d_1$ are $\mathcal M$-interalgebraic, we have $\dim(x/s_1t)=\dim(x/s_1)=\dim M$; thus $x$ is generic in $S_1$ over $s_1t$, and so $x$ is also generic in $C_t$ over $s_1t$. By similar reasoning, $x$ is also generic in $D_u$ over $s_2u$. It now follows easily that $(x,t,u)$ is a generic $(\mathcal C,\mathcal D)$-non-transversality, so that by assumption $t$ and $u$ are $\mathcal M$-dependent over $\emptyset$. Then by $\mathcal M$-interalgebraicity again, so are $s_1$ and $s_2$.
	\end{proof}

	\section{Reducing to Curves when $\mathcal M$ Detects Generic Non-transversalities}
	
	In this section we prove the first main result of the paper -- that assuming $\mathcal M$ detects generic non-transversalities we have $\dim M=1$. Our main tool will be the \textit{purity of the ramification locus} (\cite{Zar}). Among other formulations, the following will be sufficient:
	
	\begin{fact}[Purity of the Ramification Locus]\label{purity} Let $V$ and $W$ be smooth equidimensional varieties over an algebraically closed field. Let $f:V\rightarrow W$ be a morphism, and let $R\subset V$ be the set of points at which $f$ is ramified. Then $R$ is a relatively closed subvariety of $V$, and each irreducible component of $R$ has codimension at most 1 in $V$.
	\end{fact} 

	Roughly speaking, our strategy will be to use Fact \ref{purity} to show that generic non-transversalities, when they exist, occur in codimension 1. The result will then follow from Corollary \ref{dim rk inequality}. The main difficulty we encounter will be ensuring that the non-transversalities we find come from generic points of families. Thus we will need a quick argument using the Compactness Theorem. Otherwise, the proof will be quite straightforward.
	
	\subsection{Producing Tangencies}
	
	Let us proceed with the argument. Throughout this section, we fix an excellent family $\mathcal C=\{C_t:t\in T\}$ of plane curves, as provided by Assumption \ref{M and K}. We let $C\subset M^2\times T$ be the graph of $\mathcal C$, and $I\subset M^2\times T^2$ the graph of the family of intersections $\mathcal I_{\mathcal C,\mathcal C}$; thus by Lemma \ref{family ranks} we have $\rk(C)=3$ and $\rk(I)=4$. Now our goal in this first subsection is roughly to produce generic non-transversalities in codimension 1 in $I$. To do this, we first show:
	
	\begin{lemma}\label{M' noninjective} Let $C'\subset C$ be any $\mathcal K$-definable set which is large in $C$. Then there is an element $\hat w=(\hat x,\hat t,\hat{t'})\in I$ such that:
		\begin{enumerate}
			\item Each of $(\hat x,\hat t)$ and $(\hat x,\hat{t'})$ belongs to $C'$. 
			\item $\dim(\hat w/\emptyset)\geq 4\cdot\dim M-1$.
			\item $\hat w$ is non-injective in the projection $I\rightarrow T^2$.
		\end{enumerate}
	\end{lemma}
	\begin{proof} Let us denote $m=\dim M$. First, shrinking $C'$ if necessary, we may assume it is relatively open in $C$ and is a smooth variety of dimension $3m$. Now let $(x,t)\in C'$ be generic over any set of parameters defining $C'$ in $\mathcal K$. Then by Lemmas \ref{family ranks} and \ref{smooth locus definable}, $t$ is smooth in $T$; and by Proposition \ref{smooth fiber product}(1), $(x,t,t)$ is smooth in $I$. Let $T'\subset T$ and $I'\subset I$ witness these two instances of smoothness, as in Definition \ref{smooth point}. Shrinking if necessary, we may assume that the projection of $C'$ to $T$ is contained in $T'$, and moreover that each of the projections $I'\rightarrow C$ is contained in $C'$.
		
	Thus the projection $I\rightarrow T^2$ restricts to a morphism $I'\rightarrow(T')^2$ of smooth $4m$-dimensional varieties. Note that $(x,t,t)$ is clearly a generic $(\mathcal C,\mathcal C)$-non-transversality, since $T_x(C_t)\cap T_x(C_t)=T_x(C_t)$ is non-trivial. So by Proposition \ref{smooth fiber product}(3) and Lemma \ref{non-injective characterization}, $I'\rightarrow(C')^2$ ramifies at $(x,t,t)$. Let $R$ be the ramification locus of $I'\rightarrow(C')^2$; so we have observed that $R$ is non-empty. By Fact \ref{purity}, the purity of the ramification locus (\cite{Zar}), it then follows that $\dim R\geq 4m-1$ (either because $\dim R=4m$ or because $I'\rightarrow(C')^2$ is generically unramified but not unramified). 
	
	Now let $\hat w=(\hat x,\hat t,\hat{t'})$ be generic in $R$ over all relevant parameters. By Lemma \ref{non-injective characterization}, $\hat w$ is non-injective. By the preceding remarks about $R$, $\dim(\hat w/\emptyset)\geq 4m-1$. And finally, by construction of $I'$ we have $(\hat x,\hat t),(\hat x,\hat{t'})\in C'$. 
	\end{proof}

	We next conclude the following stronger statement, which is the main goal of this subsection:
	
	\begin{corollary}\label{non-injective points exist} There is a generic $(\mathcal C,\mathcal C)$-non-injectivity $\hat w=(\hat x,\hat t,\hat{t'})$ satisfying $\dim(\hat w)\geq 4\cdot\dim M-1$.
	\end{corollary}
	\begin{proof}
		Equivalently, we want to find an element $\hat w=(\hat x,\hat t,\hat{t'})\in I$ satisfying each of the following:
		\begin{enumerate}
			\item $(\hat x,\hat t)$ and $(\hat x,\hat{t'})$ are both generic in $C$.
			\item $\dim(\hat w)\geq 4\cdot\dim M-1$.
			\item $\hat w$ is non-injective in the projection $I\rightarrow T^2$.
		\end{enumerate}
		Now the main observation to make is that conditions (1) through (3) can be expressed by the conjunction of infinitely many formulas over $\emptyset$ in the language of $\mathcal K$. Let us begin by justifying this assertion. First note that (3) is expressible with a single formula, using (1) of Lemma \ref{non-injective characterization}. We thus focus on (1) and (2). But these two clauses can be separated into the assertions $\dim(\hat w/\emptyset)\geq 4\cdot\dim M-1$, $\dim(\hat x\hat t/\emptyset)\geq 3\cdot\dim M$, and $\dim(\hat x\hat{t'}/\emptyset)\geq 3\cdot\dim M$. So it suffices to consider assertions of the form $\dim(z)\geq k$ for a tuple $z$ and an integer $k$. But this is equivalent to negating each formula over $\emptyset$ in the arity of $z$ which has dimension $<k$ -- so we just take the set of negations of formulas of this type.
		
		Now let $p(w)$ be the set of infinitely many formulas given in the previous paragraph. Since $\mathcal K$ is saturated, the assertion that $p$ is realized in $\mathcal K$ is equivalent to the assertion that $p$ is consistent. By the Compactness Theorem, this is further equivalent to the assertion that $p$ is finitely satisfiable.
		
		So, it suffices to realize any finite subset of $p$. In particular, we need not insist that each $(\hat x,\hat t)$ and $(\hat x,\hat{t'})$ be generic in $C$, as for any fixed finite subset $q\subset p$ there is a single non-generic subset $Z\subset C$ such that $q$ only insists that $(\hat x,\hat t)$ and $(\hat x,\hat{t'})$ avoid $Z$. It then suffices to choose $\hat w$ satisfying Proposition \ref{M' noninjective} with $C'=C-Z$.
	\end{proof}
	
	\subsection{Reducing to Curves}
	
	We end this section by proving the first main result of the paper:
	
	\begin{theorem}\label{n=1} If $\mathcal M$ detects generic non-transversalities, then $\dim M=1$.
	\end{theorem}
	\begin{proof} Let $\hat w=(\hat x,\hat t,\hat{t'})$ be as in Corollary \ref{non-injective points exist}. By Lemma \ref{detecting tangency equivalence} and the assumption that $\mathcal M$ detects generic non-transversalities, we conclude that $\hat t$ and $\hat{t'}$ are $\mathcal M$-dependent over $\emptyset$. Then we note:
	\begin{claim} $\rk(\hat w)\leq 3$.
	\end{claim}
	\begin{proof} From above we have $\rk(\hat t\hat{t'})\leq 3$. In particular, if $\rk(\hat x/\hat t\hat{t'})=0$ then we are done. So, we assume that $\rk(\hat x/\hat t\hat{t'})\geq 1$. But in this case $C_{\hat t}\cap C_{\hat{t'}}$ is infinite, so by almost faithfulness $\hat t$ and $\hat{t'}$ are $\mathcal M$-interalgebraic. So we get that $\rk(\hat w)=\rk(\hat x\hat t)=3$, as desired.
	\end{proof}
	Now as in Lemma \ref{M' noninjective}, let us denote $m=\dim M$. Then by the claim and Corollary \ref{dim rk inequality}, we have $\dim(\hat w)\leq 3m$. On the other hand, by construction we also have $\dim(\hat w)\geq 4m-1$. Thus $4m-1\leq 3m$, which when rearranged gives $m\leq 1$. That $m=1$ now follows since $M$ is infinite.
	\end{proof}
	
	\section{Interpreting a Field when $\mathcal M$ Detects Generic Non-transversalities}
	
	In this section we prove the second main result of the paper -- that if $\mathcal M$ detects generic non-transversalities (and thus $\dim M=1$ by Theorem \ref{n=1}), $\mathcal M$ interprets an algebraically closed field. Our strategy is similar in nature to past papers on trichotomy problems, in that we (roughly) use the detection of tangency to interpret the set of `slopes' of plane curves through a point, and recover addition and multiplication of slopes via abstract sum and composition operations on curves. The main differences are that (1) we work with slopes at generic points rather than diagonal points, making the recovery of the field operations less straightfoward, and (2) we interpret the field `in one step' (this is also the case in \cite{Rab}), rather than first interpreting a group. 
	
	Before giving the argument, let us give a more precise summary of the main steps:
	
	\begin{enumerate}
		\item In the first subsection we introduce the slope of a plane curve $C$ at a generic point $(x,y)$, relative to a pair of identifications of $T_x(M)$ and $T_y(M)$ with $\mathbb C$.
		\item We then introduce the operations of composition and formal sum on plane curves, and show that they induce multiplication and addition of slopes at generic points.
		\item Next we generate and define the `standard representatives' of a generic slope at a generic point -- certain curves with nice properties that attain the given slope at the given point. These will essentially come from a fixed excellent family $\mathcal C$.
		\item Assuming that $\mathcal M$ detects generic non-transversalities, we then use (2) to show that $\mathcal M$ detects `generic multiplication and addition' -- namely, given standard representatives of generic independent slopes $\alpha_1$ and $\alpha_2$ through certain pairs of distinct generic points, one can recover the standard representatives of $\alpha_1\alpha_2$ or $\alpha_1+\alpha_2$ through a certain third point. Here it is important that we use three distinct points, in order to only use generic non-transversalities.
		\item We next prove a strengthening of (4), in which all relevant slopes occur at a \textit{single} generic point $(x_0,y_0)$. The trick is to break the desired operations at $(x_0,y_0)$ into a sequence of iterations of (4) where we move to different points, making sure that in the end we return to the original point. 
		\item Finally, using (5) we show that the action on $\mathbb C$ of the group $\textrm{AGL}_1(\mathbb C)$ can roughly be `encoded' into $\mathcal M$, using standard representatives at $(x_0,y_0)$ in place of field elements. By a result of Hrushovski, it is then straightforward to show that $\mathcal M$ interprets an algebraically closed field.
	\end{enumerate}
	
	\subsection{Slopes}
	
	\begin{assumption}
		Throughout this section, we assume that $\mathcal M$ detects generic non-transversalities. By Theorem \ref{n=1}, we moreover have $\dim M=1$.
	\end{assumption}

	\begin{remark} In this section, we will at times use parameters from $\mathbb C$ (i.e. not strictly from $\mathcal M^{\textrm{eq}}$ -- see for example Remark \ref{fixed 5 points remark}(4)). We point out that we will only do this while working in the structure $(\mathbb C,+,\cdot)$, and will always restrict to parameters from $\mathcal M^{\textrm{eq}}$ when working with $\mathcal M$.
	\end{remark}
	
	Before introducing slopes of plane curves, we need to develop a couple related notions:
	
	\begin{definition}\label{infinitesimal coordinate system}
		Let $x\in M$ be generic over $\emptyset$. By an \textit{infinitesimal coordinate system} at $x$, we mean a $\mathbb C$-linear isomorphism $f:T_x(M)\rightarrow\mathbb C$.
	\end{definition}

	Note that infinitesimal coordinate systems always exist at generic points; moreover, each one is $\mathcal K$-definable (possibly over extra parameters). We will often work relative to fixed infinitesimal coordinate systems at finitely many specified points. 
	
	\begin{definition}\label{non-vertical} Let $C\subset M^2$ be a plane curve. We say that $C$ is \textit{non-vertical} if the projection $C\rightarrow M$ to the left coordinate is finite-to-one.
	\end{definition}

	Non-vertical plane curves provide a natural setting for slopes. The reason is the following:
	
	\begin{lemma}\label{slope finite} Let $C$ be a non-vertical plane curve, $\mathcal M$-definable over a set $A$. Let $(x,y)$ be generic in $C$ over $A$. Then $T_x(M)$ is well-defined, and the differential map $T_{(x,y)}(C)\rightarrow T_x(M)$ is an isomorphism.
	\end{lemma}
	\begin{proof} Using that the projection $C\rightarrow M$ to the left coordinate is finite-to-one, it is easy to see that $x$ is generic in $M$ over $A$, which shows that $T_x(M)$ is well-defined. Moreover, for the same reason it is easy to see that $C\rightarrow M$ is, after restricting to relative neighborhoods of $(x,y)$ and $x$, a dominant quasi-finite morphism of smooth varieties. Since we are in characteristic zero, such maps are generically \'{e}tale; in particular, since $(x,y)$ is generic in $C$, it follows that $C\rightarrow M$ is \'{e}tale at $(x,y)$, which is enough.
	\end{proof}

	We can now discuss slopes. Suppose that $C\subset M^2$ is a non-vertical plane curve, $\mathcal M$-definable over a set $A$. Let $(x,y)\in C$ be generic over $A$, and assume that each of $x$ and $y$ is generic in $M$ over $\emptyset$. Then we can view $T_{(x,y)}(C)$ as a sub-vector space of $T_x(M)\times T_y(M)$. By Lemma \ref{slope finite}, $T_{(x,y)}(C)\rightarrow T_x(M)$ is an isomorphism; thus $T_{(x,y)}(C)$ is in fact the graph of a $\mathbb C$-linear map $T_x(M)\rightarrow T_y(M)$. Now if we have fixed infinitesimal coordinate systems $f:T_x(M)\rightarrow\mathbb C$ and $g:T_y(M)\rightarrow\mathbb C$, then the image of $T_{(x,y)}(C)$ in $\mathbb C^2$ (via $f$ and $g$) is the graph of a $\mathbb C$-linear endomorphism of $\mathbb C$. In particular, this endomorphism is given by $z\mapsto \alpha z$ for some unique $\alpha\in\mathbb C$.
	
	\begin{definition}\label{slope} Suppose that $C\subset M^2$ is a non-vertical plane curve, $\mathcal M$-definable over a set $A$, and let $(x,y)\in C$ be generic over $A$. Assume that each of $x$ and $y$ is generic in $M$ over $\emptyset$, and fix infinitesimal coordinate systems $f:T_x(M)\rightarrow\mathbb C$ and $g:T_y(M)\rightarrow\mathbb C$. Then by the \textit{slope} of $C$ at $(x,y)$ relative to $f$ and $g$, we mean the unique $\alpha\in\mathbb C$ such that the image of $T_{(x,y)}(C)$ in $\mathbb C^2$, via $f$ and $g$, is the graph of scaling by $\alpha$.
	\end{definition}

	\begin{notation} If the infinitesimal coordinate systems $f$ and $g$ are clear from context, we will abuse notation and simply call $\alpha$ from Definition \ref{slope} the \textit{slope} of $C$ at $(x,y)$, denoted $\tau_{(x,y)}(C)$.
	\end{notation}
	
	\begin{remark}\label{slope remark} Before moving on, let us make some comments:
	\begin{enumerate}
		\item Suppose that $C$ is non-trivial, and $(x,y)\in C$ is generic. Then the slope of $C$ at $(x,y)$ (with respect to any infitesimal coordinate systems) is nonzero. This is because the argument of Lemma \ref{slope finite} applies to the second projection $C\rightarrow M$ as well, which shows that $T_{(x,y)}(C)$ is in fact the graph of an isomorphism $T_x(M)\rightarrow T_y(M)$.
		\item Note that slopes are relatively $\mathcal K$-definable in families, by Lemma \ref{tangent space definable} (the term `relatively' is because we are restricting to generic points). In particular, the slope of $C$ at $(x,y)$ relative to $f$ and $g$ is $\mathcal K$-definable over any set of parameters defining $C$, $(x,y)$, $f$, and $g$.
		\item Note also that we only assume $(x,y)$ is generic over the parameters defining $C$ (not $f$ and $g$). Thus we can change the infinitesimal coordinate systems at $x$ and $y$ without losing our genericity assumptions.
		\item Finally, let us point out that slopes can also be interpreted analytically. That is, suppose that $C$, $A$, $(x,y)$, $f$, and $g$ are as in Definition \ref{slope}. Then by the Implicit Function Theorem, there are analytic neighborhoods $U$ and $V$ of $x$ and $y$ such that the restriction of $C$ to $U\times V$ is the graph of a holomorphic function $h:U\rightarrow V$. Then the map $T_x(M)\rightarrow T_y(M)$ induced by $T_{(x,y)}(C)$ is simply the differential of $h$. In particular, the slope is the unique $\alpha$ such that the image of the differential map of $h$ in $\mathbb C^2$ is given by scaling by $\alpha$. 
		\end{enumerate}
	\end{remark}

	\subsection{Operations on Curves and Slopes}
	
	In this section we introduce the composition and formal sum operations on plane curves, and show that they induce multiplication and addition of slopes. We start with composition because it is easier.
	
	\begin{definition} Let $C_1$ and $C_2$ be non-trivial plane curves. Then by $C_2\circ C_1$ we mean the set of all $(x,z)\in M^2$ such that for some $y\in M$ we have $(x,y)\in C_1$ and $(y,z)\in C_2$.
	\end{definition}
	
	It is easy to see that the composition of two non-trivial plane curves is again a non-trivial plane curve (while this may not hold for non-vertical plane curves). Moreover, the composition is $\mathcal M$-definable over any set which defines each of the two curves. Now we check:
	
	\begin{lemma}\label{slope multiplication} Suppose that $C_1$ and $C_2$ are non-trivial plane curves, each $\mathcal M$-definable over a set $A$. Let $x$, $y$, and $z$ be generic in $M$ over $\emptyset$, each with a fixed infinitesimal coordinate system. Assume further that $(x,y)$ is generic in $C_1$ over $A$, and $(y,z)$ is generic in $C_2$ over $A$. Then $(x,z)$ is generic in $C_2\circ C_1$ over $A$, and $\tau_{(x,z)}(C_2\circ C_1)=\tau_{(x,y)}(C_1)\cdot\tau_{(y,z)}(C_2)$.
	\end{lemma} 
	\begin{proof} First note by definition that $(x,z)\in C_2\circ C_1$, and that $C_2\circ C_1$ is $\mathcal M$-definable over $A$. Now it follows from the non-triviality of each $C_i$ that $x$, $y$, and $z$ are pairwise $\mathcal M$-interalgebraic over $A$. Thus $\dim(x,y,z/A)=\dim(x,y/A)=1$, which shows that $(x,z)$ is generic in $C_2\circ C_1$ over $A$.
		
	Now let $h_1$ and $h_2$ be holomorphic functions as in Remark \ref{slope remark}(4), giving the points of $C_1$ and $C_2$ near $(x,y)$ and $(y,z)$, respectively. We may assume after shrinking that $\operatorname{Im}(h_1)\subset\operatorname{dom}(h_2)$. Then $h_2\circ h_1$ is a holomorphic function giving points of $C_2\circ C_1$ in a neighborhood of $(x,z)$. Since $(x,z)$ is generic in $C_2\circ C_1$, there is only one such function up to agreement on a neighborhood of $x$. So, as in Remark \ref{slope remark}(4), it follows that $T_{(x,z)}(C_2\circ C_1)$ is the graph of the differential of $h_2\circ h_1$ at $x$. By the chain rule, this is the composition of the differentials of $h_2$ and $h_1$ at $(y,z)$ and $(x,y)$. So, after applying our infinitesimal coordinate systems, we get that scaling by $\tau_{(x,z)}(C_2\circ C_1)$ is the composition of the scalings by $\tau_{(y,z)}(C_2)$ and $\tau_{(x,y)}(C_1)$. This implies the second clause of the lemma.
	\end{proof}	

	We now turn toward formal sums and addition. The following fact is not necessary, but makes the argument run more smoothly:
	
	\begin{fact}\label{gamma exists} There is a set $\Gamma\subset M^3$ which is $\mathcal M$-definable over $\emptyset$ and of rank 2, such that each projection $\Gamma\rightarrow M^2$ is finite-to-one and surjective.
	\end{fact}
	\begin{proof} We only sketch the proof, as it is straightforward. The idea is that a rank 1 standard family of non-trivial plane curves satisfies almost all of these properties -- the only exception being that the projections to $M^2$ are \textit{almost} surjective. One then needs to cover any remaining points in each projection without changing that each projection is finite-to-one. It is easy to see that this is possible. 
		
	The only thing we have not addressed is that such families exist, or in other words that $\Gamma$ can be taken to be $\mathcal M$-definable over $\emptyset$. In fact we could just assume this by adding constants to the language of $\mathcal M$, and the proof would be unhindered. If the reader wants to be picky, however, it is still possible to construct such a $\Gamma$ using our preset assumption (i.e. Assumption \ref{M and K}) that the language of $\mathcal M$ has infinitely many distinct constants. Indeed, one can start with an excellent family $\mathcal C$ which is $\mathcal M$-definable over $\emptyset$, and then choose a rank 1 subfamily of $\mathcal C$ which is defined over an appropriate constant.
	\end{proof}

	For the rest of this section, we fix a set $\Gamma$ as in Fact \ref{gamma exists}. The reader should think of $\Gamma$ as playing the abstract role of a group operation.

	\begin{definition}\label{formal sum} Let $C_1$ and $C_2$ be non-vertical plane curves. Then by the \textit{formal sum} of $C_1$ and $C_2$, denoted $C_1\oplus C_2$, we mean the set of all $(x,z)\in M^2$ such that for some $u,v\in M$ we have $(x,u)\in C_1$, $(x,v)\in C_2$, and $(u,v,z)\in\Gamma$.
	\end{definition}
	\begin{remark} For intuition, we give the following example. Suppose that each $C_i$ is the graph of a function $f_i:M\rightarrow M$, and $\Gamma$ is the graph of a group operation $\oplus$ on $M$. Then $C_1\oplus C_2$ is the graph of the point-wise sum of the $f_i$, i.e. $x\mapsto f_1(x)\oplus f_2(x)$.
	\end{remark}

	Using that each projection $\Gamma\rightarrow M^2$ is finite-to-one and surjective, it is easy to see that the formal sum of two non-vertical plane curves is again a non-vertical plane curve (while this may not hold for non-trivial plane curves). Moreover, the formal sum is $\mathcal M$-definable over any set which defines the two curves. 
	
	Now we would like to prove the literal analog of Lemma \ref{slope multiplication} for formal sums and addition of slopes; but in this case the calculus is more complicated, and we need to be a bit careful. Namely, because we are working with $\Gamma\subset M^3$, we need to first discuss \textit{partial slopes}.
	
	\begin{notation}\label{gamma fibers} Let $u_0,v_0\in M$.
		\begin{enumerate}
			\item We denote the set $\{(v,z):(u_0,v,z)\in\Gamma\}$ by $\Gamma_1(u_0)$.
			\item We denote the set $\{(u,z):(u,v_0,z)\in\Gamma\}$ by $\Gamma_2(v_0)$.
		\end{enumerate}
	\end{notation}

	\begin{remark}\label{gamma remark} The following are all immediate from the choice of $\Gamma$:
		\begin{enumerate}
			\item Each $\Gamma_1(u_0)$ and $\Gamma_2(v_0)$ is a non-trivial plane curve.
			\item If $(u_0,v_0,z_0)\in\Gamma$ is generic over a set $A$ then each of $u_0$, $v_0$, and $z_0$ is generic in $M$ over $A$.
			\item If $(u_0,v_0,z_0)\in\Gamma$ is generic over a set $A$ then $(v_0,z_0)$ is generic in $\Gamma_1(u_0)$ over $Au_0$, and $(u_0,z_0)$ is generic in $\Gamma_2(v_0)$ over $Av_0$.
		\end{enumerate}
	\end{remark}

	By Remark \ref{gamma remark}, the following makes sense:
	
	\begin{definition}\label{partial slopes} Let $(u_0,v_0,z_0)\in\Gamma$ be generic over $\emptyset$, and assume we have fixed infinitesimal coordinate systems at each of $u_0$, $v_0$, and $z_0$. Then we define the \textit{partial slopes} of $\Gamma$ at $(u_0,v_0,z_0)$ (relative to our infinitesimal coordinate systems) to be $\Gamma_u(u_0,v_0,z_0)=\tau_{(u_0,z_0)}(\Gamma_2(v_0))$ and $\Gamma_v(u_0,v_0,z_0)=\tau_{(v_0,z_0)}(\Gamma_1(u_0))$.
	\end{definition}
	
	\begin{remark}\label{partial slope holomorphic} As in Remark \ref{slope remark}(4), we can interpret partial slopes analytically. Namely, suppose $(u_0,v_0,z_0)\in\Gamma$ is generic. Then as in Lemma \ref{slope finite}, the projection $\Gamma\rightarrow M^2$ to the left two coordinates is an \'{e}tale morphism in a neighborhood of $(u_0,v_0,z_0)$; so by the Implicit Function Theorem, we can find an analytic neighborhood $U\times V\times Z$ of $(u_0,v_0,z_0)$ such that the restriction of $\Gamma$ to $U\times V\times Z$ is the graph of a holomorphic function $h:U\times V\rightarrow Z$. Then the partial slopes of $\Gamma$ at $(u_0,v_0,z_0)$ are the images in $\mathbb C^2$ of the partial differentials of $h$.
	\end{remark}

	We need one more concept before discussing addition:
	
	\begin{definition}\label{normal system} Let $(u_0,v_0,z_0)\in\Gamma$ be generic, and fix infinitesimal coordinate systems $f_1,f_2,f_3$ at $u_0$, $v_0$, and $z_0$. We say that $(u_0,v_0,z_0)$ is \textit{normal} with respect to the $f_i$, or that the $f_i$ \textit{normalize} $(u_0,v_0,z_0)$, if both partial slopes of $\Gamma$ at $(u_0,v_0,z_0)$ are 1.
	\end{definition}

	We note:
	
	\begin{lemma}\label{normalization exists} Let $(u_0,v_0,z_0)\in\Gamma$ be generic. Then there are infinitesimal coordinate systems at $u_0$, $v_0$, and $z_0$ which normalize $(u_0,v_0,z_0)$.
	\end{lemma}
	\begin{proof} First fix any infinitesimal coordinate systems at each point, say $f_1,f_2,f_3$. Now as in Remark \ref{gamma remark}, $\Gamma_1(u_0)$ and $\Gamma_2(v_0)$ are non-trivial; so by Remark \ref{slope remark}(1), the partial slopes with respect to the $f_i$ are nonzero. Now we can rescale $f_1$ and $f_2$ independently until each partial slope is 1.
	\end{proof}
	
	We are now ready to give our analog of Lemma \ref{slope multiplication}:
	
	\begin{lemma}\label{slope addition} Suppose that $C_1$ and $C_2$ are non-vertical plane curves, each $\mathcal M$-definable over a set $A$. Let $x_0\in M$ and $(u_0,v_0,z_0)\in\Gamma$ each be generic over $A$. Fix infinitesimal coordinate systems at each of $x_0,u_0,v_0,z_0$, and assume that they normalize $(u_0,v_0,z_0)\in\Gamma$. Assume further that $(x_0,u_0)$ is generic in $C_1$ over $A$, and $(x_0,v_0)$ is generic in $C_2$ over $A$. Then $(x_0,z_0)$ is generic in $C_1\oplus C_2$ over $A$, and $\tau_{(x_0,z_0)}(C_1\oplus C_2)=\tau_{(x_0,u_0)}(C_1)+\tau_{(x_0,v_0)}(C_2)$.
	\end{lemma}
	\begin{proof}
		The proof is very similar to that of Lemma \ref{slope multiplication}. First note by definition that $(x_0,z_0)\in C_1\oplus C_2$, and that $C_1\oplus C_2$ is $\mathcal M$-definable over $A$. Moreover, since the projections $C_2\rightarrow M$ and $\Gamma\rightarrow M^2$ to the leftmost coordinates are finite-to-one, it follows that $(x_0,u_0,v_0,z_0)$ is $\mathcal M$-algebraic over $Ax_0u_0$. Thus $\dim(x_0,u_0,v_0,z_0/A)=\dim(x_0,u_0/A)=1$, which shows that $(x_0,z_0)$ is generic in $C_1\oplus C_2$ over $A$.

		Now let $h_1$ and $h_2$ be holomorphic functions as in Remark \ref{slope remark}(4), giving the points of $C_1$ and $C_2$ near $(x_0,u_0)$ and $(x_0,v_0)$, respectively. Moreover, let $h$ be a two-variable holomorphic function as in Remark \ref{partial slope holomorphic}, giving the points of $\Gamma$ near $(u_0,v_0,z_0)$. Then arguing as in Lemma \ref{slope multiplication}, the points of $C_1\oplus C_2$ near $(x_0,z_0)$ are given by the function $h(h_1(x),h_2(x))$. Continuing to argue as in Lemma \ref{slope multiplication} but using the higher dimensional chain rule on $h(h_1(x),h_2(x))$, one then calculates that $$\tau_{(x_0,z_0)}(C_1\oplus C_2)=\tau_{(x_0,u_0)}(C_1)\cdot\Gamma_u(u_0,v_0,z_0)+\tau_{(x_0,v_0)}(C_2)\cdot\Gamma_v(u_0,v_0,z_0).$$ Since our infinitesimal coordinate systems normalize $(u_0,v_0,z_0)$, this equation simplifies to the desired result.
	\end{proof}

	\subsection{Standard Representatives of Slopes}
	
	In this subsection we prove, roughly, that one can always find curves through a given generic point with a given generic slope. As a result we define the notion of a `standard representative' of a point and slope. Now the main result is:
	
	\begin{lemma}\label{slopes exist} Let $(x,y)\in M^2$ be generic over $\emptyset$. Fix infinitesimal coordinate systems at $x$ and $y$, and assume they are $\mathcal K$-definable over a set $A$. Let $\alpha\in\mathbb C$ be generic over $Axy$. Then one can find a stationary non-trivial plane curve $C$, with canonical base $c$, satisfying the following properties:
		\begin{enumerate}
			\item $C$ is $\mathcal M$-definable over $c$.
			\item $c$ is coherent of rank 2 over $\emptyset$, and coherent of rank 1 over $xy$.
			\item $(x,y)$ is a generic element of $C$ over $c$.
			\item With respect to the given infinitesimal coordinate systems, $\tau_{(x,y)}(C)=\alpha$.
			\item $c$ is interalgebraic (in the sense of $\mathcal K$) with $\alpha$ over $Axy$.
		\end{enumerate}
	\end{lemma}
	\begin{proof} Throughout we write $s$ for $(x,y)$. Let $\mathcal D=\{D_u:u\in U\}$ be an excellent family of plane curves which is $\mathcal M$-definable over $\emptyset$, as provided by Assumption \ref{M and K}. Let $D\subset M^2\times U$ be the graph of $\mathcal D$.
		
	First let $u'$ be generic in ${_sD}$ over $As$. Then $(s,u')\in D$ is generic over $\emptyset$, so $s$ is generic in $D_{u'}$ over $u'$. Using Lemma \ref{tangent space of component}, let $C'$ be a stationary component of $D_{u'}$, and $c'$ its canonical base, so that $s$ is a generic element of $C'$ over $c'$. After editing finitely many points (and potentially sacrificing that $C'\subset D_{u'})$, we may assume that $C'$ is $\mathcal M$-definable over $c'$. Let $\alpha'=\tau_s(C')$. Then by Remark \ref{slope remark}(2), $\alpha'$ is $\mathcal K$-definable over $Asc'$.
		
	Let $c''$ be an independent realization of $tp_{\mathcal K}(c'/As\alpha')$ over $Asc'\alpha'$. Then $c''$ is the canonical base of a plane curve $C''$ which also contains $s$ as a generic point over $c''$, and also has slope $\alpha'$ at $s$.
	
	Now by Corollary \ref{coherent family ranks}, and since $(s,u')$ is generic in $D$, we conclude that $u'$ is coherent of rank 2 over $\emptyset$, and coherent of rank 1 over $s$. Moreover, by Fact \ref{almost faithful facts}, $u'$ and $c'$ are $\mathcal M$-interalgebraic over $\emptyset$ -- so $c'$ is also coherent of rank 2 over $\emptyset$, and coherent of rank 1 over $s$. Finally, since $c''$ realizes the same $\mathcal K$-type over $s$ as $c'$, it is also coherent of rank 2 over $\emptyset$, and coherent of rank 1 over $s$. 
		
	Note that we have now established that $(s,C',C'',c',c'')$ is a generic non-transversality. So since $\mathcal M$ detects generic non-transversalities, we conclude that $c'$ and $c''$ are $\mathcal M$-dependent over $\emptyset$. 
	
	\begin{claim} $\rk(c''/sc')=0$.
	\end{claim}
	\begin{proof} Suppose not. Then $\rk(c'c''s)\geq 4$. But clearly $\rk(s/c'c'')\leq 1$, so $\rk(c'c'')\geq 3$, and thus $\rk(c''/c')\geq 1$. In particular $c'\neq c''$, which shows that $\rk(s/c'c'')=0$. But then $\rk(c'c'')\geq 4$, which contradicts that $c'$ and $c''$ are $\mathcal M$-dependent.
	\end{proof} 
	
	Now by the claim, and since $c''$ was chosen $\mathcal K$-independently over $Asc'\alpha'$, it follows that $\dim(c'/As\alpha')=0$. Since $\alpha'$ is $\mathcal K$-definable over $Asc'$, we have in fact that $c'$ and $\alpha'$ are interalgebraic over $As$. In particular $\dim(\alpha'/As)=1$, which shows that $\alpha'$ is generic in $\mathbb C$ over $As$. Then since $\alpha$ and $\alpha'$ are both generics in $\mathbb C$ over $As$, and since $\mathcal K$ is saturated, there is an automorphism $\sigma$ of $\mathcal K$ fixing $As$ and sending $\alpha'$ to $\alpha$. Let $c=\sigma(c')$. Then by the automorphism, $c$ is the canonical base of a non-trivial plane curve $C$ which is $\mathcal M$-definable over $c$, contains $s$ as a generic point over $c$, and has slope $\alpha$ at $s$. Moreover, all dimension-theoretic properties of $(s,c,\alpha)$ over $\emptyset$ and $A$ are identical to those of $(s,c',\alpha')$. This is enough to show that $C$ and $c$ are the desired objects satisfying (1)-(5) of the lemma.
	\end{proof}

	We now define:
	
	\begin{definition}\label{standard representative} Let $(x,y)$, $A$, and $\alpha$ be as in the statement of Lemma \ref{slopes exist} -- so we have fixed infinitesimal coordinate systems at $x$ and $y$ (say $f$ and $g$) which are $\mathcal K$-definable over $A$. We will call any pair $(C,c)$ satisfying (1)-(5) of Lemma \ref{slopes exist} a \textit{standard representative} of $\alpha$ at $(x,y)$, relative to $f$, $g$, and $A$. When $f$, $g$, and $A$ are clear from context, we simply call $(C,c)$ a \textit{standard representative} of $\alpha$ at $(x,y)$.
	\end{definition}
	
	\subsection{Detecting Generic Multiplication and Addition}
	
	We next show that $\mathcal M$ can detect the field operations on slopes `generically,' i.e. in the situation of Lemmas \ref{slope multiplication} and \ref{slope addition}. Modulo some computations, the statement will essentially follow from Lemmas \ref{slope multiplication} and \ref{slope addition} and the assumption that $\mathcal M$ detects generic non-transversalities.
	
	Before stating the result, we should recall that in any stable group, the product of two independent generics is again a generic. We should also point out that if $x\in M$ and $(u,v,z)\in\Gamma$ are independent generics, then each of $(x,u)$, $(x,v)$, and $(x,z)$ is generic in $M^2$ (because of the independence from $x$, combined with the fact that each of $u$, $v$, and $z$ is by itself generic in $M$). These two facts are indeed necessary for the statement to make sense. Now we show:
	
	\begin{lemma}\label{detecting products} $\mathcal M$ detects generic multiplication and addition of slopes, in the following senses:
		\begin{enumerate} 
			\item Let $x,y,z$ be independent generics in $M$ over $\emptyset$, equipped with fixed infinitesimal coordinate systems that are $\mathcal K$-definable over a set $A$. Let $\alpha_1$ and $\alpha_2$ be independent generics in $\mathbb C$ over $Axyz$, and let $\alpha_3=\alpha_1\alpha_2$. Let $(C_1,c_1)$, $(C_2,c_2)$, and $(C_3,c_3)$ be standard representatives of $\alpha_1$, $\alpha_2$, and $\alpha_3$ at $(x,y)$, $(y,z)$, and $(x,z)$, respectively. Then $c_3\in\operatorname{acl}_{\mathcal M}(c_1c_2xyz)$.
			\item Let $x\in M$ and $(u,v,z)\in\Gamma$ be independent generics over $\emptyset$, equipped with fixed infinitesimal coordinate systems (at $x,u,v,z$) which are $\mathcal K$-definable over a set $A$ and which normalize $(u,v,z)$. Let $\alpha_1$ and $\alpha_2$ be independent generics in $\mathbb C$ over $Axuvz$, and let $\alpha_3=\alpha_1+\alpha_2$. Let $(C_1,c_1)$, $(C_2,c_2)$, and $(C_3,c_3)$ be standard representatives of $\alpha_1$, $\alpha_2$, and $\alpha_3$ at $(x,u)$, $(x,v)$, and $(x,z)$, respectively. Then $c_3\in\operatorname{acl}_{\mathcal M}(c_1c_2xuvz)$.
		\end{enumerate}
	\end{lemma}
	
	\begin{proof} The proofs of (1) and (2) are almost identical, so we will prove them simultaneously. Let $w$ be $(x,y,z)$ in (1) and $(x,u,v,z)$ in (2). Let $s_1$ be $(x,y)$ in (1) and $(x,u)$ in (2). Let $s_2$ be $(y,z)$ in (1) and $(x,v)$ in (2). Let $s_3$ be $(x,z)$ in each of (1) and (2). So in each case, for each $i$ we have $s_i\in C_i$. Finally, let $\bar c=(c_1,c_2,c_3)$.
		
	Note that in either case, $w$ is coherent of rank 3 over $\emptyset$. Moreover, using that each $c_i$ is interalgebraic with $\alpha_i$ over $Aw$, it follows that $c_1c_2$ is coherent of rank 2 over $w$. In particular $c_1c_2w$ is coherent of rank 5 over $\emptyset$. Since each $\dim(c_i)=2$, it follows that $\dim(w/c_1c_2)\geq 1$. On the other hand, using that each projection $C_i\rightarrow M$ and $\Gamma\rightarrow M^2$ is finite-to-one, it is easy to see that the coordinates of $w$ are pairwise $\mathcal M$-interalgebraic over $c_1c_2$. In particular we get that $\dim(s_1/c_1c_2)=\dim(s_2/c_1c_2)\geq 1$, which shows that $s_1$ and $s_2$ are generic in $C_1$ and $C_2$, respectively, over $c_1c_2$.
	
	Let $D$ be $C_2\circ C_1$ in (1) and $C_1\oplus C_2$ in (2). Then by either Lemma \ref{slope multiplication} or Lemma \ref{slope addition}, it follows that $s_3$ is generic in $D$ over $c_1c_2$, and $\tau_{s_3}(D)=\alpha_3$. We now apply Lemma \ref{tangent space of component}: the result is a stationary component $C$ of $D$, with canonical base $c$, such that $s_3$ is generic in $C$ over $c_1c_2c$ and $\tau_{s_3}(C)=\alpha_3$. Note that $c$ is $\mathcal M$-algebraic over $c_1c_2$, since $D$ is $\mathcal M$-definable over $c_1c_2$. Thus $c$ is also $\mathcal M$-algebraic over $c_1c_2w$ -- so since $c_1c_2w$ is coherent, so is $c$. 
	
	Note that we have now verified that $(s_3,C_3,C,c_3,c)$ is a generic non-transversality. So since $\mathcal M$ detects generic non-transversalities, we conclude that $c_3$ and $c$ are $\mathcal M$-dependent over $\emptyset$. In particular we get that $\rk(c_3/c)\leq 1$. But since $c$ is $\mathcal M$-algebraic over $c_1c_2$, this implies that $\rk(c_3/c_1c_2)\leq 1$. Thus $\rk(\bar c)\leq 5$. We now proceed in two cases:

	\begin{itemize}
		\item First suppose that $\rk(w/\bar c)=0$. Then $\rk(\bar cw)\leq 5$. But since $\rk(c_1c_2w)=5$, this is only possible if $\rk(c_3/c_1c_2w)=0$, as desired.
		\item Now suppose that $\rk(w/\bar c)\geq 1$. But as stated above, all coordinates of $w$ are pairwise $\mathcal M$-interalgebraic over $\bar c$. So in fact $\rk(s_3/\bar c)\geq 1$. In particular this implies that $C_3\cap D$ is infinite, and thus $C_3$ is a stationary component of $D$. Since $D$ is $\mathcal M$-definable over $c_1c_2$, we conclude that $c_3$ is $\mathcal M$-algebraic over $c_1c_2$, which implies the desired result.
	\end{itemize}
	\end{proof}
	
	\subsection{Detecting Multiplication and Addition at a Single Point}
	
	In this subsection we prove a version of Lemma \ref{detecting products} in which all standard representatives are chosen at a single point $(x_0,y_0)\in M^2$. This is important for the interpretation of a field, roughly because it will give us a single underlying set on which to interpret the field operations.
	
	The result of this subsection will need to take place in the presence of additional parameters -- so we begin by fixing several pieces of notation.
	
	\begin{convention}\label{fixed 5 points} For the rest of this section, we adopt the following:
	\begin{enumerate}
		\item We fix independent generics $x_0\in M$, $y_0\in M$, and $(u_0,v_0,z_0)\in\Gamma$, and denote the 5-tuple $(x_0,y_0,u_0,v_0,z_0)$ by $w_0$. 
		\item We also fix infinitesimal coordinate systems at each of the five coordinates of $w_0$, and assume that they normalize $(u_0,v_0,z_0)$ and are $\mathcal K$-definable over a set $B$. 
		\item Next we fix a generic element $\beta\in\mathbb C$ over $Bw_0$.
		\item We then fix standard representatives $(D_1,d_1)$ of $\beta$ at $(y_0,u_0)$, $(D_2,d_2)$ of $\beta$ at $(y_0,v_0)$, $(D_3,d_3)$ of $\beta$ at $(y_0,z_0)$, and $(D_4,d_4)$ of $\frac{1}{\beta}$ at $(z_0,x_0)$.
		\item Finally, we let $A$ be the parameter set consisting of $w_0$, and $d_i$ for $1\leq i\leq 4$.
	\end{enumerate}
	\end{convention}

	\begin{remark}\label{fixed 5 points remark} Let us make some comments about Convention \ref{fixed 5 points}:
	\begin{enumerate}
		\item Note that any two coordinates of $w_0$ are independent generics in $M$ over $\emptyset$ (this is necessary for (4) to make sense). Moreover, the same almost holds for any three coordinates of $w_0$, the sole exception being the triple $(u_0,v_0,z_0)$.
		\item That (2) is possible is given by Lemma \ref{normalization exists}.
		\item It follows by the genericity of $\beta$ over $Bw_0$ that $\beta\neq 0$, and moreover that $\frac{1}{\beta}$ is also generic over $Bw_0$, since it is interalgebraic with $\beta$. This is implicitly used in (4).
		\item Finally, we note that $A$ lives entirely within $\mathcal M^{\textrm{eq}}$, while $B$ involves parameters from $\mathcal K$. We will typically use $AB$ as a parameter set in the presence of the $\dim$ function, making sure to restrict to $A$ when using $\rk$. Note that $\beta$ is indirectly included in $AB$ (more precisely, its definable closure), for example because it is the slope of $D_1$ at $(y_0,u_0)$ with respect to infinitesimal coordinate systems defined over $A$.
	\end{enumerate}
	\end{remark}

	The proof of the main lemma of this subsection will be rather quick and potentially confusing. However it is based on a simple picture, so before giving the lemma we will give a brief informal account of the strategy. Indeed this is arguably the most important result of the section, since it is the reason that a field can be found using only generic non-transversalities.
	
	Now intuitively we will show that, given the ability to operate on slopes while moving points around generically (as in Lemma \ref{detecting products}), it is actually possible to operate on slopes \textit{without} moving around. To do this, the idea is to break a desired multiplication or addition of slopes into a longer sequence of steps, each of which moves to a new point as in Lemma \ref{detecting products}; we then apply Lemma \ref{detecting products} at each step along the way. The trick is that Lemma \ref{detecting products} only requires each step by itself to involve a generic configuration of points, while it does not require any independence between the points used in different steps. So if we design our path carefully, we can end up at the point where we started. This is the reason for the extra coordinates $u_0$, $v_0$, and $z_0$ -- they will make up the `intermediate steps' on our way \textit{back} to $(x_0,y_0)$. 
	
	Now in order to apply Lemma \ref{detecting products} along the way, we need to perform a generic slope operation in each step. This is why we need $\beta$ and $\frac{1}{\beta}$ -- indeed, we will ensure that each `extra' step in the sequence is just a scaling by $\beta$ or $\frac{1}{\beta}$; and again, designing the path carefully, these extra scalings will cancel each other out in the end. Thus, as the reader might now be able to guess, all of the extra steps will be given by the $D_i$'s from Convention \ref{fixed 5 points}.
	
	We now give the result:
	
	\begin{lemma}\label{detecting products at a point} $\mathcal M$ detects generic multiplication and addition of slopes at $(x_0,y_0)$, in the following sense: Let $\alpha_1$ and $\alpha_2$ be independent generics in $\mathbb C$ over $AB$, and let $\alpha_3$ be eitehr $\alpha_1\alpha_2$ or $\alpha_1+\alpha_2$. For each $i$ let $(C_i,c_i)$ be a standard representative of $\alpha_i$ at $(x_0,y_0)$. Then $c_3\in\operatorname{acl}_{\mathcal M}(Ac_1c_2)$.
	\end{lemma}
	\begin{proof} We need different arguments for multiplication and addition:
		\begin{itemize}
			\item First suppose that $\alpha_3=\alpha_1\alpha_2$. Let $(C_4,c_4)$ be a standard representative of $\alpha_1\beta$ at $(x_0,z_0)$, and let $(C_5,c_5)$ be a standard representative of $\frac{\alpha_2}{\beta}$ at $(z_0,y_0)$. Then the following are all consequences of Lemma \ref{detecting products}:
			\begin{enumerate}
				\item $c_4\in\operatorname{acl}_{\mathcal M}(Ac_1)$, by considering the curves $(C_1,D_3,C_4)$, points $(x_0,y_0,z_0)$, and slopes $(\alpha_1)\cdot(\beta)=\alpha_1\beta$.
				\item $c_5\in\operatorname{acl}_{\mathcal M}(Ac_2)$, by considering the curves $(D_4,C_2,C_5)$, points $(z_0,x_0,y_0)$, and slopes $(\frac{1}{\beta})\cdot(\alpha_2)=\frac{\alpha_2}{\beta}$.
				\item $c_3\in\operatorname{acl}_{\mathcal M}(Ac_4c_5)$, by considering the curves $(C_4,C_5,C_3)$, points $(x_0,z_0,y_0)$, and slopes $(\alpha_1\beta)\cdot(\frac{\alpha_2}{\beta})=\alpha_1\alpha_2$.
			\end{enumerate}
		In particular, combining (1), (2), and (3) gives the desired result.		
			\item Now suppose that $\alpha_3=\alpha_1+\alpha_2$. Let $(C_4,c_4)$ be a standard representative of $\alpha_1\beta$ at $(x_0,u_0)$, let $(C_5,c_5)$ be a standard representative of $\alpha_2\beta$ at $(x_0,v_0)$, and let $(C_6,c_6)$ be a standard representative of $(\alpha_1+\alpha_2)\beta$ at $(x_0,z_0)$. Then the following are all consequences of Lemma \ref{detecting products}:
			\begin{enumerate}
				\item $c_4\in\operatorname{acl}_{\mathcal M}(Ac_1)$, by considering the curves $(C_1,D_1,C_4)$, points $(x_0,y_0,u_0)$, and slopes $(\alpha_1)\cdot(\beta)=\alpha_1\beta$.
				\item $c_5\in\operatorname{acl}_{\mathcal M}(Ac_2)$, by considering the curves $(C_2,D_2,C_5)$, points $(x_0,y_0,v_0)$, and slopes $(\alpha_2)\cdot(\beta)=\alpha_2\beta$.
				\item $c_6\in\operatorname{acl}_{\mathcal M}(Ac_4c_5)$, by considering the curves $(C_4,C_5,C_6)$, points $(x_0,u_0,v_0,z_0)$, and slopes $(\alpha_1\beta)+(\alpha_2\beta)=(\alpha_1+\alpha_2)\beta$.
				\item $c_3\in\operatorname{acl}_{\mathcal M}(Ac_6)$, by considering the curves $(C_6,D_4,C_3)$, points $(x_0,z_0,y_0)$, and slopes $((\alpha_1+\alpha_2)\beta)\cdot(\frac{1}{\beta})=\alpha_1+\alpha_2$.
			\end{enumerate}
			In particular, combining (1), (2), (3), and (4) gives the desired result.
		\end{itemize}
	\end{proof}

	\begin{remark} To be precise, one should check that each instance of Lemmas \ref{slopes exist} and \ref{detecting products} in the above proof is valid. This amounts to Remark \ref{fixed 5 points remark}(1), the fact that all slopes in question are generic over $Bw_0$, and the fact that any two slopes used in the same application of Lemma \ref{detecting products} are independent over $Bw_0$. But the required statements about slopes follow easily from the fact that $\alpha_1$, $\alpha_2$, and $\beta$ are independent generics in $\mathbb C$ over $Bw_0$; we leave the details to the reader.
	\end{remark}

	\subsection{Field Configurations and the Proof of the Theorem}
	
	In the final subsection we prove the interpretability of an algebraically closed field, still assuming that $\mathcal M$ detects generic non-transversalities. We do this by building a \textit{field configuration}:
	
	\begin{definition}\label{field con def} Let $\mathcal X$ be a strongly minimal structure, and let $\rk$ denote Morley rank in $\mathcal X$. Let $a,b,c,x,y,z$ be tuples in $\mathcal X^{\textrm{eq}}$, and let $E$ be a set of parameters in $\mathcal X^{\textrm{eq}}$. We call $(a,b,c,x,y,z)$ a \textit{field configuration} in $\mathcal X$ over $E$ if the following hold:
		\begin{enumerate}
			\item $\rk(x/E)=\rk(y/E)=\rk(z/E)=1$.
			\item $\rk(a/E)=\rk(b/E)=\rk(c/E)=k$ for some $k>1$.
			\item Any two of the six points are independent over $E$, and moreover $a$, $b$, and $x$ are independent over $E$.
			\item $c\in\operatorname{acl}(Eab)$, $y\in\operatorname{acl}(Eax)$, and $z\in\operatorname{acl}(Eby)\cap\operatorname{acl}(Ecx)$.
		\end{enumerate}
	\end{definition}
	
	\begin{remark}\label{field con ex}
			The reader unfamiliar with this definition will find guidance in the following example: Suppose that $\mathcal X=(F,+,\cdot)$ is itself an algebraically closed field, and let $G=\operatorname{AGL}_1(F)$ denote the group of invertible affine linear maps from $F$ to $F$ -- so $G$ acts definably on $F$. Fix any set $E$, and let $a\in G$, $b\in G$, and $x\in F$ be independent generics over $E$. Then letting $c=a\cdot b$, $y=a\cdot x$, and $z=b\cdot y=c\cdot x$, it is easy to check that $(a,b,c,x,y,z)$ forms a field configuration over $E$, with $k=2$. We will call this a \textit{standard group configuration} in $F$.
	\end{remark}
	
	By Remark \ref{field con ex}, any algebraically closed field admits a field configuration over any set of parameters. A remarkable theorem of Hrushovski (see \cite{Bou}) gives a converse:	
	
	\begin{fact}\label{field con fact} Let $\mathcal X$ be strongly minimal, and suppose that there is a field configuration in $\mathcal X$ over some parameter set. Then $\mathcal X$ interprets an algebraically closed field. 
	\end{fact}
	
	\begin{remark} The full statement of Fact \ref{field con fact} actually classifies field configurations up to point-wise interalgebraicity. All we need is the interpretability of the field, however.
	\end{remark} 

	Now let us continue with the proof. We retain all of the notation from Convention \ref{fixed 5 points}. From now on we will only use standard representatives taken at $(x_0,y_0)$. Thus the following will ease the presentation:
	
	\begin{notation}\label{fixed representatives} For each $\alpha\in\mathbb C$ which is generic over $AB$, we fix one standard representative of $\alpha$ at $(x_0,y_0)$, and denote its canonical base by $s(\alpha)$. If $a=(\alpha_1,...,\alpha_k)$ is a tuple of such elements of $\mathbb C$, then we let $s(a)=(s(\alpha_1),...,s(\alpha_k))$.
	\end{notation}

	Let $G=\operatorname{AGL}_1(\mathbb C)$. We view $G$ in the usual way as a semidirect product of the additive and multiplicative groups of $\mathbb C$. In particular, a generic element $a\in G$ is identified with an ordered pair $(\alpha_1,\alpha_2)$, where $\alpha_1$ and $\alpha_2$ are independent generics in $\mathbb C$. Thus Notation \ref{fixed representatives} applies to $a$, so $s(a)$ is well-defined. Note that, as in Remark \ref{field con ex}, we will write the action of $G$ on $\mathbb C$ multiplicatively.
	
	Now let us begin by restating Lemma \ref{detecting products at a point} in this language:
	
	\begin{lemma}\label{sum and product} If $\alpha_1$ and $\alpha_2$ are independent generics in $\mathbb C$ over $AB$, then $s(\alpha_1\cdot\alpha_2),s(\alpha_1+\alpha_2)\in\operatorname{acl}_{\mathcal M}(As(\alpha_1)s(\alpha_2))$.
	\end{lemma}

	We thus conclude:
	
	\begin{lemma}\label{coding agl} Let $a\in G$, $b\in G$, and $x\in\mathbb C$ be independent generics over $AB$. Then:
		\begin{enumerate}
			\item $s(a\cdot x)\in\operatorname{acl}_{\mathcal M}(As(a)s(x))$.
			\item $s(a\cdot b)\in\operatorname{acl}_{\mathcal M}(As(a)s(b))$.
		\end{enumerate}
	\end{lemma}
	\begin{proof} Let $a=(\alpha_1,\alpha_2)$ and $b=(\beta_1,\beta_2)$. So $a\cdot x=\alpha_1x+\beta_1$, and $a\cdot b=(\alpha_1\beta_1,\alpha_1\beta_2+\alpha_2)$.
		\begin{enumerate}
			\item Let $x'=\alpha_1x$. By Lemma \ref{sum and product}, $s(x')\in\operatorname{acl}_{\mathcal M}(As(\alpha_1)s(x))$. Again by Lemma \ref{sum and product}, $s(a\cdot x)\in\operatorname{acl}_{\mathcal M}(As(x')s(\beta_1))$. Combining these two assertions gives the desired result.
			\item By Lemma \ref{sum and product} applied to $\alpha_1\beta_1$ and (1) applied to $a\cdot\beta_2$.
		\end{enumerate}
	\end{proof}

	We now show:
	
	\begin{theorem}\label{field interpretation} Under the assumptions that $\dim M=1$ and $\mathcal M$ detects generic non-transversalities, $\mathcal M$ interprets an algebraically closed field.
	\end{theorem}

	\begin{proof} Let $(a,b,c,x,y,z)$ be a standard field configuration in $\mathbb C$ over $AB$. We will show that $(s(a),s(b),s(c),s(x),s(y),s(z))$ forms a field configuration in $\mathcal M$ over $A$. To do this, first note that each clause in (4) of Definition \ref{field con def} follows from Lemma \ref{coding agl}. So it remains to show (1), (2), and (3). But each of these is implied by the following claim:
	\begin{claim} Let $w$ be a subtuple of $(a,b,c,x,y,z)$ whose elements are independent over $AB$. Then $\rk(s(w)/A)=\dim(w/AB)$.
	\end{claim} 
	\begin{proof}
		Let $\dim(w/AB)=j$. Note by independence and the construction of the tuple $(a,b,c,x,y,z)$ that $w$ is a tuple of $j$ independent generics in $\mathbb C$ over $AB$. Now by Lemma \ref{slopes exist}, $s(w)$ and $w$ are interalgebraic over $AB$ -- so we have $\dim(s(w)/AB)=j$, thus $\dim(s(w)/A)\geq j$, and thus by Corollary \ref{dim rk inequality} $\rk(s(w)/A)\geq j$. On the other hand, $s(w)$ is a tuple of $j$ canonical bases of standard representatives at $(x_0,y_0)$, each of which has rank $\leq 1$ over $A$ by Lemma \ref{slopes exist}. So in fact $\rk(s(w)/A)=j$, as desired.
	\end{proof}
	So we have indeed found a field configuration in $\mathcal M$ over $A$. Then we are done by Fact \ref{field con fact}.
	\end{proof}
		
	\section{On Closures of $\mathcal M$-Definable Sets}
	We now drop all of the assumptions gathered over the course of the previous section, reverting instead to only Assumption \ref{M and K}. So we allow the case $\dim M>1$, and we do not assume detection of generic non-transversalities. Our task in the final two sections will be to show without any such assumptions that $\mathcal M$ \textit{does} detect generic non-transversalities, and subsequently deduce the main theorem. The present section contains the proof of the main technical tool -- our bound on closure points of $\mathcal M$-definable sets. Once this is done, the proof that $\mathcal M$ detects generic non-transversalities will be quite short.
	
	In order to motivate our work in this section, let us briefly explain a strategy that did not work for detecting generic non-transversalities. Our general goal is to bound the frontiers of $\mathcal M$-definable sets, by showing that $\mathcal M$ can `detect' frontier points. Among the most ideal precise statements of this form, for example, we might single out the following:
	
	\begin{definition}\label{detecting frontier} Say that $\mathcal M$ \textit{detects frontiers} if whenever $X$ is $\mathcal M$-definable over a set $A$, and $x\in\operatorname{Fr}(X)$, then $\rk(x/A)<\rk(X)$.
		\end{definition}
	
	We remark that results of this type have appeared in similar settings when $\mathcal M$ expands a group and $X$ is a plane curve -- see for example \cite{PetStaACF}, \cite{HK}, and \cite{EHP}. Now the following lemma serves as the motivation for the entire rest of the paper:
	
	\begin{lemma}\label{ideal scenario} If $\mathcal M$ detects frontiers, then $\mathcal M$ detects generic non-transversalities.
	\end{lemma}
	\begin{proof} We use (3) of Lemma \ref{detecting tangency equivalence}. So, let $\mathcal C=\{C_t:t\in T\}$ and $\mathcal D=\{D_u:u\in U\}$ be standard families of plane curves, and let $(\hat x,\hat t,\hat u)$ be a generic $(\mathcal C,\mathcal D)$-non-injectivity. By Lemma \ref{detecting tangency equivalence}, it suffices to show that $\hat t$ and $\hat u$ are $\mathcal M$-dependent over $\emptyset$ -- or equivalently that $\rk(\hat t,\hat u)<\rk(T\times U)$. Let $I\subset M^2\times T\times U$ be the graph of the family of intersections $\mathcal I_{\mathcal C,\mathcal D}$, so that by assumption $(\hat x,\hat t,\hat u)$ is non-injective in the projection $I\rightarrow T\times U$. Also let $Z\subset T\times U$ be the set of $(t,u)$ such that $C_t\cap D_u$ is infinite, and note by almost faithfulness that the projection $Z\rightarrow T$ is finite-to-one. Now we consider two cases:
		\begin{itemize}
			\item First suppose $(\hat t,\hat u)\in\overline Z$. Then assuming $\mathcal M$ detects frontiers (and since either $(\hat t,\hat u)\in Z$ or $(\hat t,\hat u)\in\operatorname{Fr}(Z)$), we get that $\rk(\hat t,\hat u)\leq\rk(Z)$. Moreover, since the projection $Z\rightarrow T$ is finite-to-one we have $\rk(Z)\leq\rk(T)$, and since $\mathcal D$ is a standard family (i.e. of rank $\geq 1$) we have $\rk(T)<\rk(T\times U)$. Thus, combined, we get $$\rk(\hat t,\hat u)\leq\rk(Z)\leq\rk(T)<\rk(T\times U),$$ as desired.
			\item Now assume $\rk(\hat t,\hat u)\notin\overline Z$. Let $I'$ be the set of $(x,t,u)\in I$ such that $(t,u)\notin Z$. Then by assumption it follows that $(\hat x,\hat t,\hat u)$ is still non-injective in the projection $I'\rightarrow T\times U$. Let $P$ be the set of $(x,x',t,u)$ such that $x\neq x'$ and $(x,t,u),(x',t,u)\in I'$. So by non-injectivity we have $(\hat x,\hat x,\hat t,\hat u)\in\operatorname{Fr}(P)$, which by assumption gives that $\rk(\hat x,\hat x,\hat t,\hat u)<\rk(P)$. On the other hand, by the choice of $I'$ the projection $P\rightarrow T\times U$ is finite-to-one, which shows that $\rk(P)\leq\rk(T\times U)$, and thus $\rk(\hat x,\hat x,\hat t,\hat u)<\rk(T\times U)$. In particular $\rk(\hat t,\hat u)<\rk(T\times U)$, again as desired.		
		\end{itemize}
	\end{proof}

	So, a reasonable strategy would be to try to adapt and generalize the results of \cite{PetStaACF}, \cite{HK}, and \cite{EHP} to arbitrary $\mathcal M$-definable sets, and apply Lemma \ref{ideal scenario}. Unfortunately, though, the full statement of Definition \ref{detecting frontier} seems too far out of reach -- and indeed doesn't quite make sense in complete generality, for reasons we will explain below. 
	
	The ultimately successful strategy, then, is the identification of a weaker form of Definition \ref{detecting frontier}, which both can be proven in full generality and is strong enough to admit an adaptation of Lemma \ref{ideal scenario}. We turn now toward stating and discussing this result. We need the following two definitions:

	\begin{definition}\label{coordinate-wise generic} Let $x=(x_1,...,x_n)\in M^n$ be a tuple. We say that $x$ is \textit{coordinate-wise generic} if each $x_i$ is generic in $M$ over $\emptyset$.
	\end{definition}
	So coordinate-wise genericity is weaker than genericity, because we do not require the coordinates to be independent. Importantly, coordinate-wise genericity should always be interpreted over $\emptyset$, even if there is a natural ambient parameter set around.
	
	\begin{definition}\label{independent projections} Let $X\subset M^n$ be non-empty and $\mathcal M$-definable over $A$, and let $\pi_i,\pi_j:M^n\rightarrow M$ be projections. Then $\pi_i$ and $\pi_j$ are \textit{independent on} $X$ if for all $\mathcal M$-generic $x\in X$ over $A$, the elements $\pi_i(x)$ and $\pi_j(x)$ are $\mathcal M$-independent over $A$.
	\end{definition}
	\begin{remark} It is easy to see that Definition \ref{independent projections} does not depend on the parameter set $A$. It is also easy to see that, if we made different versions of the definition using genericity and independence in the sense of $\mathcal K$ and $\mathcal M$, the resulting notions would be equivalent. Thus we do not use clarifying prefixes to indicate the choice of structure and parameter set.
	\end{remark}
	Now our result is:
		\begin{proposition}\label{closure prop} Let $X\subset M^n$ be $\mathcal M$-definable over $A$ and of rank $r\geq 0$, and let $x=(x_1,...,x_n)\in\overline X$ be a coordinate-wise generic point. Then $\operatorname{rk}(x/A)\leq r$. Moreover, one of the following holds:
		\begin{enumerate}
			\item $\operatorname{rk}(x/A)<r$.
			\item For all $i\neq j$ such that the projections $\pi_i,\pi_j:M^n\rightarrow M$ are independent on $X$, the elements $x_i$ and $x_j$ are $\mathcal M$-independent over $A$.
		\end{enumerate}
		\end{proposition}

	That is, we only consider coordinate-wise generic points; we include the possibility that $x\in X$; we in general replace the strict inequality of Definition \ref{detecting frontier} with a weak inequality; and most importantly, the inequality can be made strict under certain conditions, relating to which pairs of coordinates in the specified point are $\mathcal M$-independent. For intuition, the reader could think of this independence condition as saying that any full-rank frontier points of $X$ (i.e. those where $\rk(x/A)=\rk(X)$) must `resemble' generic elements of $X$ in a precise sense -- namely, $x$ must look identical to a generic point of $X$ at the level of pairwise independence of its coordinates.
	
	As it turns out, the dichotomy of (1) and (2) in the statement of Proposition \ref{closure prop} will be perfect for adapting Lemma \ref{ideal scenario}. Indeed, the frontier point $(\hat x,\hat x,\hat a,\hat b)\in\operatorname{Fr}(P)$ considered in the proof of the lemma has two pairs of equal (thus $\mathcal M$-dependent) coordinates, which after a small amount of work will let us rule out (2), and thereby reduce to the ideal scenario of (1). We will then be able to conclude that $\mathcal M$ detects generic non-transversalities exactly as in Lemma \ref{ideal scenario}.
	
	Before proceeding, let us also explain the necessity of restricting to coordinate-wise generic points. The issue is that any constructible set can be given many different Zariski topologies, based on different embeddings into projective space. For some sets (say an irreducible variety, for example) there is a canonical `correct' topology, but in general there is no clear way to choose which topology we want. Moreover, it seems that any reasonable treatment of frontier points would not prioritize one such topology over another, at least for arbitrary constructible sets. Thus, in a sense, the frontier we are after is not even well-defined; this is why we stated above that Definition \ref{detecting frontier} does not truly make sense in full generality. Instead, we have chosen to work exclusively with the affine topology (which is often not the canonical one), but only considering coordinate-wise generic points. Indeed, even though there may be several topologies on a single constructible set, any two such topologies will agree upon restricting to a large constructible subset. Thus, if we work only with points built from generic coordinates in $M$, the topology is well-defined on all powers of $M$.
	
	\subsection{Overview and Inductive Assumption}
	Let us now outline the proof of Proposition \ref{closure prop}. On a global scale, the proof will be done by induction on $r$. It will be trivial to verify the case $r=0$; the vast majority of the work, then, will really be an inductive step. Due to its length, we have chosen to organize the inductive step into several distinct propositions, with a blanket inductive hypothesis throughout. So, for now, we do the following:
	
	\begin{assumption}\label{inductive assumption}
		Until stated otherwise, we assume $r\geq 1$ and the statement of Proposition \ref{closure prop} holds for all sets of rank $<r$.
	\end{assumption} 
	
	We now outline the inductive step:
	
	\begin{enumerate}
		\item Since $X\subset M^n$, we have $n\geq r$. It will be trivial to verify the case $n=r$, so we assume $n\geq r+1$.
		\item First we assume each of the following:
		\begin{enumerate}
			\item $X$ is a stationary `hypersurface' -- that is, $X$ is stationary and $n=r+1$.
			\item $X$ is a generic member of a very large almost faithful family.
			\item Each projection $X\rightarrow M^r$ is almost finite-to-one (in our terminology, $X$ is \textit{non-trivial}).
			\item The point $x$ is generic in $M^{r+1}$ over $\emptyset$ (not just coordinate-wise).
		\end{enumerate} 
		In this case, we are able to adapt the geometric argument of \cite{PetStaACF}, \cite{HK}, and \cite{EHP} for plane curves in strongly minimal groups. The idea is to study the intersections of $X$ with the members of a carefully chosen family of `curves' (rank one sets) in $M^{r+1}$. We will see that, under certain assumptions, those curves containing $x$ have fewer intersections with $X$ than do other curves -- a phenomenon enabling the $\mathcal M$-definable recognition of $x$.
		\item Next, again following previous authors, we extend (2) to the statement of Proposition \ref{closure prop} in the general case that $X$ is a non-trivial stationary hypersurface (in the sense described above). This step involves using an abstract version of `translation' to transform the setup into (2) above.
		\item We then consider `trivial' stationary hypersurfaces $X$ -- i.e. those with a projection $X\rightarrow M^r$ that is not almost finite-to-one. Such a projection allows us to reduce to the case of $r-1$-hypersurfaces, and subsequently apply the inductive hypothesis for $r-1$. Note that this step, in combination with (3), will cover all stationary $r$-hypersurfaces.
		\item We next extend to stationary sets in the case $n\geq r+2$. Here the argument proceeds by studying all projections of $X$ to $M^{r+1}$, and applying (4).
		\item Finally, we extend to sets which might not be stationary, by breaking each one into finitely many stationary components and applying (5).
	\end{enumerate}
	
	We now proceed with the inductive step.

	\subsection{Families of Hypersurfaces}
		The purpose of this subsection is to introduce the setting in which we will be able to adapt the more classical geometric argument of \cite{PetStaACF}, \cite{HK}, and \cite{EHP}; Namely, we develop a brief theory of families of `hypersurfaces'. We will then first show Proposition \ref{closure prop} for generic members of very large such families.
		
		\begin{definition} An $r$-hypersurface is an $\mathcal M$-definable subset of $M^{r+1}$ of rank $r$.
		\end{definition}
	
		\begin{convention} It will make our notation a bit more natural and elegant to express the coordinates of a point $x\in M^{r+1}$ using the indices 0 through $r$, rather than 1 through $r+1$. We will begin doing this now.
		\end{convention}
		
		So, for example, a plane curve is precisely a 1-hypersurface. Our next goal is to find an analog of non-triviality for hypersurfaces:.
		
		\begin{lemma}\label{non-trivial hypersurface fact} Let $X\subset M^{r+1}$ be an $r$-hypersurface, $\mathcal M$-definable over a set $A$. Then the following are equivalent: 
			\begin{enumerate}
				\item Each of the $r+1$ projections $X\rightarrow M^r$ is almost finite-to-one.
				\item For all $\mathcal M$-generic $x=(x_0,...,x_r)\in X$ over $A$, any $r$ of the coordinates $\{x_0,...,x_r\}$ have rank $r$ over $A$. 
			\end{enumerate}
		\end{lemma}
		\begin{proof} Easy by the definition of almost finite-to-one.
			\end{proof}
	
		\begin{definition}\label{non-trivial hypersurace} We call $r$-hypersurfaces satisfying (1) and (2) of Lemma \ref{non-trivial hypersurface fact} \textit{non-trivial}.
		\end{definition}
		
		Note that non-triviality is $\mathcal M$-definable in families, by (1) in the definition. Thus it makes sense to consider families of non-trivial $r$-hypersurfaces. In particular, if $X$ is a stationary non-trivial hypersurface then Lemmas \ref{almost faithful families exist} and \ref{coherent version of almost faithful family existence} apply to $X$, and we can assume that each hypersurface in the resulting family is non-trivial.
		
		Finally, we discuss common points in families of hypersurfaces:
		
		\begin{notation} Let $\mathcal H=\{H_t:t\in T\}$ be an $\mathcal M$-definable family of subsets of $M^{r+1}$. For $x\in M^{r+1}$, we set ${_xH}=\{t\in T:x\in H_t\}$.
		\end{notation}
		
		\begin{definition}\label{hypersurface common point} Let $\mathcal H=\{H_t:t\in T\}$ be an $\mathcal M$-definable almost faithful family of $r$-hypersurfaces. We say that $x\in M^{r+1}$ is $\mathcal H$-common if $\rk({_xH})=\rk(T)$.
		\end{definition}
		
		Note that the set of $\mathcal H$-common points is $\mathcal M$-definable over the same parameters that define $\mathcal H$. Now the main point we need is the following (note that (1) below is a generalization of Lemma \ref{finitely many common points}(1)):
		\begin{lemma}\label{hypersurfaces finitely many common points} Let $\mathcal H=\{H_t:t\in T\}$ be an $\mathcal M$-definable almost faithful family of $r$-hypersurfaces, where $\rk(T)\geq 1$.
			\begin{enumerate}
				\item The set of $\mathcal H$-common points has rank at most $r-1$.
				\item Letting $H\subset M^{r+1}\times T$ be the graph of $\mathcal H$, the projection $H\rightarrow M^{r+1}$ is generically dominant.
			\end{enumerate}
		\end{lemma}
		\begin{proof} For ease of notation we assume that $\mathcal H$ is $\mathcal M$-definable over $\emptyset$.
			\begin{enumerate}
				\item Let $x$ be $\mathcal H$-common. It will suffice to show that $\rk(x)\leq r-1$. Now let $t$ and $u$ be $\mathcal M$-independent $\mathcal M$-generics in ${_xH}$ over $x$. Then by commonness, $\rk(tu/x)=2\cdot\rk(T)$. Thus $t$ and $u$ are also $\mathcal M$-independent $\mathcal M$-generics in $T$ over $\emptyset$, which shows that $x$ and $tu$ are $\mathcal M$-independent over $\emptyset$. In particular, it will suffice to show that $\rk(x/tu)\leq r-1$. But by the aforementioned independence of $t$ and $u$ we have $\rk(u/t)=\rk(T)\geq 1$ -- so by almost faithfulness we have $\rk(H_t\cap H_u)\leq r-1$, thus $\rk(x/tu)\leq r-1$.
				\item Let $(x,t)\in H$ be $\mathcal M$-generic. Clearly this implies that $t$ is $\mathcal M$-generic in $T$ and $x$ is $\mathcal M$-generic in $H_t$ over $t$. In particular we have $\rk(x)\geq\rk(x/t)=r$, which by (1) shows that $x$ is not $\mathcal H$-common. Thus $\rk(t/x)\leq\rk(T)-1$. On the other hand we have $\rk(xt)=\rk(T)+r$ by definition, so by additivity we get $\rk(x)\geq r+1$. Thus $x$ is $\mathcal M$-generic in $M^{r+1}$, which proves (2).
			\end{enumerate}
		\end{proof}		
	\subsection{The Main Geometric Argument}
	
	We are now ready to begin the proof of Proposition \ref{closure prop}. The main content of the proposition is Proposition \ref{frontier proof} below; the proof will be quite long and complicated, so we will try to explain things as we go. 
	
	\begin{proposition}\label{frontier proof} Let $\mathcal H=\{H_t:t\in T\}$ be an almost faithful family of non-trivial $r$-hypersurfaces of rank $k>(r+1)\cdot\dim M$, and let $\mathcal C=\{C_s:s\in S\}$ be an excellent family of plane curves. Assume that each of $\mathcal H$ and $\mathcal C$ is $\mathcal M$-definable over a set $A$. Let $\hat t\in T$ and $\hat x=(\hat x_0,...\hat x_r)\in M^{r+t}$ each be generic over $A$, and assume that $\hat x\in\overline{H_{\hat t}}$. Then at least one of the following three things happens:
		\begin{enumerate}
			\item $\rk(\hat x/A\hat t)<r$.
			\item There is some $x\in H_{\hat t}$ such that for each $i=1,...,r$ the pairs $(\hat x_0,\hat x_i)$ and $(x_0,x_i)$ are semi-indistinguishable in $\mathcal C$.
			\item There is some $i\geq 2$ such that $\hat x_i\in\operatorname{acl}_{\mathcal M}(A\hat t\hat x_0...\hat x_{i-1})$.
		\end{enumerate}
	\end{proposition}
	\begin{remark} Before giving the proof, we pause briefly to explain how (2) and (3) in the proposition statement are relevant toward achieving our goal. In the broadest sense, we are roughly trying to say that closure points of $X=H_{\hat t}$ resemble actual points of $X$. Now if (2) holds, we have found a strong correlation between the closure point $\hat x$ and the actual point $x$ of $X$. Meanwhile if (3) holds, we have shown that there exists a dependence among the coordinates of $\hat x$, and moreover that it happens \textit{after} the first two coordinates; that is, we have some control over specifically \textit{which} coordinates are dependent. This is the genesis of how we will control the pairwise independence structure among closure points in general.
	\end{remark}
	\begin{proof}
		For ease of notation, we will suppress the parameter set $A$ by adding it to the language. Thus we assume that $A=\emptyset$. 
		
		We will argue by contradiction. So, suppose that each of (1), (2), and (3) fails. Our strategy will be to study the pairwise intersections of $H_{\hat t}$ with a family of curves in $M^{r+1}$. Roughly speaking, we will choose a generic curve containing $\hat x$, and then show that perturbing this curve slightly results in a curve meeting $H_{\hat t}$ in strictly more points than the original. This will lead to an $\mathcal M$-formula relating $\hat t$ and the parameter defining our curve, which will give the necessary contradiction.
		
		So, we begin by defining our family of curves. The definition will involve the given excellent family $\mathcal C$.
		
		\begin{definition}
			Let $U=S^r$. Then we define the family $\mathcal D=\{D_u\}_{u\in U}$ as follows: for $u=(s_1,...,s_r)\in U$, we set $$D_u=\{(x_0,...,x_r)\in M^{r+1}:(x_0,x_i)\in C_{s_i}\textrm{for each }i=1,...,r\}.$$
		\end{definition}
		We thus immediately deduce:
		\begin{lemma} For each $u=(s_1,...,s_r)\in U$ we have $\rk(D_u)=1$.
		\end{lemma}
		\begin{proof}
			Since each $C_s$ is non-trivial, for each $i$ and $x_0$ there are only finitely many extensions of $x_0$ to an element of $C_{s_i}$. Thus the leftmost projection $D_u\rightarrow M$ is finite-to-one, which shows that $\rk(D_u)\leq 1$. On the other hand, also by non-triviality, if $x_0$ is generic over $u$ then it extends to an element of $C_{s_i}$ for each $i$. Thus the same projection $D_u\rightarrow M$ is almost surjective, and so $\rk(D_u)=1$.
		\end{proof}
		Since $S$ is generic in $M^2$, we have $\rk(U)=r\cdot\rk(S)=2r$. Thus, if $D\subset M^{r+1}\times U$ is the graph of $\mathcal D$, then $\rk(D)=2r+1$.
		
		Let $N$ be the set of normal points of $\mathcal C$, as in Definition \ref{common points}. So each fiber ${_aC}$ for $a\in N$ is a non-trivial plane curve. Now we can also show:
		
		\begin{lemma}\label{dual G fibers} For each $x=(x_0,...x_r)\in M^{r+1}$, the fiber ${_xD}=\{u\in U:x\in D_u\}$ has rank at most $r$. If each $(x_0,x_i)\in N$, for $i=1,...,n$, then $rk({_xD})=r$. In particular, for all generic $x\in M^{r+1}$ we have $\rk({_xD})=r$.
			\end{lemma}
		\begin{proof} Let $x=(x_0,...,x_r)\in M^{r+1}$. Then, by definition, ${_xD}$ consists of those $(s_1,...,s_r)$ such that each $s_i\in{_{(x_0,x_i)}C}$. Equivalently, ${_xD}$ is the product of the sets ${_{(x_0,x_i)}C}$ for $i=1,...,r$. Now ${_{(x_0,x_i)}C}$ is either of rank 1 or empty, depending on whether $(x_0,x_i)\in N$. This immediately gives that $\rk({_xD})\leq r$, with equality if each $(x_0,x_i)\in N$. But $N$ is both generic in $M^2$ and $\mathcal M$-definable over $\emptyset$. Thus, for generic $x\in M^{r+1}$, each $(x_0,x_i)\in N$, and so $\rk{(_xD)}=r$.
		\end{proof}
		
		Our goal will be to work with a particular curve $D_{\hat u}$ containing $\hat x$, and study intersections of `nearby' $D_u$ with $H_{\hat t}$. However, in doing this we want to use the inductive hypothesis on $\hat u$ (and some set of `nearby $u$'s'). Now since $\rk(U)=2r$, while the inductive hypothesis only covers sets of rank $<r$, we are not well-equipped to handle arbitrary subsets of $U$ in the way we want. In order to circumvent this, we will simultaneously work with both $\mathcal D$ and a fixed rank $r$ subfamily $\mathcal Y$ of $\mathcal D$:
		
		\begin{notation}
			We define the following:
			\begin{enumerate}
				\item Let $\hat p=(\hat p_1,...,\hat p_r)$ be a fixed generic element of $M^r$ over $\hat t\hat x$.
				\item Let $U_{\hat p}$ be the set of all $u=((p_1,q_1),...,(p_r,q_r))\in B$ such that $p_i=\hat p_i$ for each $i=1,...,r$.
				\item Let $\mathcal Y$ be the subfamily of $\mathcal D$ indexed by $U_{\hat p}$: that is, the family of curves with graph $Y\subset M^{r+1}\times U_{\hat p}$ given by restricting $D$ to $M^{r+1}\times U_{\hat p}$.
				\end{enumerate}
		\end{notation}
	
	That is, $U_{\hat p}$ is obtained by choosing a fixed generic left coordinate for each $s_i$ in a tuple $u\in U$. Since $U$ is generic in $M^{2r}$, and $\hat p$ is generic, it follows that we can identify $U_{\hat p}$ with a generic subset of $M^r$ (by identifying a tuple $u=(s_1,...,s_r)$ with the sequence of right coordinates of the $s_i$'s). It follows immediately that $\rk(U_{\hat p})=r$, and thus $\rk(Y)=r+1$. We conclude:
	
	\begin{lemma}\label{H finite-to-one} The projection $Y\rightarrow M^{r+1}$ is finite-to-one and almost surjective.
	\end{lemma}
	\begin{proof}
		Let $x=(x_0,...,x_r)\in M^{r+1}$. By definition, the fiber ${_xY}={_xD}\cap U_{\hat p}$ consists of all points $((\hat p_1,q_1),...,(\hat p_r,q_r))$ with each $(\hat p_i,q_i)\in{_{(x_0,x_i)}C}$. Since each ${_aC}$ is non-trivial, each $\hat p_i$ has only finitely many extensions to a pair $(\hat p_i,q_i)\in{_{(x_0,x_i)}C}$. This shows that $Y\rightarrow M^{r+1}$ is finite-to-one. Almost surjectivity then follows by rank considerations, since $\rk(Y)=\rk(M^{r+1})$ and $M^{r+1}$ is stationary.
	\end{proof}

	Since $\hat x$ is generic in $M^{r+1}$, and $\hat p$ is generic in $M^r$ over $\hat x$, it follows by symmetry that $\hat x$ is generic in $M^{r+1}$ over $\hat p$. In particular, by Lemma \ref{H finite-to-one} the fiber ${_{\hat x}Y}$ is non-empty. The following is now justified:
	
	\begin{notation}
		We fix $\hat u=(\hat s_1,...,\hat s_r)$ to be any element of the fiber ${_{\hat x}Y}$. For each $i$, we denote $\hat s_i=(\hat p_i,\hat q_i)$.
	\end{notation}

We immediately get the following two lemmas:

	\begin{lemma}\label{b gen} $\hat u$ is generic in $U_{\hat p}$ over $\hat p$, and is moreover generic in $U$ over $\emptyset$.
	\end{lemma}
	\begin{proof} As stated above, $\hat x$ is generic in $M^{r+1}$ over $\hat p$. Thus $(\hat x,\hat u)$ is generic in $Y$ over $\hat p$, by Lemma \ref{H finite-to-one} and dimension considerations. In then follows by Lemma \ref{generic dominance facts}(4) that the projection $Y\rightarrow U_{\hat p}$ is generically dominant, which gives that $\hat u$ is generic in $U_{\hat p}$ over $\hat p$. 
		
		So we have shown the first clause of the lemma. Now this in particular gives that $\dim(\hat u/\hat p)=r\cdot\dim M$. But by definition $\dim(\hat p/\hat u)=0$ and $\dim(\hat p)=r\cdot\dim M$, so in fact $$\dim(\hat u)=\dim(\hat u\hat p)=2r\cdot\dim M=\dim U,$$ which shows the second clause. 
		\end{proof}

	\begin{lemma}\label{b gen over x}
		$\hat u$ is generic in ${_{\hat x}D}$ over $\hat t\hat x$. In particular, $\rk(\hat u/\hat t\hat x)=r$ and $\hat u$ is coherent over $\hat t\hat x$.
	\end{lemma}
	\begin{proof} Since each ${_aC}$ is non-trivial, it follows that $\hat u$ is $\mathcal M$-interalgebraic with $\hat p$ over $\hat t\hat x$. So the lemma follows since by definition $\hat p$ is coherent and of rank $r$ over $\hat t\hat x$.
		\end{proof}
	
	Now that we have $D_{\hat u}$, our rough goal is to study the intersection points with $H_{\hat t}$ that one obtains by perturbing $\hat u$. Our conclusion will eventually be that $\hat u$ is not $\mathcal M$-generic in $U_{\hat p}$ over $\hat t\hat p$, because it belongs to the closure of the $u\in U_{\hat p}$ such that $D_u$ intersects $H_{\hat t}$ in more than the generic number of points (that is, such that $|H_{\hat t}\cap D_u|>|H_{\hat t}\cap D_{u'}|$ for generic $u'\in U_{\hat p}$ over $\hat t\hat p$). Indeed, this observation will ultimately lead to a contradiction that proves Proposition \ref{frontier proof}. The reason is the following:
	
	\begin{lemma}\label{b gen over a} $\hat u$ is both $\mathcal M$-generic in $U$ over $\hat t$, and $\mathcal M$-generic in $U_{\hat p}$ over $\hat t\hat p$.
	\end{lemma}
	\begin{proof} For each $i=1,...,r$, let $\hat a_i=(\hat x_0,\hat x_i)$. Since $\hat x$ is generic in $M^{r+1}$, each $\hat a_i\in N$. We will also use the notation $\hat p_{<i}$ for $(\hat p_1,...,\hat p_{i-1})$, and similarly for $\hat a_{<i}$ and $\hat t_{<i}$. Now the main observation is:
		\begin{claim}\label{s curve is generic} For each $i=1,...,r$ we have $\rk(\hat a_i/\hat t,\hat p_{<i},\hat a_{<i},\hat s_{<i})\geq 1$.
	\end{claim}
	\begin{proof}
		Since each $\hat s_j=(\hat p_j,\hat q_j)$ belongs to the non-trivial plane curve ${_{\hat a_j}C}$, it follows that $\hat s_j$ and $\hat p_j$ are $\mathcal M$-interalgebraic over $\hat a_j$. Thus it is equivalent to show that $\rk(\hat a_i/\hat t,\hat p_{<i},\hat a_{<i})\geq 1$. By definition $\hat p$ is $\mathcal M$-independent from $(\hat t,\hat x)$, so it is moreover equivalent to show that $\rk(\hat a_i/\hat t,\hat a_{<i})\geq 1$. 
		
		Now suppose otherwise: that is, suppose that $\hat a_i\in\operatorname{acl}_{\mathcal M}(\hat t\hat a_1...\hat a_{i-1})$. We will show that one of (1) and (3) from the proposition statement happens, contradicting our assumption. We consider two cases:
		
		\begin{itemize}
			\item First suppose that $i=1$. Then $\hat a_1=(\hat x_0,\hat x_1)\in\operatorname{acl}_{\mathcal M}(\hat t)$. Since there are only $r-1$ remaining coordinates in the tuple $\hat x$, it follows that $\rk(\hat x/\hat t)\leq r-1$, so we get (1).
		
			\item Now suppose that $i\geq 2$. In this case note that $(\hat a_1,...,\hat a_{i-1})$ is $\mathcal M$-interdefinable with $(\hat x_0,...,\hat x_{i-1})$. It then follows that $\hat a_i=(\hat x_0,\hat x_i)\in\operatorname{acl}_{\mathcal M}(\hat t\hat x_0...\hat x_{i-1})$, so in particular we get (3).
		\end{itemize}
	\end{proof}
		We also show:
		\begin{claim}\label{t generic in s} For each $i=1,...,r$ we have $\rk(\hat s_i/\hat t,\hat a_i,\hat p_{<i},\hat a_{<i},\hat s_{<i})=1$.
		\end{claim}
		\begin{proof} As in Claim \ref{s curve is generic}, each $\hat s_j$ is $\mathcal M$-interalgebraic with $\hat p_j$ over $\hat a_j$. Thus it is equivalent to show that $\rk(\hat p_i/\hat t,\hat a_i,\hat p_{<i},\hat a_{<i})\geq 1$ (note that we can switch $\hat s_i$ for $\hat p_i$ because this time $\hat a_i$ is in the base). But this is automatic, since the tuple $\hat p$ is generic in $M^r$ over $\hat t\hat x$.
		\end{proof}
		Using the previous two claims, we conclude:
		\begin{claim}\label{each t gen} For each $i=1,...,r$ we have $\rk(\hat s_i/\hat t,\hat p_{<i},\hat a_{<i},\hat s_{<i})=2$. 
		\end{claim}
		\begin{proof} By Claim \ref{s curve is generic}, there is a set $N_i$ which is $\mathcal M$-definable over $(\hat t,\hat p_{<i},\hat a_{<i},\hat s_{<i})$, has rank at least 1, and contains $\hat a_i$ as an $\mathcal M$-generic element (over the same parameters). Since $\hat a_i\in N$, we may assume without loss of generality that $N_i\subset N$. We then view $N_i$ as a parameter space for a positive rank subfamily $\mathcal C_i$ of the dual family $\mathcal C^{\vee}=\{{_aC}\}_{a\in N}$. So $\mathcal C_i$ is almost faithful, because $\mathcal C^{\vee}$ is. 
			
		So, over the base tuple $(\hat t,\hat p_{<i},\hat a_{<i},\hat s_{<i})$, the curve ${_{\hat a_i}C}$ is an $\mathcal M$-generic member of the positive rank almost faithful family $\mathcal C_i$; and by Claim \ref{t generic in s}, the point $\hat s_i$ is an $\mathcal M$-generic element of this curve over $\hat a_i$ (and the same base parameters). By for example the remarks after Lemma \ref{family ranks}, this implies that $\hat s_i$ is $\mathcal M$-generic in $M^2$ over $(\hat t,\hat p_{<i},\hat a_{<i},\hat s_{<i})$, which is enough. 
	\end{proof}
	We now finish the proof of Lemma \ref{b gen over a}. By Claim \ref{each t gen}, we have in particular that $\rk(\hat s_i/\hat t,\hat s_{<i})=2$ for each $i$. It follows that $\rk(\hat u/\hat t)=2r$, which shows that $\hat u$ is $\mathcal M$-generic in $U$ over $\hat t$. Now since by definition $\rk(\hat p/\hat u)=0$ and $\rk(\hat p/\hat t)=r$, it follows easily that $\rk(\hat u/\hat t\hat p)=r$. Thus $\hat u$ is also $\mathcal M$-generic in $U_{\hat p}$ over $\hat t\hat p$.
	
\end{proof}

	Before proceeding, let us note the following easy consequence of Lemma \ref{b gen over a}:
	
	\begin{lemma}\label{ab finite intersection} $H_{\hat t}\cap D_{\hat u}$ is finite.
		\end{lemma}
	\begin{proof} Let $x\in H_{\hat t}\cap D_{\hat u}$. It will suffice to show that $\rk(x/\hat t\hat u)=0$. Now we have $\rk(x/\hat t)\leq r$, since $x\in H_{\hat t}$. Moreover, by Lemma \ref{dual G fibers} we have $\rk(\hat u/x)\leq r$. It follows that $\rk(x\hat u/\hat t)\leq 2r$. On the other hand, Lemma \ref{b gen over a} gives that $\rk(\hat u/\hat t)=2r$. Thus we are forced to conclude that $\rk(x/\hat t\hat u)=0$.  
		\end{proof}
	
	The remainder of the proof of Proposition \ref{frontier proof} will be devoted to showing that Lemma \ref{b gen over a} is false. That is, we will obtain a contradiction by showing that $\hat u$ is in fact not $\mathcal M$-generic in $U_{\hat p}$ over $\hat t\hat u$. 
	
	To begin, we define the set $$I=\{(x,t,u)\in M^{r+1}\times T\times U_{\hat p}:x\in H_t\cap D_u\}$$ -- that is, the graph of the `family of intersections' of $\mathcal H$ and $\mathcal Y$.
	
	The main step in the rest of the proof will be showing that, for each $x\in H_{\hat t}\cap D_{\hat u}$, the projection $I\rightarrow T\times U_{\hat p}$ is analytically open near $(x,\hat t,\hat u)$. For now, let us show why this is enough to reach a contradiction:
	
	\begin{lemma}\label{open is enough} Assume that for each $x\in H_{\hat t}\cap D_{\hat u}$ the projection $I\rightarrow T\times U_{\hat p}$ is analytically open near $(x,\hat t,\hat u)$. Then $\hat u$ is not $\mathcal M$-generic in $U_{\hat p}$ over $\hat t\hat p$.
		\end{lemma}
	\begin{proof} Let $w_1,...,w_l$ be the distinct intersection points of $H_{\hat t}$ and $D_{\hat u}$ ($l$ is finite by Lemma \ref{ab finite intersection}). We first make a couple of easy remarks:
		
	Recall first that $(\hat x,\hat u)$ is generic in $Y$ over $\hat p$ (for example, see the proof of Claim \ref{b gen}). So by Lemma \ref{H finite-to-one} and Corollary \ref{generic openness}, the projection $Y\rightarrow M^{r+1}$ is analytically open near $(\hat x,\hat u)$.
		
	Note also that we are assuming $\hat x\notin H_{\hat t}$ -- indeed, otherwise condition (2) of Proposition \ref{frontier proof} would trivially hold with $x=\hat x$ (to be precise, (2) would follow from the fact that each ${_{(\hat x_0,\hat x_i)}C}$ is infinite, which holds by the genericity of $\hat x$). Thus the elements $\hat x,w_1,...,w_l$ are distinct.
	
	Now let $Z$ be the set of all $u\in U_{\hat p}$ such that $H_{\hat t}\cap D_u$ is not of size $l$. So $Z$ is $\mathcal M$-definable over $\hat t\hat p$. Then the main point is the following, which is arguably the most important of the analytic topology in this paper:
	
	\begin{claim}\label{closure of big intersections} $\hat u\in\overline Z$.
	\end{claim}
	\begin{proof} Let $V$ be any analytic neighborhood of $\hat u$ in $U_{\hat p}$. We will find some $u\in V\cap Z$.
		
	First, using the analytic openness of $Y\rightarrow M^{r+1}$ near $(\hat x,\hat u)$ and $I\rightarrow T\times U_{\hat p}$ near each $(w_i,\hat t,\hat u)$, and after applying various shrinkings, we can choose sets $W_0,...,W_l$ such that:
		\begin{enumerate}
			\item $W_0$ is an analytic neighborhood of $\hat x$ in $M^{r+1}$, and for each $i\geq 1$ $W_i$ is an analytic neighborhood of $w_i$ in $M^{r+1}$.
			\item The sets $W_0,...,W_l$ are pairwise disjoint.
			\item For each $x\in W_0$ there is some $u\in V$ such that $(x,u)\in Y$.
			\item For each $u\in V$ and $i\leq l$ there is some $x\in W_i$ with $(x,\hat t,u)\in I$.
		\end{enumerate}
	Now using (1) and the fact that $\hat x\in\overline{H_{\hat t}}$, we can choose $x\in W_0\cap H_{\hat t}$, and then choose $u$ as in (3). Then by (4), the intersection $H_{\hat t}\cap D_u$ contains $x$, in addition to one point in each $W_i$. By (2) all of these points are distinct, so $|H_{\hat t}\cap D_u|\geq l+1$, and thus $u\in V\cap Z$.
	\begin{remark} Note that the proof of Claim \ref{closure of big intersections} critically uses that the analytic topology is Hausdorff, in order to guarantee that the $l+1$ intersection points of $H_{\hat t}$ and $D_u$ are distinct.
	\end{remark}
	
		
		\end{proof}

	Continuing with the proof, we also point out:
	
	\begin{claim}\label{IH applies} $\hat u$ is coordinate-wise generic.
	\end{claim}
	\begin{proof} This follows from Lemma \ref{b gen}. Indeed, we have that $\hat u$ is generic in $U$; but since $U$ is itself a generic subset of $M^{2r}$, it follows immediately that $\hat u$ is moreover generic in $M^{2r}$. In particular, $\hat u$ is certainly coordinate-wise generic.
	\end{proof}
	
	Now by Claims \ref{closure of big intersections} and \ref{IH applies}, Proposition \ref{closure prop} should apply to $\hat u\in\overline Z$, with parameter tuple $\hat t\hat p$. In particular, if $\rk(Z)<r$, then the inductive hypothesis will apply. Using this observation, we can easily finish the proof of Lemma \ref{open is enough} in two cases:
	
	\begin{itemize}
		\item First suppose that $Z$ is generic in $U_{\hat p}$. Recall that $U_{\hat p}$ is naturally identified with a generic subset of $M^r$; in particular, $U_{\hat p}$ is stationary. We thus conclude that $U_{\hat p}-Z$ is non-generic in $U_{\hat p}$. But by definition $\hat u\in U_{\hat p}-Z$, so it follows immediately that $\hat u$ is not $\mathcal M$-generic in $U_{\hat p}$ over $\hat t\hat p$.
		\item Now suppose that $Z$ is non-generic in $U_{\hat p}$. So $\rk(Z)<r$. By the inductive hypotheses, and since $\hat u\in\overline Z$, we conclude that $$\rk(\hat u/\hat t\hat p)\leq\rk(Z)<r,$$ so that again $\hat u$ is not $\mathcal M$-generic in $U_{\hat p}$ over $\hat t\hat p$.
	\end{itemize}
	The proof of Lemma \ref{open is enough} is now complete.
	\end{proof} 
	So it remains to show that the projection $I\rightarrow T\times U_{\hat p}$ is analytically open near each intersection point of $H_{\hat t}$ and $D_{\hat u}$. We thus fix, for the rest of the proof of Proposition \ref{frontier proof}, an element $x=(x_0,...,x_r)\in H_{\hat t}\cap D_{\hat u}$. Our strategy will be to verify each of the hypotheses of Proposition \ref{smooth fiber product}(4), with $X=M^{r+1}$, $Y_1=T$, $Y_2=U_{\hat p}$, $Z_1=H$, $Z_2=Y$, $A_1=\emptyset$, and $A_2=\{\hat p\}$. Thus our main remaining task is to show that $(x,\hat t)$ and $(x,\hat u)$ are generic in $H$ and $Y$ over $\emptyset$ and $\hat p$, respectively -- or equivalently that $x$ is generic in $H_{\hat t}$ over $\hat t$ and in $D_{\hat u}$ over $\hat p\hat u$. We turn now toward this goal. Unfortunately, establishing the required genericities will be quite complicated -- in particular, we will need to undergo a hefty rank computation. The first of many steps is the following:
	\begin{lemma}\label{j=0} For each $i=1,...,r$, the pairs $(\hat x_0,\hat x_i)$ and $(x_0,x_i)$ are not semi-indistinguishable in $\mathcal C$.
	\end{lemma}
	\begin{proof} Assume the lemma is false. Let $J\subset\{1,...,r\}$ be the (non-empty by assumption) set of $i$ such that these pairs are semi-indistinguishable in $\mathcal C$, and let $j=|J|$. Note that by assumption $J$ cannot be all of $\{1,...,r\}$ -- indeed, otherwise we would have (2) in the statement of Proposition \ref{frontier proof}. We thus have $0<j<r$.
		
	Let $\hat x_J$ be the tuple of $\hat x_i$'s for $i\in J$, and similarly let $x_J$ be the tuple of $x_i$'s for $i\in J$. We note:
	
	\begin{claim}\label{ind locus interalg} The tuples $\hat x_0\hat x_J$ and $x_0x_J$ are $\mathcal M$-interalgebraic over $\emptyset$.
		\end{claim}
	\begin{proof} For each $i\in J$, it follows from Lemma \ref{indistinguishable interalgebraic} that $(\hat x_0,\hat x_i)$ and $(x_0,x_i)$ are $\mathcal M$-interalgebraic over $\emptyset$. The claim follows immediately.
		\end{proof}
	
	We pause to point out a subtle feature of Claim \ref{ind locus interalg} -- namely, the claim would fail if $j=0$. Indeed, in this case there would be no obvious way to connect $\hat x_0$ to $x_0$ without using the $\mathcal M$-interalgebraicity of $(\hat x_0,\hat x_i)$ and $(x_0,x_i)$ for some $i\geq 1$. This small point is quite crucial: indeed, it is the only use we will make of the assumption that $j>0$.
	
	Continuing, we note the following:
	
	\begin{claim}\label{x_J at least j} $\rk(x_0x_J/\hat t)\geq j$.
	\end{claim}
	\begin{proof} By Claim \ref{ind locus interalg}, it is equivalent to show that $\rk(\hat x_0\hat x_J/\hat t)\geq j$. But if this were not true, then we would have a $j+1$-subtuple of $\hat x$ -- namely $\hat x_0\hat x_J$ -- of rank at most $j-1$ over $\hat t$. It would follow that $\hat x$ has corank at least 2 over $\hat t$, or more precisely that $\rk(\hat x/\hat t)\leq(r+1)-2=r-1$. We thus obtain (1) in the statement of Proposition \ref{frontier proof}, contradicting our assumption.
	\end{proof}

	The following consequence of Claim \ref{x_J at least j} is crucial, but quite subtle; it is in fact the only use we will make of the non-triviality of $H_{\hat t}$:

	\begin{claim}\label{the use of non-triviality} $\rk(x/\hat t\hat x)<r-j$.
		\end{claim}
	\begin{proof}
		Toward a contradiction, assume that $\rk(x/\hat t\hat x)\geq r-j$. Since $\hat x_0\hat x_J$ is a subtuple of $\hat x$, it follows that $\rk(x/\hat t\hat x_0\hat x_J)\geq r-j$. Equivalently, by Claim \ref{ind locus interalg}, we get that $\rk(x/\hat tx_0x_J)\geq r-j$. Now combining with Claim \ref{x_J at least j}, we conclude that $\rk(x/\hat t)\geq r$, or in other words that $x$ is $\mathcal M$-generic in $H_{\hat t}$ over $\hat t$.
		
		As hinted at above, we now obtain a contradiction by using the fact that $H_{\hat t}$ is non-trivial. Precisely, since $x$ is $\mathcal M$-generic in $H_{\hat t}$ over $\hat t$, it follows by non-triviality that any $r$ coordinates of $H_{\hat t}$ are $\mathcal M$-generic and $\mathcal M$-independent in $M$ over $\hat t$ -- and thus the same automatically holds for all $r'\leq r$. Now since we are assuming $j<r$, the $j+1$-subtuple $x_0x_J$ of $x$ has length at most $r$. In particular, we obtain $\rk(x_0x_J/\hat t)=j+1$. But again, from the previous paragraph we have $\rk(x/\hat tx_0x_J)\geq r-j$. By additivity, we thus obtain $\rk(x/\hat t)\geq r+1$. This is a contradiction since $x\in H_{\hat t}$.
	\end{proof}

	The following claim is the last ingredient in the proof of Lemma \ref{j=0}:
	\begin{claim}\label{b over x at most ind locus}
		$\rk(\hat u/\hat xx)\leq j$.
	\end{claim}
	\begin{proof} Note that $\hat u$ belongs to the set $L$ of those $u\in U$ such that $D_u$ contains both $\hat x$ and $x$. By definition $L$ is the product, over $i=1,...,r$, of the sets $L_i={_{(\hat x_0,\hat x_i)}C}$. Now by the excellence of the family $\mathcal C$, we have $\rk(L_i)\leq 1$ for each $i$, and by definition equality holds if and only if $i\in J$. Thus there are only $j$ values of $i$ for which $\rk(L_i)=1$, and the rest of the $L_i$ are finite. This gives that $\rk(L)=j$, which implies the claim.
		\end{proof}
	
	We can now finish the proof of Lemma \ref{j=0}. By combining Claims \ref{the use of non-triviality} and \ref{b over x at most ind locus}, we obtain that $\rk(x\hat u/\hat t\hat x)<r$. In particular, $\rk(\hat u/\hat t\hat x)<r$. But this directly contradicts Lemma \ref{b gen over x}.
	\end{proof}
	As an immediate corollary of Lemma \ref{j=0}, we obtain:
	\begin{lemma}\label{b alg over the x's} $\hat u\in\operatorname{acl}_{\mathcal M}(\hat xx)$.
	\end{lemma}
	\begin{proof} We repeat the statement and proof of Claim \ref{b over x at most ind locus}. However this time we know that $j=0$ by Lemma \ref{j=0}. The result is that $\rk(\hat u/\hat xx)=0$.
		\end{proof}
	We can now deduce the first of our two desired genericity statements:
	\begin{lemma}\label{x gen over a} $x$ is generic in $H_{\hat t}$ over $\hat t\hat x$, and thus also over $\hat t$.
	\end{lemma}
	\begin{proof}
		Combining Lemmas \ref{b alg over the x's} and \ref{ab finite intersection} gives that $x$ and $\hat u$ are $\mathcal M$-interalgebraic over $\hat t\hat x$. By Lemma \ref{b gen over x}, we conclude that $x$ is coherent of rank $r$ over $\hat t\hat x$. Since $x$ belongs to $H_{\hat t}$, and $\rk(H_{\hat t})=r$, we equivalently get that $x$ is generic in $H_{\hat t}$ over $\hat t\hat x$. Of course this implies that $x$ is generic in $H_{\hat t}$ over $\hat t$ as well.
	\end{proof}
	It will be non-trivial to get from Lemma \ref{x gen over a} to the genericitiy of $x$ in $D_{\hat u}$. The main step is the following, which says that we can transfer the genericity of $x$ over $(\hat t,\hat x)$ to genericity in the ambient space over $\hat x$.
	\begin{lemma}\label{sweeping} $x$ is generic in $M^{r+1}$ over $\hat x$.
	\end{lemma}
	\begin{proof} Recall that $k=\rk\mathcal H>(r+1)\cdot\dim M$ by assumption. We let $w_1,...,w_k$ be independent realizations of the $\mathcal K$-type of $x$ over $\hat t\hat x$ -- so in particular $w_1,...,w_k$ are independent generics in $H_{\hat t}$ over $\hat t\hat x$. Additionally, let $t$ be an independent realization of the $\mathcal K$-type of $\hat t$ over $(w_1,...,w_k)$. We then show:
		\begin{claim}\label{independent hypersurface alg} $\rk(\hat t,w_1,...,w_k)=\rk(\hat t,t,w_1,...,w_k)=k(r+1)$.
		\end{claim}
		\begin{proof} Since $\hat t$ is generic in $T$, and $w_1,...,w_k$ are independent generics in $H_{\hat t}$ over $\hat t$, we obtain $\rk(\hat t,w_1,...,w_k)=k+kr=k(r+1)$. Thus it suffices to show that $\rk(\hat t,t,w_1,...,w_k)\leq k(r+1)$. We have two cases:
		\begin{itemize}
			\item First suppose $\rk(H_{\hat t}\cap H_t)=r$. Then, by the almost faithfulness of $\mathcal H$, $t\in\operatorname{acl}_{\mathcal M}(\hat t)$. It follows that $$\rk(\hat t,t,w_1,...,w_k)=\rk(\hat t,w_1,...,w_k)=k(r+1),$$ as desired.
			\item Now suppose $\rk(H_{\hat t}\cap H_t)\leq r-1$. Then, counting in order, we have $$\rk(\hat t,t,w_1,...,w_k)\leq k+k+k(r-1)=k(r+1),$$ again as desired.
		\end{itemize}
		\end{proof}
		Now by Claim \ref{independent hypersurface alg}, and the definition of $t$, it follows that in fact $\hat t\in\operatorname{acl}_{\mathcal M}(w_1,...,w_k)$, so that $$\rk(w_1,...,w_k)=\rk(\hat t,w_1,...,w_k)=k(r+1).$$ Thus $w_1,...,w_k$ are $\mathcal M$-independent $\mathcal M$-generics in $M^{r+1}$ over $\emptyset$. We moreover conclude:
		\begin{claim}\label{w's are generic} $w_1,...,w_k$ are independent generics (that is, in the sense of $\mathcal K$) in $M^{r+1}$ over $\emptyset$.
		\end{claim}
		\begin{proof} By the preceding remarks and Lemma \ref{coherent iff generic}, it suffices to show that $(w_1,...,w_k)$ is coherent. But by Lemma \ref{coherent iff generic} again, and since $w_1,...,w_k$ are independent generics in $H_{\hat t}$, it is clear that $(w_1,...,w_k)$ is coherent over $\hat t$. Moreover, $\hat t$ is coherent over $\emptyset$ for the same reason, since it is generic in $T$. The coherence of $(w_1,...,w_k)$ then follows by repeated applications of Lemma \ref{coherent preservation}.
			\end{proof}
		Now toward a proof of Lemma \ref{sweeping}, assume $x$ is not generic in $M^{r+1}$ over $\hat x$. By definition of the $w_i$'s, each $w_i$ is therefore non-generic in $M^{r+1}$ over $\hat x$. In particular, the tuple $(\hat x,w_1,...,w_k)$ has codimension at least $k$ in $M^{(k+1)(r+1)}$. But on the other hand, Claim \ref{w's are generic} gives that $(w_1,...,w_k)$ is generic in $M^{k(r+1)}$. Thus the codimension of $(\hat x,w_1,...,w_k)$ in $M^{(k+1)(r+1)}$ is at most $\dim\hat x=(r+1)\cdot\dim M$. We conclude in particular that $k\leq(r+1)\cdot\dim M$, contradicting our assumption on $k$. 
		\end{proof}
	Finally, we conclude:
	\begin{lemma}\label{x gen over b} $x$ is generic in $D_{\hat u}$ over $\hat u\hat x$, and thus also over $\hat p\hat u$.
	\end{lemma}
	\begin{proof} For ease of notation, let us temporarily denote $m=\dim M$. Now by Lemma \ref{sweeping} we have $\dim(\hat xx)=2m(r+1)$, and thus by Lemma \ref{b alg over the x's} we also have $\dim(\hat u\hat xx)=2m(r+1)$. Now $\hat u$ is generic in $U$ by Lemma \ref{b gen}. So we have $\dim(\hat u)=2mr$, and thus by additivity we get $\dim(\hat xx/\hat u)=2m$. Since $\hat x,x\in D_{\hat u}$ and $\dim(D_{\hat u})=m$, this is only possible if $\hat x$ and $x$ are independent generics in $D_{\hat u}$. In particular, $x$ is generic in $D_{\hat u}$ over $\hat u\hat x$. Then since $\hat p$ is definable over $\hat u$, it follows that $x$ is also generic in $D_{\hat u}$ over $\hat p\hat u$.
		\end{proof}
	
	Rewriting our two genericity statements, we obtain:
	
	\begin{lemma}\label{xa and xb gen} $(x,\hat t)$ is generic in $H$ over $\emptyset$, and $(x,\hat u)$ is generic in $Y$ over $\hat p$.
	\end{lemma}
	\begin{proof} Each of $H$ and $Y$ is a family of sets of the same dimension. Moreover, we already showed in Lemmas \ref{x gen over a} and \ref{x gen over b} that $x$ is generic in its respective set in each family (over the relevant parameters in each case); so it would suffice to know that $\hat t$ and $\hat u$ are generic in $T$ and $U_{\hat p}$ over $\emptyset$ and $\hat p$, respectively. But this follows for $\hat t$ by definition, and for $\hat u$ by Lemma \ref{b gen}.
		\end{proof}
	
	We are almost done -- there are just a couple remaining things to check before we can apply Proposition \ref{smooth fiber product}(4). Let us clarify these now:
	
	\begin{lemma}\label{dim R is as expected} $\dim H+\dim Y=\dim T+\dim U+\dim(M^{r+1})$.
	\end{lemma}
		\begin{proof} By Lemma \ref{dim rk equality}, it is enough to show the same equality with rank instead of dimension. One then easily computes that $\rk(H)=k+r$, $\rk(Y)=r+1$, $\rk(T)=k$, $\rk(U)=r$, and $\rk(M^{r+1})=r+1$, and plugging these values shows that the equality is indeed satisfied, with both sides equal to $k+2r+1$.
		\end{proof}

	\begin{lemma}\label{generic dominance for openness} Each of the projections $H\rightarrow M^{r+1}$, $H\rightarrow T$, $Y\rightarrow M^{r+1}$, and $Y\rightarrow U_{\hat p}$ is generically dominant.
	\end{lemma}
	\begin{proof}
		For $H\rightarrow T$ and $Y\rightarrow U_{\hat p}$ this follows from Lemma \ref{generic dominance facts}(4). For $H\rightarrow M^{r+1}$ this is Lemma \ref{hypersurfaces finitely many common points}(2), and for $Y\rightarrow M^{r+1}$ this follows from Lemma \ref{H finite-to-one}.
	\end{proof}

	Finally we can now conclude:

	\begin{lemma}
		The projection $I\rightarrow T\times U_{\hat p}$ is analytically open near $(x,\hat t,\hat u)$.
	\end{lemma}
	\begin{proof} By Proposition \ref{smooth fiber product}(4) with $X=M^{r+1}$, $Y_1=$T, $Y_2=U_{\hat p}$, $Z_1=H$, $Z_2=Y$, $A_1=\emptyset$, and $A_2=\{\hat p\}$. That the required hypotheses are met is given in Lemmas \ref{ab finite intersection}, \ref{xa and xb gen}, \ref{dim R is as expected}, and \ref{generic dominance for openness}.
	\end{proof}

We are now done considering out special point $x$. To sum up our work we have shown:

\begin{lemma}\label{the projection is open} The projection $I\rightarrow T\times U_{\hat p}$ is analytically open near $(x,\hat t,\hat u)$ for each $x\in H_{\hat t}\cap D_{\hat u}$.
	\end{lemma} 

Finally, we have finished the proof of Proposition \ref{frontier proof}. Indeed,  
the combination of Lemmas \ref{b gen over a}, \ref{open is enough}, and \ref{the projection is open} is contradictory.
\end{proof}
	\subsection{Generic Non-trivial $r$-Hypersurfaces}
	
	In this subsection we give the first complete case of Proposition \ref{closure prop}. Namely, we show that in the context of Proposition \ref{frontier proof}, Proposition \ref{frontier proof} actually implies Proposition \ref{closure prop}:
	
	\begin{proposition}\label{generic non-trivial hypersurfaces}
		Let $\mathcal H=\{H_t:t\in T\}$ be an almost faithful family of non-trivial $r$-hypersurfaces of rank $k>(r+1)\cdot\dim M$, and assume $\mathcal H$ is $\mathcal M$-definable over a set $A$. Let $\hat t\in T$ and $\hat x=(\hat x_0,...,\hat x_r)\in M^{r+1}$ each be generic over $A$, and assume that $\hat x\in\overline{H_{\hat t}}$. Then:
		\begin{enumerate}
			\item $\rk(\hat x/A\hat t)\leq r$.
			\item If $\rk(\hat x/A\hat t)=r\geq 2$ then for all $i\neq j\in\{0,...,r\}$ we have $\rk(\hat x_i\hat x_j/A\hat t)=2$.
		\end{enumerate}
	\end{proposition}
	\begin{proof} As in the proof of Proposition \ref{frontier proof}, we assume that $A=\emptyset$. Let $\mathcal C=\{C_s:s\in S\}$ be an excellent family of plane curves which is $\mathcal M$-definable over $\emptyset$, as provided by Assumption \ref{M and K}. We will prove both statements in the proposition by cases according to the conclusions of Proposition \ref{frontier proof} with $\mathcal H$, $\mathcal C$, $\hat t$, and $\hat x$. Namely, we show the following two claims:
		\begin{claim} $\rk(\hat x/\hat t)\leq r$.
		\end{claim}
		\begin{proof}
			We consider the three cases of Proposition \ref{frontier proof}:
			\begin{enumerate}
				\item Suppose $\rk(\hat x/\hat t)<r$. Then trivially $\rk(\hat x/\hat t)\leq r$.
				\item Suppose there is some $x=(x_0,...,x_r)\in H_{\hat t}$ such that for all $i=1,...,r$ the points $(\hat x_0,\hat x_i)$ and $(x_0,x_i)$ are semi-indistinguishable in $\mathcal C$. Then, by Lemma \ref{indistinguishable interalgebraic}, each of these pairs of points are $\mathcal M$-interalgebraic over $\emptyset$. In particular, it follows that $\hat x$ and $x$ are $\mathcal M$-interalgebraic over $\emptyset$ (as in the comments after Claim \ref{ind locus interalg}, we are technically using that $r\geq 1$ here). In particular, we obtain $\rk(\hat x/\hat t)=\rk(x/\hat t)$. But since $x\in H_{\hat t}$ we have $\rk(x/\hat t)\leq r$, so that $\rk(\hat x/\hat t)\leq r$.
				\item Suppose there is some $i\geq 2$ such that $x_i\in\operatorname{acl}_{\mathcal M}(\hat t\hat x_0...\hat x_{i-1})$. Then, in particular, it follows that $\hat x$ is not $\mathcal M$-generic in $M^{r+1}$ over $\hat t$. Equivalently, $\rk(\hat x/\hat t)\leq r$.
			\end{enumerate}	
		\end{proof}
		\begin{claim} Assume $\rk(\hat x/\hat t)=r\geq 2$, and let $i\neq j\in\{0,...,r\}$. Then $\rk(\hat x_i\hat x_j/\hat t)=2$.
		\end{claim}
		\begin{proof}
			Before applying Proposition \ref{frontier proof}, we may apply a coordinate permutation uniformly across the family $\mathcal H$. After doing this, it suffices to assume $i=0$ and $j=1$. So we need only show that $\rk(\hat x_0\hat x_1/\hat t)=2$. We do that, again, according to the cases of Proposition \ref{frontier proof}:
			
			\begin{enumerate}
				\item It is not possible to have $\rk(\hat x/\hat t)<r$, because we are assuming $\rk(\hat x/\hat t)=r$. So the first case is true vacuously.
				\item Suppose there is some $x\in H_{\hat t}$ such that for each $i=1,...,r$ the points $(\hat x_0,\hat x_i)$ and $(x_0,x_i)$ are semi-indistinguishable in $\mathcal C$. As in the previous claim, each of these pairs are $\mathcal M$-interalgebraic over $\emptyset$. In particular $(\hat x_0,\hat x_1)$ and $(x_0,x_1)$ are $\mathcal M$-algebraic over $\emptyset$. So we have $\rk(\hat x_0\hat x_1/\hat t)=\rk(x_0x_1/\hat t)$. Thus it will suffice to show that $\rk(x_0x_1/\hat t)=2$.
				
				Now we are assuming $\rk(\hat x/\hat t)=r$. As in the previous claim, our semi-indistinguishability assumption gives that $\hat x$ and $x$ are $\mathcal M$-interalgebraic over $\emptyset$, so that $\rk(x/\hat t)=r$ as well. This means that $x$ is $\mathcal M$-generic in $H_{\hat t}$ over $\hat t$. But $H_{\hat t}$ is non-trivial; so by definition any $r$ coordinates of $x$ are $\mathcal M$-generic and $\mathcal M$-independent over $\hat t$. In particular, since we are assuming $r\geq 2$, this gives $\rk(x_0x_1/\hat t)=2$, as desired.
				
				\item Suppose there is some $i\geq 2$ such that $\hat x_i\in\operatorname{acl}_{\mathcal M}(\hat t\hat x_0...\hat x_{i-1})$. Since $\hat x$ is an $r+1$-tuple and we are assuming $\rk(\hat x/\hat t)=r$, there can only be one $\mathcal M$-dependence among the coordinates of $\hat x$ over $\hat t$. In other words, $\hat x_i$ is the only coordinate of $\hat x$ which can be $\mathcal M$-algebraic over the previous coordinates and $\hat t$. In particular, this means there are no $\mathcal M$-dependencies among $\hat x_0$ and $\hat x_1$ over $\hat t$; so we conclude that $\rk(\hat x_0\hat x_1/\hat t)=2$, as desired.
			\end{enumerate} 
			\end{proof}
			By the two claims, the proof of Proposition \ref{generic non-trivial hypersurfaces} is complete.		
		\end{proof}

	\subsection{Stationary Non-trivial $r$-Hypersurfaces}
	The next step will be to generalize Proposition \ref{generic non-trivial hypersurfaces} by relaxing the genericity requirements on $\hat t$ and $\hat x$. Namely, we will prove Proposition \ref{closure prop} for \textit{any} stationary non-trivial hypersurface, without assuming it is generic in a large family; and we will consider \textit{all} coordinate-wise generic points, not just actually generic ones. This will be the last technically challenging step of the proof of Proposition \ref{closure prop} -- the remaining steps are much easier. Before proceeding with this step, we briefly outline the strategy:
	
	As in Proposition \ref{frontier proof}, we draw inspiration from the treatment in \cite{HK} and \cite{EHP} of frontiers of plane curves in strongly minimal groups. For example, the proof of the relevant result in \cite{EHP} begins with the case of a generic point and a generic curve in a large family, as in Proposition \ref{frontier proof}. Then, given an arbitrary point and curve, they show how to transform the setup (in a frontier preserving way) into an appropriately generic setup. This involves two stages. First, the point is transferred to a generic point using translation by a generic group element. Second, the curve is transferred to a suitably generic curve by `composing' it with a generic curve from a large family. It is important for the first step to precede the second: indeed, one needs to work with a generic (at least coordinate-wise) point to show that composition preserves the frontier.
	
	In our case, we are already presented with a coordinate-wise generic point; and we only need to preserve \textit{closure}, not frontier. These two factors, as it turns out, mean that we can get away with just one step. This step will more closely resemble the initial `translation' stage; however, instead of translating using group elements, we will translate using plane curves.
	
	Let us explain. Using the intuition of functions, one could imagine `evaluating' a non-trivial plane curve $C$ at a point $x\in M$: the result is the finite set $C(x)=\{y:(x,y)\in C\}$. We view the association $x\mapsto C(x)$ as the `translation' of $x$ by $C$. What we will do, then, is transform our set $X$ and closure point $\hat x$ through the `map' $(x_0,...,x_r)\mapsto(C_0(x_0),...,C_r(x_r))$, where $(C_0,...,C_r)$ is a `generic' sequence of plane curves. This is analogous to translating a tuple in $r+1$-space over a group by a generic tuple in $r+1$-space.
	
	Once we perform this `translation,' we will replace our original set and closure point, $\hat x\in\overline X$, with a `more generic' version $\hat y\in\overline Y$. Most of the work, then, is showing that $\hat y$ and $Y$ satisfy the hypotheses of Proposition \ref{generic non-trivial hypersurfaces} in an appropriate sense; after applying Proposition \ref{generic non-trivial hypersurfaces}, we then fairly easily transfer the result back to $\hat x$ and $X$.
	
	Let us proceed with the details. We use the following notion often, so we define it explicitly:
	
	\begin{definition}\label{*} Let $x=(x_0,...,x_r)\in M^{r+1}$ be a tuple, and $A$ a set of parameters. We say that $x$ \textit{satisfies} $*$ \textit{over} $A$, also denoted $*(x,A)$, if the following hold:
	\begin{enumerate}
		\item $\rk(x/A)\leq r$.
		\item If $\rk(x/A)=r\geq 2$ then for all $i\neq j\in\{0,...,r\}$ we have $\rk(x_ix_j/A)=2$.
	\end{enumerate}
	\end{definition}

	So the conclusion of Proposition \ref{generic non-trivial hypersurfaces} equivalently says that $\hat x$ satisfies $*$ over $A\hat t$. The following preservation properties are easy to verify: 
	
	\begin{lemma}\label{* preservation} Let $x$ and $y$ be tuples in $M^{r+1}$, and $A$ and $B$ sets of parameters.
		\begin{enumerate}
			\item If $x$ is $\mathcal M$-independent from $B$ over $A$, then $*(x,A)$ and $*(x,B)$ are equivalent.
			\item If $x$ and $y$ are coordinate-wise $\mathcal M$-interalgebraic over $A$, then $*(x,A)$ and $*(y,A)$ are equivalent.
			\item $*(x,A)$ is equivalent to $*(x,\operatorname{acl}_{\mathcal M}(A))$.
			\item If $*(x,A)$ and $B\supset A$ then $*(x,B)$.
		\end{enumerate}
	\end{lemma}
	\begin{proof}
		\begin{enumerate}
			\item $*$ only references the rank data of the coordinates of $x$ over $A$. If $x$ and $B$ are $\mathcal M$-independent over $A$, then adding $B$ to the parameter set does not affect this rank data.
			\item Clear, again since $*$ only reference the rank data of the coordinates of $x$ over $A$, and these ranks are preserved by $\mathcal M$-interalgebraicity.
			\item By (1), since $x$ is $\mathcal M$-independent from $\operatorname{acl}_{\mathcal M}(A)$ over $A$.
			\item By (1), we can assume $x$ and $B$ are dependent over $A$. So $$\rk(x/B)<\rk(x/A)\leq r,$$ and thus $*(x,B)$ holds trivially.
		\end{enumerate}
	\end{proof}
	
	Now the main result is:
	
	\begin{proposition}\label{stationary non-trivial hypersurfaces} Let $X$ be a stationary non-trivial $r$-hypersurface which is $\mathcal M$-definable over a set $A$. Let $\hat x=(\hat x_0,...,\hat x_r)$ be coordinate-wise generic, and assume that $\hat x\in\overline X$. Then $\hat x$ satisfies $*$ over $A$.
	\end{proposition}
	\begin{proof}
		We first define the curves we will use as translation maps:
		
		\begin{lemma}\label{translators} There are plane curves $C_0,...,C_r$, tuples $c_0,...,c_r$ from $\mathcal M^{\textrm{eq}}$, and a positive integer $k>(r+1)\cdot\dim M$, which satisfy the following:
			\begin{enumerate}
				\item Each $C_i$ is stationary and non-trivial.
				\item Each $C_i$ is $\mathcal M$-definable over $c_i$, and each $c_i=\operatorname{Cb}(C_i)$.
				\item The $c_i$ are independent over $\emptyset$, and each individual $c_i$ is coherent of rank $k$ over $\emptyset$.
				\item The tuple $(c_0,...,c_r)$ is independent from $A\hat x$ over $\emptyset$.
				\end{enumerate} 
			\end{lemma}
		\begin{proof} By Fact \ref{characterization of local modularity}, there is an almost faithful family $\{D_u\}_{u\in U}$ of non-trivial plane curves which has rank $k>(r+1)\cdot\dim M$ and is $\mathcal M$-definable over $\emptyset$. Let $(u_0,...,u_r)$ be generic in $U^{r+1}$ over $A\hat x$. For each $i$, let $S_i$ be a stationary component of $D_{u_i}$, and let $c_i=\operatorname{Cb}(S_i)$. By Lemma \ref{almost faithful facts}, each $c_i$ is $\mathcal M$-interalgebraic with $u_i$ over $\emptyset$; this is enough to verify (3) and (4), since by definition these two clauses hold of $(u_0,...,u_r)$.
			
		Now by Fact \ref{cb facts}, for each $i$ there is some set $C_i$ which is almost equal to $S_i$ and $\mathcal M$-definable over $c_i$. Note that $C_i$ is still stationary with canonical base $c_i$, so we get (2). Finally, the non-triviality of $D_{u_i}$ implies the non-triviality of $S_i$, and thus also of $C_i$. So we get (1).
		\end{proof}
	
		Fix $C_0,...,C_r,c_0,...,c_r,k$ as in Lemma \ref{translators}. For shorthand, we define $c=(c_0,...,c_r)$, and $c_{<r}=(c_0,...,c_{r-1})$. We will use $C_0,...,C_r$ as our translation maps, as described in the introduction to this subsection. Thus our next step, using the $C_i$, is to define our `translation' of $X$: 
		
		\begin{notation}\label{D'} For the remainder of the proof of Proposition \ref{stationary non-trivial hypersurfaces}, we set $$Y=\{(y_0,...,y_r):\textrm{ for some }(x_0,...,x_r)\in X\textrm{ we have that each }(x_i,y_i)\in C_i\}.$$
			\end{notation}
		
		So $Y$ is $\mathcal M$-definable over $Ac$. We note that the tuples $(x_0,...,x_r)$ and $(y_0,...,y_r)$ appearing in Notation \ref{D'} are coordinate-wise $\mathcal M$-interalgebraic over $c$, by the non-triviality of the $C_i$. In particular, each element of $Y$ is coordinate-wise $\mathcal M$-interalgebraic with an element of $X$ over $Ac$. Using this, our first observation is:
		
		\begin{lemma}\label{D' non-trivial} $Y$ is a non-trivial $r$-hypersurface.
			\end{lemma} 
		\begin{proof}
			By the preceding remarks we have $\rk(Y)\leq\rk(X)=r$. On the other hand, let $x=(x_0,...,x_r)$ be generic in $X$ over $Ac$. Since $X$ is non-trivial, and $r\geq 1$, it follows that each $x_i$ is generic in $M$ over $Ac$. In particular, by non-triviality of the $C_i$, for each $i$ we can find some $y_i$ such that $(x_i,y_i)\in C_i$. Then the tuple $y=(y_0,...,y_r)$ belongs to $Y$, and is coordinate-wise $\mathcal M$-interalgebraic with $x$ over $Ac$. Thus $\rk(y/Ac)=\rk(x/Ac)=r$, which shows that $\rk(Y)=r$.
			
			To see that $Y$ is non-trivial, let $y=(y_0,...,y_r)$ be generic in $Y$ over $Ac$. It will suffice to show that any $r$ coordinates of $y$ have rank $r$ over $Ac$. Now since $y\in Y$ there is some $x=(x_0,...,x_r)$ which is coordinate-wise $\mathcal M$-interalgebraic with $y$ over $Ac$. Since $\rk(X)=\rk(Y)=r$, $x$ is generic in $X$ over $Ac$. Since $X$ is non-trivial, we conclude that any $r$ coordinates of $x$ have rank $r$ over $Ac$; then, again by coordinate-wise $\mathcal M$-interalgebraicity, the same holds of the coordinates of $y$. 
			\end{proof}
		
		Our next goal is to find an appropriate translation of our closure point $\hat x$:
		
		\begin{lemma} There is some $\hat y=(\hat y_0,...,\hat y_r)\in M^{r+1}$ satisfying the following:
			\begin{enumerate}
				\item $\hat x$ and $\hat y$ are coordinate-wise $\mathcal M$-interalgebraic over $Ac$.
				\item $\hat y$ is generic in $M^{r+1}$ over $A\hat x$, and thus also over $A$.
				\item $\hat y\in\overline{Y}$.
			\end{enumerate}
		\end{lemma}
		\begin{proof}
			By definition, $\hat x$ and $c$ are independent over $\emptyset$. Thus $\hat x$ is coordinate-wise generic over $c$. In particular, for each $i=0,...,r$ we have that $\hat x_i$ is generic in $M$ over $c_i$. Since $C_i$ is non-trivial, we can thus find an element $\hat y_i\in M$ such that $(\hat x_i,\hat y_i)$ is a generic element of $C_i$ over $c_i$. We thus obtain a tuple $\hat y=(\hat y_0,...,\hat y_r)$ with each $(\hat x_i,\hat y_i)$ generic in $C_i$. Note that this is enough for (1) in the lemma. 
			
			Next we show (2). For each $i$, since $\rk(c_i)=k>0$, the generic $\mathcal M$-type of $C_i$ cannot have rank 1 over $\emptyset$ -- indeed otherwise $C_i$ would be a stationary component of a rank 1 set over $\emptyset$, so its canonical base would be of rank 0. In particular, since $(\hat x_i,\hat y_i)$ is generic in $C_i$ over $c_i$ for each $i$, we obtain that $\rk(\hat x_i\hat y_i)=2$. 
			
			On the other hand we have $\rk(c_i\hat x_i\hat y_i)=k+1$, counting from left to right, so we conclude that $\rk(c_i/\hat x_i\hat y_i)=k-1$. Adding across all $i$, we get $$\rk(c/A\hat x\hat y)\leq(r+1)(k-1).$$
			
			Now since $c$ is independent from $A\hat x$ over $\emptyset$, we obtain $$\rk(c\hat y/A\hat x)\geq\rk(c/A\hat x)=(r+1)k.$$ Combining with our previous assertion, we conclude that $$\rk(\hat y/A\hat x)\geq r+1.$$ Moreover, note that $c$ is coherent over $A\hat x$ by definition, and $\hat y$ is $\mathcal M$-algebraic over $Ac\hat x$ by (1). So by Lemma \ref{coherent preservation}, $\hat y$ is coherent over $A\hat x$. Thus $\hat y$ is coherent of rank $\geq r+1$ over $A\hat x$, which shows that $\hat y$ is indeed generic in $M^{r+1}$ over $A\hat x$. Of course this implies that $\hat y$ is also generic in $M^{r+1}$ over $A$, as desired.
			
			Finally we show (3). Let $V=V_0\times...\times V_r$ be any analytic neighborhood of $\hat y$ in $M^{r+1}$. We will find an element $y\in V\cap Y$. Now since each $(\hat x_i,\hat y_i)$ is generic in $C_i$ over $c_i$, Corollary \ref{generic openness} implies that the projection $\pi:C_i\rightarrow M$, $(x,y)\mapsto x$, is analytically open near $(\hat x_i,\hat y_i)$. After applying various shrinkings, it follows that we can find an analytic neighorhood $U=U_0\times...\times U_r$ of $\hat x$ in $M^{r+1}$, so that for each $i$ and each $u\in U_i$ there is some $v\in V_i$ with $(u,v)\in C_i$. Now since $\hat x\in\overline X$, there is some $x=(x_0,...,x_r)\in U\cap X$. We thus obtain $y=(y_0,...,y_r)\in V$ with each $(x_i,y_i)\in C_i$, which implies that $y\in Y$, as desired.
			
		\end{proof}
		We are almost ready to apply Proposition \ref{generic non-trivial hypersurfaces}. What remains is to construct an almost faithful family and, roughly, to realize $Y$ as a generic member of that family. What we will actually do is use a stationary set which is almost contained in $Y$. 
		
		First note that, since $Y$ is $\mathcal M$-definable over $Ac$, we can write $Y$ as a union of stationary components each of which is $\mathcal M$-definable over $\operatorname{acl}_{\mathcal M}(Ac)$. Since $\hat y\in\overline{Y}$, it belongs to the closure of one of these components. The following is now justified:
		
		\begin{notation} Let $Z$ be a stationary subset of $Y$ of rank $r$ which is $\mathcal M$-definable over $\operatorname{acl}_{\mathcal M}(Ac)$ and satisfies $\hat y\in\overline Z$. 
			\end{notation}
		
		Since $Z$ has rank $r$, it is also an $r$-hypersurface. Moreover, since it is contained in $Y$, it is non-trivial. What remains, essentially, is to show that $Z$ is `generic enough' to play the role of $H_{\hat t}$ from Proposition \ref{generic non-trivial hypersurfaces}. 
		
		We let $z=\operatorname{Cb}(Z)$. Our goal, precisely, is to show that $z$ is coherent of rank $\geq k$ over $A$. The first of these statements is already clear:
		
		\begin{lemma} $z\in\operatorname{acl}_{\mathcal M}(Ac)$. In particular, $z$ is coherent over $A$.
			\end{lemma}
		\begin{proof} The first clause follows because $Z$ is a stationary component of $Y$, and $Y$ is $\mathcal M$-definable over $Ac$. The second clause then follows from Lemma \ref{coherent preservation}.
		\end{proof}
	
		We next show that $\rk(z/A)\geq k$. The reason, essentially, is that one can recover $C_r$ from $C_0,...,C_{r-1}$ and $Z$. Precisely, we prove the following:
		
		\begin{lemma}\label{z_<r alg} $c_r\in\operatorname{acl}_{\mathcal M}(Azc_{<r})$. In particular, $\rk(z/A)\geq k$.
		\end{lemma}
		\begin{proof}
			Since $z=\operatorname{Cb}(Z)$, there is a set $Z'$ which is $\mathcal M$-definable over $z$ and almost equal to $Z$. Note that $Z'$ is therefore also a non-trivial $r$-hypersurface -- so in particular, each projection $Z'\rightarrow M^r$ is almost finite-to-one. By deleting all infinite fibers in such projections, we may assume that each projection $Z'\rightarrow M^r$ is everywhere finite-to-one.
			
			We will work with the following set:
			
			\begin{definition} Let $W$ be the set of all $(x_0,...,x_r,y_0,...,y_r)$ such that:				
				\begin{enumerate}
				\item $(x_0,...,x_r)\in X$.
				\item $(y_0,...,y_r)\in Z'$.
				\item For each $i=0,...,r-1$, $(x_i,y_i)\in C_i$.
			\end{enumerate}
		\end{definition}
		So $W$ is $\mathcal M$-definable over $Azc_{<r}$. We note:
			
			\begin{claim}\label{rk Y at most r} $\rk(W)\leq r$.
			\end{claim}
			\begin{proof} There is a natural projection $W\rightarrow X$ to the leftmost $r+1$ coordinates. We claim that this projection is finite-to-one, which gives that $\rk(W)\leq\rk(X)=r$, as desired.
				
				To see this, let $(x_0,...,x_r)\in X$. Since each $C_i$ is non-trivial, there are only finitely many extensions of $(x_0,...,x_r)$ by a tuple $(y_0,...,y_{r-1})$ with each $(x_i,y_i)\in C_i$. Then by our assumption on $Z'$, there are only finitely many extensions of each such tuple $(y_0,...,y_{r-1})$ to an element $y_r$ with $(y_0,...,y_r)\in Z'$. Thus there are only finitely many extensions of $(x_0,...,x_r)$ to an element of $W$, as desired.
				\end{proof}
			We also define the following set:
			\begin{definition} Let $V$ be the set of all $(x_r,y_r)$ such that the fiber $W_{(x_r,y_r)}$ has rank $\geq r-1$; that is, such that the set of $(x_0,...,x_{r-1},y_0,...,y_{r-1})$ with $(x_0,...,x_r,y_0,...,y_r)\in W$ has rank $\geq r-1$.
			\end{definition}
			So $V$ is also $\mathcal M$-definable over $Azc_{<r}$. By Claim \ref{rk Y at most r}, it is immediate that $\rk(V)\leq 1$. On the other hand, we conclude:
			\begin{claim}\label{Z_r generic type of U} $C_r$ is almost contained in $V$.
				\end{claim}
			\begin{proof} Fix $y=(y_0,...,y_r)$ generic in $Z'$ over $Azc$. So $y$ is also generic in $Y$ over $Azc$, and thus there is some $x=(x_0,...,x_r)\in X$ such that $(x_i,y_i)\in C_i$ for each $i$. Thus $(x,y)\in W$. It is clear that $\rk(xy/Azc)=r$, by definition of $y$ and the $\mathcal M$-interalgebraicity of $x$ and $y$ over $Ac$. 
				
			On the other hand, the non-triviality of $Z'$ implies that $\rk(y_r/Azc)=1$. Since $(x_r,y_r)\in C_r$, it follows that $\rk(x_ry_r/Azc)=1$, so that $(x_r,y_r)$ is in fact $\mathcal M$-generic in $C_r$ over $Azc$. Using additivity and the previous paragraph, it now follows that $\rk(xy/Azcx_ry_r)=r-1$. Since $(x,y)\in W$, this implies that $(x_r,y_r)\in V$. Then since $(x_r,y_r)$ is generic in $C_r$ over $Azc$, the same holds of any generic $(x_r',y_r')\in C_r$ over $Azc$. The claim now follows.
				\end{proof}
			By Claim \ref{Z_r generic type of U} and the fact that $\rk(V)\leq 1$, it follows that $C_r$ is (up to finitely many points) a stationary component of $V$. Since $V$ is $\mathcal M$-definable over $Azc_{<r}$, we conclude that $c_r$ is indeed $\mathcal M$-algebraic over $Azc_{<r}$.
			
			Now to finish the proof of Lemma \ref{z_<r alg}, we note the following two facts:
			\begin{enumerate}
				\item $\rk(c_r/Ac_{<r})=k$. Indeed, this follows since the $c_i$ are independent over $A$.
				\item $\rk(c_r/Azc_{<r})=0$. Indeed, this is what we have just proved.
			\end{enumerate}
			By the combination of (1) and (2), it must be that $\rk(z/Ac_{<r})\geq k$, and thus $\rk(z/A)\geq k$, as desired.
		\end{proof}
		
		We can now extract a family of hypersurfaces. Namely, let $\mathcal H=\{H_t:t\in T\}$ be an almost faithful family of non-trivial $r$-hypersurfaces, and $t\in T$ a fixed generic element, which satisfy Lemma \ref{coherent version of almost faithful family existence}(a) for $Z$ with parameter set $A$; so $Z$ and $H_t$ are almost equal, and $z$ and $t$ are $\mathcal M$-interalgebraic over $A$. Note also that $\rk(\mathcal H)=k$, so that Proposition \ref{generic non-trivial hypersurfaces} will apply. We now show:
		
		\begin{lemma}\label{* over extra parameters} $\hat y$ satisfies $*$ over $Azct$.
			\end{lemma}
		\begin{proof} We proceed in two cases:
			\begin{itemize}
				\item First suppose $\hat y\in\overline{H_t}$. Then all hypotheses of Proposition \ref{generic non-trivial hypersurfaces} are met, and we conclude that $\hat y$ satisfies $*$ over $At$. The lemma then follows from (4) of Lemma \ref{* preservation}.
				\item Now suppose $\hat y\notin\overline{H_t}$. Since $\hat y\in\overline Z$, it follows that $\hat y\in\overline{Z-H_t}$. But $Z$ and $H_t$ are almost equal $r$-hypersurfaces, so $\rk(Z-H_t)<r$. Moreover, note that $Z-H_t$ is $\mathcal M$-definable over $\operatorname{acl}_{\mathcal M}(Azct)$. By the inductive hypothesis, we conclude that $\rk(\hat y/\operatorname{acl}_{\mathcal M}(Azct))<r$, so that $\hat y$ trivially satisfies $*$ over $\operatorname{acl}_{\mathcal M}(Azct)$. The lemma then follows by (3) of Lemma \ref{* preservation}.
			\end{itemize}
		\end{proof}
	
	We now finish the proof of Proposition \ref{stationary non-trivial hypersurfaces}. Namely, recall that $z\in\operatorname{acl}_{\mathcal M}(Ac)$, and $z$ and $t$ are $\mathcal M$-interalgebraic over $A$. Thus $(z,t)\in\operatorname{acl}_{\mathcal M}(Ac)$. So by Lemma \ref{* over extra parameters} and (3) of Lemma \ref{* preservation}, $\hat y$ satisfies $*$ over $Ac$. Next, recall that $\hat x$ and $\hat y$ are coordinate-wise $\mathcal M$-interalgebraic over $Ac$. So by (2) of Lemma \ref{* preservation}, $\hat x$ also satisfies $*$ over $Ac$. Finally, recall by definition that $\hat x$ and $c$ are $\mathcal M$-independent over $A$. So by (1) of Lemma \ref{* preservation}, $\hat x$ satisfies $*$ over $A$. 
			
	\end{proof}
	\subsection{The Remaining Cases}
	
	In this subsection we finish the proof of Proposition \ref{closure prop} by extending the inductive step to arbitrary sets of rank $r$. So, we are still assuming for now that Proposition \ref{closure prop} holds for all sets of rank less than $r$, where $r\geq 1$ is fixed. There are three main things left to prove, but they are all significantly easier than the cases treated before now. So, rather than explaining the strategies, let us simply present the arguments. We begin with the case of arbitrary stationary hypersurfaces:
	
	\begin{proposition}\label{stationary hypersurfaces} Proposition \ref{closure prop} holds whenever $X$ is a stationary $r$-hypersurface.
		\end{proposition}
	\begin{proof} Let $X$, $A$, and $x=(x_0,...,x_r)$ be as in the statement of Proposition \ref{closure prop}, where $X$ is a stationary $r$-hypersurface. First suppose $X$ is non-trivial. Then by Proposition \ref{stationary non-trivial hypersurfaces}, $*(x,A)$ holds; in particular, $\rk(x/A)\leq r$. So it remains to show that, if $\rk(x/A)=r$ and $\pi_i$ and $\pi_j$ are distinct projections to $M$ which are independent on $X$, then $x_i$ and $x_j$ are $\mathcal M$-independent over $A$. But if $r\geq 2$, then this is automatic by $*(x,A)$, since then $x_i$ and $x_j$ are $\mathcal M$-independent $\mathcal M$-generics in $M$ over $A$. So we may assume that $r=1$. In this case there are only two projections to $M$, and by the non-triviality of $X$ they are not independent on $X$; so the desired statement holds vacuously. 
		
		Now assume $X$ is trivial, i.e. that there is a projection $\pi:M^{r+1}\rightarrow M^r$ which is not almost finite-to-one on $X$. Fix such a projection $\pi$. Without loss of generality $\pi$ projects to the first $r$ coordinates (with indices $0,...,r-1$). Let $X'$ be the union of the infinite fibers of $\pi$ in $X$; so $X'$ is also $\mathcal M$-definable over $A$. Since $\pi$ is not almost finite-to-one on $X$, $X'$ is generic in $X$, and thus also large in $X$ by stationarity. It is then easy to see that $\pi(X')\subset M^r$ is a stationary $r-1$-hypersurface, which is also $\mathcal M$-definable over $A$.
		
		We first show:
		\begin{lemma} We may assume that $x\in\overline{X'}$.
		\end{lemma}
		\begin{proof} Suppose not. Since $x\in\overline X$ by assumption, it follows that $x\in\overline{X-X'}$. But since $X'$ is large in $X$ we have $\rk(X-X')<r$, so by the inductive hypothesis we have $$\rk(x/A)\leq\rk(X-X')<r,$$ thus the desired statement holds automatically.
		\end{proof}
		So, assume that $x\in\overline{X'}$. Thus $\pi(x)\in\overline{\pi(X')}$. Note that $\pi(x)$ is also coordinate-wise generic, because $x$ is. Thus we may apply the inductive hypothesis to the stationary $r-1$-hypersurface $\pi(X')$ and the closure point $\pi(x)$. We conclude first that $\rk(\pi(x)/A)\leq r-1$; but since $\pi(x)$ contains all coordinates of $x$ but one, this easily implies that $\rk(x/A)\leq r$.
		
		Now assume that $\rk(x/A)=r$. By the same reasoning as above, this forces that $\rk(\pi(x)/A)=r-1$ and $\rk(x_r/A\pi(x))=1$. In particular $x_r$ is $\mathcal M$-independent from $(x_0,...,x_{r-1})$ over $A$.
		
		To complete the proof, let $\pi_i$ and $\pi_j$ be distinct projections to $M$ which are independent on $X$. By the previous paragraph, we may assume that $i$ and $j$ are among the leftmost $r$ projections, i.e. $i,j\leq r-1$. So we can factor $\pi_i$ and $\pi_j$ through our fixed projection $\pi:M^{r+1}\rightarrow M^r$. Now it is easy to see (since $\pi$ is not almost finite-to-one on $X$) that any generic element of $\pi(X')$ is the image of a generic element of $X$ under $\pi$. Using the definition of independent projections, it then follows easily that $\pi_i$ and $\pi_j$ are also independent on $X'$. So, by the inductive hypothesis again, it follows that $\pi_i(\pi(x))=x_i$ and $\pi_j(\pi(x))=x_j$ are $\mathcal M$-independent over $A$, as desired.
		\end{proof}
	
		Next, we extend to stationary sets which are not hypersurfaces. Recall from the statement of Proposition \ref{closure prop} that $n$ denotes the ambient dimension, i.e. $D\subset M^n$. We will now revert to notating points $x\in M^n$ with indices $1,...,n$ (rather than $0,...,r$ for $r$-hypersurfaces). 
		
		\begin{proposition}\label{m>r} Proposition \ref{closure prop} holds whenever $X$ is stationary of rank $r$ and $n\geq r+1$.
		\end{proposition}
		\begin{proof} Let $X$, $n$, $A$, and $x=(x_1,...,x_n)$ be as in Proposition \ref{closure prop}. We first note:
			\begin{lemma}\label{rk x in projection} If $\pi:M^n\rightarrow M^{r+1}$ is any projection, then $\rk(\pi(x)/A)\leq r$.
			\end{lemma}
			\begin{proof} Note that $\pi(X)$ is also $\mathcal M$-definable over $A$, and has rank at most $r$. Moreover, note that $\pi(x)$ is also coordinate-wise generic, and since $x\in X$ it follows that $\pi(x)\in\overline{\pi(X)}$. We now proceed in two cases:
				\begin{itemize}
					\item First suppose $\rk(\pi(X))=r$. Using that $X$ is stationary, it follows easily that $\pi(X)$ is a stationary $r$-hypersuface. So we may apply Proposition \ref{stationary hypersurfaces}, and in particular we get that $\rk(\pi(x)/A)\leq r$.
					\item Now suppose $\rk(\pi(X))<r$. Then we may apply the inductive hypothesis, and similarly conclude that $\rk(\pi(x)/A)\leq\rk(\pi(X))<r$. 
				\end{itemize} 
			\end{proof}
			We thus conclude the first part of Proposition \ref{closure prop}:
			\begin{lemma} $\rk(x/A)\leq r$.
			\end{lemma}
			\begin{proof}
				If $\rk(x/A)\geq r+1$ then there would be a set of $r+1$ coordinates of $x$ of rank $r+1$ over $A$, contradicting Lemma \ref{rk x in projection}.
			\end{proof}
			Now to complete the proof of Proposition \ref{m>r}, assume that $\rk(x/A)=r$, and let $\pi_i$ and $\pi_j$ be distinct projections to $M$ which are independent on $X$. Without loss of generality assume that $i=1$ and $j=2$. We thus wish to show that $x_1$ and $x_2$ are $\mathcal M$-independent over $A$.
			\begin{lemma} We may assume that $\rk(x_1x_2/A)=1$.
			\end{lemma}
			\begin{proof} Otherwise $x_1$ and $x_2$ would automatically be $\mathcal M$-independent over $A$ -- either because they are $\mathcal M$-independent $\mathcal M$-generics in $M$ over $A$, or because they are both $\mathcal M$-algebraic over $A$.
			\end{proof}
			So, assume that $\rk(x_1x_2/A)=1$. Since $\rk(x/A)=r$, we can thus extend $(x_1,x_2)$ to a subtuple of $x$ which has length $r+1$ and rank $r$ over $A$. Without loss of generality this subtuple is $(x_1,...,x_{r+1})$. Let $\pi:M^n\rightarrow M^{r+1}$ project to the first $r+1$ coordinates; so we have $\rk(\pi(x)/A)=r$. Then we conclude:
			\begin{lemma}\label{rk D doesn't drop under projection} $\rk(\pi(X))=r$.
			\end{lemma}
			\begin{proof}
				Since $\rk(X)=r$ we automatically have $\rk(\pi(X))\leq r$. Now suppose that $\rk(\pi(X))<r$. Note that $\pi(x)\in\overline{\pi(X)}$, since $x\in\overline X$. Moreover $\pi(X)$ is $\mathcal M$-definable over $A$ because $X$ is, and $\pi(x)$ is coordinate-wise generic because $x$ is. So the inductive hypothesis applies. We conclude that $\rk(\pi(x)/A)\leq\rk(\pi(X))<r$, contradicting the assumption that $\rk(\pi(x)/A)=r$. 
			\end{proof}
			Now using Lemma \ref{rk D doesn't drop under projection} and the fact that $X$ is stationary, it follows easily that $\pi$ is almost finite-to-one on $X$, and moreover that $\pi(X)$ is a stationary $r$-hypersurface. As in Lemma \ref{rk D doesn't drop under projection}, $\pi(X)$ is $\mathcal M$-definable over $A$, and $\pi(x)$ is coordinate-wise generic and belongs to $\overline{\pi(X)}$. Thus we may apply Proposition \ref{stationary hypersurfaces}. Since $\rk(\pi(x)/A)=r$, we conclude that any two independent projections on $\pi(X)$, when evaluated at $\pi(x)$, output $\mathcal M$-independent elements over $A$. 
			
			On the other hand, using that $\rk(\pi(X))=r$ it is easy to see that any generic element of $\pi(X)$ is the image of a generic element of $X$. Then using the definition of independent projections, it follows that any two projections among $\pi_1,...,\pi_{r+1}$ which are independent on $X$, are also independent on $\pi(X)$. In particular, $\pi_1$ and $\pi_2$ are independent on $\pi(X)$. Thus we conclude that $x_1$ and $x_2$ are $\mathcal M$-independent over $A$, as desired.
			\end{proof}
			Our last main step is now to drop the assumption of stationarity:
			\begin{proposition}\label{m>r general} Proposition \ref{closure prop} holds whenever $X$ has rank $r$ and $n\geq r+1$.
			\end{proposition}
			\begin{proof}
				Let $X$, $n$, $A$, and $x=(x_1,...,x_n)$ be as in Proposition \ref{closure prop}. Write $X$ as a finite union $X_1\cup...\cup X_l$ of stationary sets of rank $r$; without loss of generality we may assume that each $X_k$ is $\mathcal M$-definable over $A':=\operatorname{acl}_{\mathcal M}(A)$. Since $x\in\overline X$, we get that $x\in\overline{X_k}$ for some $k$. 
				
				We will apply Proposition \ref{m>r} to $X_k$ and $x$, with parameter set $A'$. Since $A'$ is just the algebraic closure of $A$ (in $\mathcal M^{\textrm{eq}}$ that is), $x$ and $A'$ are $\mathcal M$-independent over $A$. This will allow us to transfer results about $x$ and $A'$ to the analogous results about $x$ and $A$.
				
				Now let us apply Proposition \ref{m>r}. We conclude first that $\rk(x/A')\leq r$; by $\mathcal M$-independence, this immediately gives that $\rk(x/A)=\rk(x/A')\leq r$ as well.
				
				Now assume that $\rk(x/A)=r$. By $\mathcal M$-independence this implies that $\rk(x/A')=r$. So by Proposition \ref{m>r}, any two distinct independent projections $X_k\rightarrow M$, evaluated on $x$, return elements which are $\mathcal M$-independent over $A'$. By $\mathcal M$-independence again, the same statement holds with $A'$ replaced by $A$.
				
				Now let $\pi_i$ and $\pi_j$ be distinct projections which are independent on $X$. By the definition of independent projections, it is clear that $\pi_i$ and $\pi_j$ are in particular independent on $X_k$ (since any generic of $X_k$ is also a generic of $X$). Thus, by the previous paragraph, $x_i$ and $x_j$ are $\mathcal M$-independent over $A$, as desired.
			\end{proof}
				At long last, we are ready to completely prove Proposition \ref{closure prop}. We now drop the inductive assumption (i.e. Assumption \ref{inductive assumption}). And, finally:
				
			\begin{proof}[Proof of Proposition \ref{closure prop}] Let $X$, $n$, $A$, and $x=(x_1,...,x_n)$ be as in Proposition \ref{closure prop}. We begin with two easy observations:
				\begin{lemma}\label{x in D case} Proposition \ref{closure prop} holds if \ref{closure prop} if $x\in X$.
				\end{lemma}
				\begin{proof} If $x\in X$ then of course $\rk(x/A)\leq r$, as witnessed by $X$. Now assume that $\rk(x/A)=r$; so $x$ is in fact $\mathcal M$-generic in $X$ over $A$. Let $\pi_i$ and $\pi_j$ be distinct projections to $M$ which are independent on $X$. Then by the definition of independent projections, $\pi_i(x')$ and $\pi_j(x')$ are $\mathcal M$-independent over $A$ whenever $x'$ is $\mathcal M$-generic in $X$ over $A$. In particular, setting $x'=x$ gives that $x_i$ and $x_j$ are $\mathcal M$-independent over $A$, which proves the proposition in this case.
				\end{proof}
				\begin{lemma}\label{D finite case} Proposition \ref{closure prop} holds if $X$ is finite.
					\end{lemma}
				\begin{proof} If $X$ is finite and $x\in\overline X$ then $x\in X$, so we reduce to Lemma \ref{x in D case}.
					\end{proof}
				We now proceed in general by induction on $r$. If $r=0$ then we are done by Lemma \ref{D finite case}. So, we assume that $r\geq 1$ and Proposition \ref{closure prop} holds for all $r'<r$. That is, we are in the situation of Assumption \ref{inductive assumption}, so we can use all of the results from Section 8 up to this point.
				
				Now since $\rk(X)=r$ and $X\subset M^n$, we have $n\geq r$. We proceed in two cases:
				\begin{itemize}
					\item First suppose $n=r$. So $X$ is large in $M^n$. Now if $x\in X$, then we are done by Lemma \ref{x in D case}; so we may assume that $x\notin X$. But in this case the largeness of $X$ in $M^n$ implies that $\rk(x/A)<r$, from which Proposition \ref{closure prop} holds automatically.
					\item Now suppose $n\geq r+1$. Then we are precisely in the situation of Proposition \ref{m>r general}, and so we are done by that proposition. 
				\end{itemize}
				So, by induction, the proof of Proposition \ref{closure prop} is now complete.
				\end{proof}

	\section{Detecting Generic Non-transversalities}
	In the final section we use Proposition \ref{closure prop} to show that $\mathcal M$ detects generic non-transversalities. The idea is to mimic the proof of Lemma \ref{ideal scenario} while using the weaker bound on frontier points offered by Proposition \ref{closure prop}. The trick is that since the closure point in question has two pairs of equal coordinates, we are able to rule out the clause about independent projections, and show instead that the rank strictly drops. This will not quite work in one step, because we need to `preen' the set in consideration until its dimension-theoretic properties match the desired setup; however, Lemma \ref{ideal scenario} still serves as a good model for the argument.
	
	\begin{theorem}\label{detecting generic non-transversalities} $\mathcal M$ detects generic non-transversalities.
	\end{theorem}
	\begin{proof} As in Lemma \ref{ideal scenario}, we use (3) of Lemma \ref{detecting tangency equivalence}. So, let $\mathcal C=\{C_t:t\in T\}$ and $D=\{D_u:u\in U\}$ be standard families of plane curves, and let $\hat w=(\hat x,\hat t,\hat u)$ be a generic $(\mathcal C,\mathcal D)$-non-injectivity. By Lemma \ref{detecting tangency equivalence}, it suffices to show that $\hat t$ and $\hat u$ are $\mathcal M$-dependent over $\emptyset$. Note that $\hat w$ is coordinate-wise generic, by the setup of Lemma \ref{detecting tangency equivalence}(3); this will guarantee that any applications of Proposition \ref{closure prop} below are valid.
		
	Let $I\subset M^2\times T\times U$ be the graph of the family of intersections $\mathcal I_{\mathcal C,\mathcal D}$, so that by assumption $\hat w$ is non-injective in the projection $I\rightarrow T\times U$. Also let $Z\subset T\times U$ be the set of $(t,u)$ such that $C_t\cap D_u$ is infinite, and note by almost faithfulness that the projection $Z\rightarrow T$ is finite-to-one. Finally, let $r_{\mathcal C}=\rk(\mathcal F)$ and $r_{\mathcal D}=\rk(\mathcal G)$. So $\rk(\hat t)=r_{\mathcal C}$ and $\rk(\hat u)=r_{\mathcal D}$, and our goal is to show that $\rk(\hat t\hat u)<r_{\mathcal C}+r_{\mathcal D}$. 
		
		\begin{lemma} We may assume that $(\hat t,\hat u)\notin\overline Z$.
		\end{lemma}
		\begin{proof} Assume that $(\hat t,\hat u)\in\overline Z$. Then by Proposition \ref{closure prop} and the fact that $Z\rightarrow T$ is finite-to-one, we have $$\rk(\hat t\hat u)\leq\rk(Z)\leq\rk(T)=r_{\mathcal C}<r_{\mathcal C}+r_{\mathcal D},$$ which proves the theorem in this case.
		\end{proof}
		
		So, assume that $(\hat t,\hat u)\notin\overline Z$. Let $I'$ be the set of $(x,t,u)\in I$ such that $(t,u)\notin Z$. Then by assumption it follows that $\hat w$ is still non-injective in $I'\rightarrow T\times U$. Let $P$ be the set of $(x,x',t,u)$ such that $x\neq x'$ and $(x,t,u),(x',t,u)\in I'$ -- so by non-injectivitiy the tuple $\hat z=(\hat x,\hat x,\hat t,\hat u)$ belongs to the frontier $\textrm{Fr}(P)$. Moreover, note by definition of $P$ that the projection $P\rightarrow T\times U$ is finite-to-one, which shows that $\rk(P)\leq r_{\mathcal C}+r_{\mathcal D}$. 
		
		Now we would like to apply Proposition \ref{closure prop} to $\hat z\in\operatorname{Fr}(P)$, with parameter set $\emptyset$. Note in particular that $P$ is $\mathcal M$-definable over $\emptyset$ because $\mathcal C$ and $\mathcal D$ are, and $\hat z$ is coordinate-wise generic because $\hat w$ is -- so this would be a valid application. However, in order to conclude what we want, we still need one more layer of `preening'.
		
		Let $\pi_1,\pi_2:M^2\rightarrow M$ denote the two projections. Then we define the sets $$P_1=\{(x,x',t,u)\in P:\pi_1(x)\neq\pi_1(x')\},$$ $$P_2=\{(x,x',t,u)\in P:\pi_2(x)\neq\pi_2(x')\}.$$ It is evident from the definition of $P$ that $P=P_1\cup P_2$. Thus we get $\hat z\in\operatorname{Fr}(P_i)$ for some $i=1,2$. Without loss of generality, we assume that $\hat z\in\operatorname{Fr}(P_1)$. We then define:
		
		\begin{definition} Let $x_1\neq x_1'\in M$. Then $x_1$ and $x_1'$ are \textit{extendable} if the preimage of $(x_1,x_1')$ in $P_1$ under the projection to the first and third coordinates has rank at least $r_{\mathcal C}+r_{\mathcal D}-1$. 
		\end{definition}
		
		Let $E$ be the set of extendable pairs. It is evident from the definition, and the fact that $\rk(P)\leq r_{\mathcal C}+r_{\mathcal D}$, that $\rk(E)\leq 1$. Then, denoting $\hat x$ as the pair $(\hat x_1,\hat x_2)$, we conclude the following:
		
		\begin{lemma}\label{frontier of extendable points} $(\hat x_1,\hat x_1)\notin\overline E$.
		\end{lemma}
		\begin{proof} If so then $(\hat x_1,\hat x_1)\in\operatorname{Fr}(E)$, since the two coordinates of $(\hat x_1,\hat x_1)$ are equal. But since $\rk(E)\leq 1$, we have $\dim(E)\leq\dim M$; so if $(\hat x_1,\hat x_1)\in\operatorname{Fr}(E)$ then $\dim(\hat x_1,\hat x_1)<\dim M$, which contradicts that $\hat x$ is generic in $M^2$.
		\end{proof}
		
		Finally, let $P'$ be the set of all $(x,x',t,u)\in P_1$ such that $\pi_1(x)$ and $\pi_1(x')$ are not extendable. Note that $P'$ is $\mathcal M$-definable over $\emptyset$, because $P$ is. Moreover, by Lemma \ref{frontier of extendable points}, it follows that $\hat z\in\operatorname{Fr}(P')$. Then we note:
		
		\begin{lemma}\label{P projections independent} If $\rk(P')=r_{\mathcal C}+r_{\mathcal D}$, then the projections to the first and third $M$-coordinates are independent on $P'$.
		\end{lemma}
		\begin{proof} Assume that $\rk(P')=r_{\mathcal C}+r_{\mathcal D}$, and let $(x,x',t,u)\in P'$ be generic. Let $x_1$ and $x_1'$ be the first and third coordinates of $(x,x',t,u)$ -- that is, the first coordinate of each of $x$ and $x'$. Then by definition of $P'$, $x_1$ and $x_1'$ are not extendable. By definition of extendability, we conclude that $\rk(xx'tu/x_1x_1')\leq r_{\mathcal C}+r_{\mathcal D}-2$. But by the choice of $(x,x',t,u)$ we have $\rk(xx'tu)=r_{\mathcal C}+r_{\mathcal D}$. So by additivity, we obtain $\rk(x_1x_1')\geq 2$. Thus $x_1$ and $x_1'$ are $\mathcal M$-independent $\mathcal M$-generics in $M$ over $\emptyset$, which is enough to prove the lemma.	
		\end{proof}
		
		We now apply Proposition \ref{closure prop} to $P'$. Since $P'\subset P$ and $\hat z\in\operatorname{Fr}(P')$, we first get $$\rk(\hat z)\leq\rk(P')\leq\rk(P)\leq r_{\mathcal C}+r_{\mathcal D}.$$ We want this to be a strict inequality between $\hat z$ and $r_{\mathcal C}+r_{\mathcal D}$. But if equality holds then all terms above must be equal, so in particular $\rk(\hat z)=\rk(P')=r_{\mathcal C}+r_{\mathcal D}$. So by Lemma \ref{P projections independent}, the projections of $P'$ to the first and third coordinates are independent. By Proposition \ref{closure prop}, this forces the first and third coordinates of $\hat z$ -- i.e. $\hat x_1$ and $\hat x_1$ -- to be $\mathcal M$-independent over $\emptyset$. But this is clearly not true, since $\hat x_1$ and $\hat x_1$ are equal generics.
	
	So the inequality must be strict, i.e. $\rk(\hat z)<r_{\mathcal C}+r_{\mathcal D}$. But $(\hat t,\hat u)$ is a subtuple of $\hat z$ -- so we also get that $\rk(\hat t\hat u)<r_{\mathcal C}+r_{\mathcal D}$, which proves the theorem.
		\end{proof}
	
	At last, we conclude:
	
	\begin{theorem}\label{main theorem} $\mathcal M$ interprets an algebraically closed field. 
	\end{theorem}
	\begin{proof} By Theorems \ref{n=1}, \ref{field interpretation}, and \ref{detecting generic non-transversalities}.
	\end{proof}

	Armed with Theorem \ref{main theorem}, our last remaining task is to drop Assumption \ref{M and K} and deduce the main theorem in its original form (Theorem \ref{general main theorem} below). So, from now on we do not retain the fixed structures $\mathcal M$ and $\mathcal K$ used until this point. Note that the remaining steps are fairly well known, but we will still give the details for completeness.
	
	Before stating the theorem, we will need the following lemma:
	
	\begin{lemma}\label{enough to interpret infinite field} Let $K$ be an algebraically closed field, let $\mathcal M$ be a structure interpreted in $K$, and let $\mathcal N$ be a structure elementarily equivalent to $\mathcal M$. If an infinite field is interpretable in $\mathcal N$, then $K$ is interpretable in $\mathcal M$. 
	\end{lemma}
	\begin{proof} Let $\mathcal L$ be the common language of $\mathcal M$ and $\mathcal N$, and let $T$ be their (complete by assumption) common $\mathcal L$-theory. Assume that $\mathcal N$ interprets an infinite field. Thus, working in $\mathcal N^{\textrm{eq}}$, there are a sort $S$, and formulas $\phi(x,a)$, $\psi(x,y,z,a)$, and $\theta(x,y,z,a)$ for $x,y,z\in S$ (where $a$ is a finite tuple of parameters from $\mathcal N$), such that the solution set of $\phi$ in $\mathcal N$ forms an infinite field when equipped with the formulas $\psi$ and $\theta$ (for the graphs of $+$ and $\cdot$, respectively). Let $w$ be a variable in the arity of $a$. 
		\begin{claim} Modulo the theory $T$, the assertion `$(\phi(x,w),\psi(x,y,z,w),\theta(x,y,z,w))$ defines an infinite field' is expressible by a single $\mathcal L$-formula in the arity of $w$.
		\end{claim}
		\begin{proof} The field axioms are clearly expressible by a single formula, so we need only show that `there are infinitely many $x$ satisfying $\phi(x,w)$' is expressible by a single formula. But this follows by the strong minimality of $K$. Indeed, note that for any $K$-interpretable family $\mathcal S$ of sets, there is a bound on the sizes of all finite members of $\mathcal S$. Since $\mathcal M$ is interpretable in $K$, it follows that the assertion $|\phi(x,w)|=\infty$ is equivalent in $\mathcal M$ (for all $w$ from $\mathcal M$) to the assertion $|\phi(x,w)|\geq N$ for some integer $N$ -- which can obviously be expressed by a single formula. Finally, since $\mathcal M$ is a model of $T$, and $T$ is complete, it follows that the same equivalence holds in all models of $T$. 
		\end{proof}
		
		Now let $\delta(w)$ be an $\mathcal L$-formula expressing the assertion in the claim. Then the sentence $\exists w\delta(w)$ is true in $\mathcal N$, as witnessed by $a$. Since $\mathcal M$ and $\mathcal N$ are elementarily equivalent, $\exists w\delta(w)$ is also true in $\mathcal M$. In particular it follows that $\mathcal M$ also interprets an infinite field. But it is well known that every infinite field interpreted in $K$ is $K$-definably isomorphic to $K$ (see Theorem 4.15 of \cite{Poi}) -- so in fact the field interpreted in $\mathcal M$ is isomorphic to $K$, as desired.
	\end{proof}

	Finally, we now end the paper by giving the main theorem:
	
	\begin{theorem}[Restricted Trichotomy in Characteristic Zero]\label{general main theorem} Let $K$ be an algebraically closed field of characteristic zero, and let $\mathcal M$ be a structure interpretable in $K$. If $\mathcal M$ is not 1-based, then $K$ is interpretable in $\mathcal M$.
	\end{theorem}
	\begin{proof} By elimination of imaginaries in $K$ (\cite{PoiEli}), we may assume that each $\mathcal M$-definable set is constructible over $K$. Now we first show:
	 
		\begin{lemma}\label{assume strongly minimal} We may assume that $\mathcal M$ is strongly minimal.
		\end{lemma}
		\begin{proof}
			Let $K'$ be an algebraically closed extension of $K$ of strictly greater cardinality than each of $K$ and the language of $\mathcal M$ (so in particular $K'$ is uncountable). Let $\mathcal M'$ be the structure interpreted in $K'$ via the same formulas that interpret $\mathcal M$ in $K$. 
			
			Note that $K'$ is a saturated elementary extension of $K$, by the strong minimality and model completenss of algebraically closed fields. Using this and the above assumptions on $|K'|$, it is easy to conclude that $\mathcal M'$ is a saturated elementary extension of $\mathcal M$. In particular, $\mathcal M'$ is also not 1-based.
			
			Now assume that Theorem \ref{general main theorem} holds in the strongly minimal case. Then using Fact \ref{1-based fact} (and the fact that $\mathcal M'$ is saturated), it follows that $\mathcal M'$ interprets a non-locally modular strongly minimal structure, say $\mathcal N$. We thus conclude by assumption that $\mathcal N$ interprets an infinite field; but then so does $\mathcal M'$, since it interprets $\mathcal N$. By Lemma \ref{enough to interpret infinite field} (applied to $K$, $\mathcal M$, and $\mathcal M$'), this is enough.
		\end{proof}
			 
			Now assume that $\mathcal M$ is strongly minimal, and let $M$ be the universe of $\mathcal M$. By non-local modularity (see Remark \ref{local modularity when not saturated}), there is a rank 2 almost faithful family of plane curves definable in $\mathcal M$, say $\mathcal C$. Let $a$ be a tuple from $K$ over which $M$ and $\mathcal C$ are interpretable -- so the reduct $\mathcal M'=(M,\mathcal C)$ is clearly interpretable in $K$ over $a$. Now since the field $\mathbb C$ is saturated, and the theory $\textrm{ACF}_0$ is complete, there is a tuple $a'$ from $\mathbb C$ realizing the same type (in the language of rings) as $a$. In particular, using the same formulas that define $M$ and $\mathcal C$, it follows that there is a structure $\mathcal N'=(N,\mathcal D)$ which is elemenarily equivalent to $\mathcal M'$ and interpretable in $\mathbb C$ over $a'$. So $\mathcal N'$ is also strongly minimal, and $\mathcal D$ is a rank 2 almost faithful family of plane curves in $\mathcal N'$ -- which shows that $\mathcal N'$ is also not locally modular. Then after adding appropriate constants to the languages of $\mathcal N'$ and $\mathbb C$, we arrive at structures $\mathcal N$ and $\mathcal K$ satisfying Assumption \ref{M and K} (with the `$\mathcal M$' from Assumption \ref{M and K} replaced by $\mathcal N$). So we are in the situation of the paper up to this point, and thus we can apply Theorem \ref{main theorem}. We conclude that $\mathcal N$ interprets an infinite field, say $L$. But since $\mathcal N$ only expands $\mathcal N'$ by constants, it follows that $\mathcal N'$ also interprets $L$, which by Lemma \ref{enough to interpret infinite field} implies that $\mathcal M'$ interprets $K$. Finally, since $\mathcal M'$ is a reduct of $\mathcal M$, it follows that $\mathcal M$ also interprets $K$, as desired.
		\end{proof}
	
	\printbibliography
	\end{document}